\documentclass[final,leqno,onefignum,onetabnum]{siamltex}
\usepackage[letterpaper,top=1.5in, bottom=1.5in, left=1in, right=1in]{geometry}

\usepackage{epsfig,epsf,fancybox}
\usepackage{amsmath}
\usepackage{mathrsfs}
\usepackage{amssymb}
\usepackage{pifont}
\usepackage{amsfonts}
\usepackage{graphicx}
\usepackage{color}
\usepackage[linktocpage,colorlinks,linkcolor=blue,anchorcolor=blue,citecolor=blue,urlcolor=blue,hypertexnames=false]{hyperref}
\usepackage{boxedminipage}
\usepackage{stmaryrd}
\usepackage{multirow}
\usepackage{booktabs}
\usepackage{cite}
\usepackage{float}
\usepackage[ruled,vlined,linesnumbered]{algorithm2e}
\usepackage{algpseudocode}

\usepackage{mathtools}

\usepackage[normalem]{ulem} 

\newcommand{\bm}[1]{\boldsymbol{#1}}

\newcommand{\va}{{\mathbf{a}}}

\newcommand{\vv}{{\mathbf{v}}}

\newcommand{\vx}{{\mathbf{x}}}
\newcommand{\vy}{{\mathbf{y}}}
\newcommand{\vz}{{\mathbf{z}}}

\newcommand{\vI}{{\mathbf{I}}}

\newcommand{\vL}{{\mathbf{L}}}

\newcommand{\vV}{{\mathbf{V}}}
\newcommand{\vW}{{\mathbf{W}}}
\newcommand{\vX}{{\mathbf{X}}}
\newcommand{\vY}{{\mathbf{Y}}}
\newcommand{\vZ}{{\mathbf{Z}}}

\newcommand{\vlam}{{\bm{\lambda}}}
\newcommand{\vLam}{{\bm{\Lambda}}}

\newcommand{\vzeta}{{\bm{\zeta}}}

\newcommand{\vxi}{{\bm{\xi}}}

\newcommand{\cB}{{\mathcal{B}}}

\newcommand{\cD}{{\mathcal{D}}}
\newcommand{\cE}{{\mathcal{E}}}

\newcommand{\cG}{{\mathcal{G}}}

\newcommand{\cJ}{{\mathcal{J}}}

\newcommand{\cN}{{\mathcal{N}}}

\newcommand{\cR}{{\mathcal{R}}}

\newcommand{\cX}{{\mathcal{X}}}
\newcommand{\cY}{{\mathcal{Y}}}
\newcommand{\cZ}{{\mathcal{Z}}}

\newcommand{\vareps}{\varepsilon}

\newcommand{\EE}{\mathbb{E}} 
\newcommand{\RR}{\mathbb{R}} 
\newcommand{\vzero}{\mathbf{0}} 
\newcommand{\vone}{{\mathbf{1}}} 

\newcommand{\dist}{\mathrm{dist}}    
\newcommand{\prox}{{\mathbf{prox}}} 
\newcommand{\dom}{{\mathrm{dom}}} 
\newcommand{\Null}{{\mathrm{Null}}} 

\newcommand{\avg}{{\mathrm{avg}}}
\newcommand{\Span}{{\mathrm{Span}}}

\newcommand{\st}{\mbox{ s.t. }}

\newcommand{\cmark}{\ding{51}}
\newcommand{\xmark}{\ding{55}}

\DeclareMathOperator*{\argmin}{arg\,min} 
\DeclareMathOperator*{\argmax}{arg\,max} 

\newcommand{\bc}{\begin{center}}
\newcommand{\ec}{\end{center}}

\newcommand{\bdm}{\begin{displaymath}}
\newcommand{\edm}{\end{displaymath}}

\newcommand{\beq}{\begin{equation}}
\newcommand{\eeq}{\end{equation}}

\newcommand{\bfl}{\begin{flushleft}}
\newcommand{\efl}{\end{flushleft}}

\newcommand{\bt}{\begin{tabbing}}
\newcommand{\et}{\end{tabbing}}

\newcommand{\beqn}{\begin{eqnarray}}
\newcommand{\eeqn}{\end{eqnarray}}

\newcommand{\beqs}{\begin{align*}} 
\newcommand{\eeqs}{\end{align*}}  

\newtheorem{remark}{Remark}[section]
\newtheorem{assumption}{Assumption}

\begin{document}

\title{Decentralized gradient descent maximization method for composite nonconvex strongly-concave minimax problems}

\author{Yangyang Xu\thanks{\url{xuy21@rpi.edu}, Department of Mathematical Sciences, Rensselaer Polytechnic Institute, Troy, NY 12180}}

\date{\today}

\maketitle

\begin{abstract} Minimax problems have recently attracted a lot of research interests. A few efforts have been made to solve decentralized nonconvex strongly-concave (NCSC) minimax-structured  optimization; however, all of them focus on smooth problems with at most a constraint on the maximization variable. In this paper, we make the first attempt on solving composite NCSC minimax problems that can have convex nonsmooth terms on both minimization and maximization variables. Our algorithm is designed based on a novel reformulation of the decentralized minimax problem that introduces a multiplier to absorb the dual consensus constraint. The removal of dual consensus constraint enables the most aggressive (i.e., local maximization instead of a gradient ascent step) dual update that leads to the benefit of taking a larger primal stepsize and better complexity results. In addition, the decoupling of the nonsmoothness and consensus on the dual variable eases the analysis of a decentralized algorithm; thus our reformulation creates a new way for interested researchers to design new (and possibly more efficient) decentralized methods on solving NCSC minimax problems. We show a global convergence result of the proposed algorithm and an iteration complexity result to produce a (near) stationary point of the reformulation. Moreover, a relation is established between the (near) stationarities of the reformulation and the original formulation. With this relation, we show that when the dual regularizer is smooth, our algorithm can have lower complexity results (with reduced dependence on a condition number) than existing ones to produce a near-stationary point of the original formulation.  Numerical experiments are conducted on a distributionally robust logistic regression to demonstrate the performance of the proposed algorithm.

\vspace{0.2cm}

\noindent {\bf Keywords:} composite minimax problem, nonconvex strongly-concave, decentralized algorithm, iteration complexity. 


\end{abstract}

\section{Introduction}
In this paper, we consider the minimax-structured problem
\begin{equation}\label{eq:min-max-prob}
\min_{\vx\in\RR^{n_1}} \max_{\vy\in \RR^{n_2}}  f(\vx, \vy) + g(\vx) - h(\vy), \text{ with }\textstyle f(\vx, \vy):=\frac{1}{m}\sum_{i=1}^m f_i(\vx, \vy),
\end{equation}
where $f_i$ is a nonconvex strongly-concave (NCSC) differentiable function with a Lipschitz continuous gradient for each $i=1,\ldots,m$, $g$ and $h$ are closed convex functions. Minimax-structured problems arise in many applications such as power control and transceiver design \cite{lu2020hybrid}, distributionally robust optimization \cite{wiesemann2014distributionally}, generative adversarial network \cite{goodfellow2020generative}, optimal transport \cite{huang2021convergence}, and reinforcement learning \cite{dai2018sbeed}. 

Recent years have witnessed a surge of research interest in designing algorithms for solving minimax-structured nonconvex problems, e.g., \cite{grimmer2022landscape, xu2023unified, zhang2021complexity, ostrovskii2021efficient, thekumparampil2019efficient, jin2020local, lin2020near, rafique2022weakly, liu2021first, zhang2020single, li2022nonsmooth, huang2022accelerated, lin2020gradient, zhang2022sapd+}. Different from most of the existing works, we are interested in designing a decentralized algorithm for solving \eqref{eq:min-max-prob}. We assume that there are $m$ agents on a connected graph or network $\cG=(\cN, \cE)$, where $\cN$ denotes the set of nodes (or agents) and $\cE$ the set of edges of $\cG$. For each $i\in \cN$, the function $f_i$ is owned \emph{privately} by the $i$-th agent. In order for the $m$ agents to collaboratively solve \eqref{eq:min-max-prob}, each agent will keep a copy 
of the primal-dual variable, 
and consensus on all local copies will be enforced. In order to achieve the consensus, the agents will send/receive local information to/from their one-hop neighbors and perform neighbor (weighted) averaging. 

\subsection{Existing decentralized methods for minimax problems} A majority of existing works on decentralized methods for solving minimax-structured optimization are mainly on (strongly-) convex (strongly-) concave problems or saddle-point or variational inequality (VI) reformulation of convex multi-agent optimization with coupling constraints, e.g., \cite{mateos2015distributed, mateos2016distributed, falsone2017dual, koshal2011multiuser, richards2004decentralized, paccagnan2016distributed, falsone2020tracking, mukherjee2020decentralized, kovalevoptimal2022, beznosikov2021distributed, metelev2022decentralized}. Several recent works study nonconvex concave (or even nonconcave) minimax problems and are more closely related to our work. Below, we review those works on decentralized minimax-structured nonconvex optimization.

The work \cite{tsaknakis2020decentralized} considers a special class of decentralized minimax problems, where consensus constraint is imposed on either the primal variable or the dual variable, but not both. Two decentralized gradient descent ascent (D-GDA) methods are given.  However, its convergence result is established only for the method that solves problems without dual consensus constraint.  When the involved functions are smooth and strongly concave about the dual variable that is restricted in a compact domain, the D-GDA in \cite{tsaknakis2020decentralized} can produce an $\vareps$-stationary solution within $O(\vareps^{-2})$ outer iterations and  communication rounds. Because there is no consensus on the dual variable and the dual part is maximized to a certain accuracy per outer iteration, the method can be viewed as a decentralized method for solving a nonconvex smooth minimization problem with inexact gradients. A decentralized policy GDA is given in \cite{lu2021decentralized} for solving multi-agent reinforcement learning (MARL). The algorithm is developed based on the Lagrangian function of the MARL, which does not have consensus on the dual variable (i.e., the Lagrangian multiplier) either. Because the Lagrangian function is merely concave about the dual variable, the complexity result established in \cite{lu2021decentralized} is $O(\vareps^{-4})$, which is higher than that in \cite{tsaknakis2020decentralized} but the best-known for nonconvex-concave problems. 

Motivated by the cooperative MARL, \cite{zhang2021taming} presents a gradient tracking based D-GDA for solving 
finite-sum structured NCSC minimax problems, where consensus is imposed on both primal and dual variables. Its complexity result, in terms of both gradient evaluation and communication, is $O(\vareps^{-2})$ to produce an $\vareps$-stationary solution. Moreover, to reduce the computational complexity dependence on the number of component functions, a momentum-based variance-reduction (VR) technique is employed to exploit the finite-sum structure. While the method in \cite{zhang2021taming} needs to periodically compute the full gradient, the decentralized method in  \cite{gao2022decentralized} uses the SAGA-type VR \cite{defazio2014saga}, and it does not need to compute a full gradient periodically but instead maintains an estimate of all component gradients. Its computation and communication complexity is in the same order as that in \cite{zhang2021taming}. Both of \cite{zhang2021taming, gao2022decentralized} consider smooth problems without a hard constraint or a nonsmooth regularizer. In contrast, \cite{chen2022simple} considers a slightly more general class of NCSC minimax problems, which can have a convex constraint on the dual variable. Compared to the results in \cite{zhang2021taming, defazio2014saga}, the computation complexity in \cite{chen2022simple} is similar, but its communication complexity has a better dependence on the condition number of the problem and the graph topology. However, the method in \cite{chen2022simple} relies on the use of multi-communication per computation step, which requires more coordination between the agents.

Decentralized nonconvex nonconcave (NCNC) minimax problems have been studied in \cite{liu2019decentralized, liu2020decentralized}. The work \cite{liu2019decentralized} gives a proximal-point based method. It first formulates the optimality conditions of the considered minimax problem, where a nonmontone operator $\cB+\cR$ is involved; see Eqn.(8) in \cite{liu2019decentralized} for details. Then it aims at finding the root of the resulting system by using  the resolvent of $\cB+\cR$ at each agent. Under the assumption of weak monotonicity of $\cB+\cR$ and the Minty VI condition \cite{liu2021first}, an $O(\vareps^{-2})$ iteration complexity result is established in \cite{liu2019decentralized} to produce an $\vareps$-stationary solution. However, the resolvent of $\cB+\cR$ is generally very expensive to compute, and this limits the applications of the algorithm in \cite{liu2019decentralized}. The method in \cite{liu2020decentralized} is a decentralized extension of the optimistic stochastic gradient method in \cite{liu2019towards}, which is a stochastic extragradient method. Different from the methods discussed above, it only needs a stochastic approximation of each gradient by using $O(1)$ samples. However, its computational complexity is much higher and reaches $O(\vareps^{-12})$ in order to find an $\vareps$-stationary solution. On stochastic VI problems that include minimax problems as special cases, \cite{beznosikov2021decentralized} presents a decentralized stochastic extragradient method, which is unified for strong-monotone, monotone, and non-monotone VIs. For the non-monotone case, it assumes the Minty VI condition, and the stationarity violation, however, will not approach to zero even if deterministic gradients are used. The Minty VI condition may not hold for NCSC minimax problems (see a simple example in the appendix); hence the existing results do not apply to our case.

\subsection{Contributions} Our contributions are three-fold. \emph{First}, we give an equivalent reformulation of \eqref{eq:min-max-prob} in a decentralized setting, under a mild assumption that $\dom(h)$ has a nonempty relative interior. The reformulation has a primal consensus constraint but eliminates the dual one by introducing a Lagrangian multiplier. \emph{Second}, the reformulation enables us to perform local (approximate) maximization on the dual variable, without coordination between the multiple agents. By this advantage, we propose a novel decentralized algorithm through performing gradient descent on the primal variable and (approximate) maximization on the dual variable. The algorithm does not need multi-communication per iteration. Also, the local (approximate) maximization can be terminated based on a checkable stopping condition. Thus it can be easily implemented. To the best of our knowledge, the proposed algorithm is the first one for solving decentralized NCSC minimax problems with regularization terms on both primal and dual variables, while existing methods in \cite{zhang2021taming, gao2022decentralized, chen2022simple} only apply to smooth problems that can at most have 
a convex constraint on the dual variable. 
\emph{Third}, we establish global convergence and iteration complexity results of the proposed algorithm. It can produce an $\vareps$-stationary solution in expectation (see Def.~\ref{def:opt-cond-phi-eps} below) of the reformulation within $O(\frac{\kappa}{\vareps^2})$ iterations. This leads to $O(\frac{\kappa}{\vareps^2})$ communication rounds and $\vx$-gradients, and $\tilde O(\frac{\kappa}{\vareps^2} \sqrt{\kappa_y})$\footnote{Throughout this paper, $\tilde O$-notation may hide a factor of $\ln\frac{1}{\vareps}$ within the big-$O$.} $\vy$-gradients, where $\kappa$ and $\kappa_y$ are two condition numbers; see \eqref{eq:kappa}. In addition, we show that when $h$ is smooth, to produce an $\vareps$-stationary point in expectation (see Remark~\ref{rm:stat-original}) of the original formulation, $O(\frac{\kappa^2}{\vareps^2})$ iterations will be sufficient, by establishing a relation between the near stationarities of the two formulations. Compared to existing results for smooth NCSC minimax problems, our result is better by a factor of $\kappa$ for $\vx$-gradients and $\frac{\kappa}{\sqrt{\kappa_y}}$ for $\vy$-gradients by ignoring the logarithmic factor; see Table~\ref{tab:complexity} for more detailed comparisons.

\renewcommand{\arraystretch}{1.2}
\begin{table}
\caption{Comparison of our decentralized method to a few existing ones for solving NCSC minimax problems in the form of \eqref{eq:min-max-prob}. Because of different settings of problems and algorithms, we specialize the complexity of the compared methods on smooth problems, i.e., $g\equiv 0$ and $h\equiv 0$, by using deterministic gradients. Here, $\iota_\cX$ and $\iota_\cY$ denote the indicator functions on convex sets $\cX$ and $\cY$; the results in \cite{tsaknakis2020decentralized, zhang2021taming} do not have an explicit dependence on $\kappa$, so we use $a$, $b$ and $c$ to denote the unknown orders; the more conditions required on $\vW$ are in addition to those in Assumption~\ref{assump-W} below.}\label{tab:complexity}
\resizebox{\textwidth}{!}{
\begin{tabular}{|c||c|c|c|c||c|c|c|}
\hline
\multirow{2}{*}{Reference} & \multicolumn{4}{|c||}{\textbf{Key Assumptions on Problem and Algorithm}} & \multicolumn{3}{|c|}{\textbf{Results for \eqref{eq:min-max-prob} if $g\equiv0, h\equiv0$}} \\\cline{2-8}
 &  setting for $(g, h)$ & More on $\vW$  & consensus &multi-comm. & \#comm. & \#$\vx$-grad. & \#$\vy$-grad. \\\hline\hline
\multirow{2}{*}{\cite{tsaknakis2020decentralized}} & $g\equiv 0, h = \iota_\cY$ & \multirow{2}{*}{symmetric} & only $\vx$ & \multirow{2}{*}{\xmark} & $O(\frac{\kappa^a}{\vareps^2})$ & $O(\frac{\kappa^a}{\vareps^2})$ & $\tilde O(\frac{\kappa^c}{\vareps^2})$ \\\cline{6-8}
& $g=\iota_\cX, h\equiv 0$  & &  only $\vy$ & & \multicolumn{3}{|c|}{not established}  \\ \hline
\cite{zhang2021taming} & $g\equiv 0, h \equiv 0$ & symmetric & $\vx$ and $\vy$ & \xmark & $O(\frac{\kappa^b}{\vareps^2})$ & \multicolumn{2}{|c|}{$O(\frac{\kappa^b}{\vareps^2})$} \\ \hline
\multirow{2}{*}{\cite{gao2022decentralized}} & \multirow{2}{*}{$g\equiv 0, h \equiv 0$} & symmetric & \multirow{2}{*}{$\vx$ and $\vy$} & \multirow{2}{*}{\xmark} & \multirow{2}{*}{$O(\frac{\kappa^3}{\vareps^2})$} & \multicolumn{2}{|c|}{\multirow{2}{*}{$O(\frac{\kappa^3}{\vareps^2})$}} \\ 
& & nonnegative & & & & \multicolumn{2}{|c|}{}  \\\hline\hline 
\multirow{2}{*}{\cite{chen2022simple}} & \multirow{2}{*}{$g\equiv0, h=\iota_\cY$} & symmetric  & \multirow{2}{*}{$\vx$ and $\vy$} & \multirow{2}{*}{\cmark} & \multirow{2}{*}{$O(\frac{\kappa^2}{\vareps^2})$} & \multicolumn{2}{|c|}{\multirow{2}{*}{$O(\frac{\kappa^2}{\vareps^2})$} } \\
& & $\vzero \preccurlyeq \vW \preccurlyeq \vI$ & & & & \multicolumn{2}{|c|}{}  \\\hline
\textbf{This paper} & convex functions & none & $\vx$ and $\vy$ & \xmark & $O(\frac{\kappa^2}{\vareps^2})$ & $O(\frac{\kappa^2}{\vareps^2})$ & $\tilde O(\frac{\kappa^{2}}{\vareps^2} \sqrt{\kappa_y})$ \\\hline
\end{tabular}
}
\end{table}

\subsection{Notation}
We use bold lower-case letters $\vx, \vy, \ldots$ for vectors and bold upper-case letters $\vX, \vY,\ldots$ for matrices. 
$\vzero$ and $\vone$ denote all-zero and all-one vectors; $\vI$ is reserved for the identity matrix. 
$[m]$ denotes the set $\{1,2,\ldots, m\}$. We use subscript $_i$ in local vectors owned by the $i$-th agent such as $\vx_i$; superscript $^t$ is used for the $t$-th iterate.
We let 
\begin{subequations}\label{eq:notation-xyz-F}
\begin{align}
&\vX = \big[ \vx_1, \ldots, \vx_m\big]^\top, \quad \vY = \big[ \vy_1, \ldots, \vy_m\big]^\top, \quad \vz_i := [\vx_i; \vy_i]\in\RR^{n_1+n_2},\forall\, i\in [m], \quad \vZ = \big[ \vz_1, \ldots, \vz_m\big]^\top, \label{eq:xyz}\\
& \nabla_\vx F(\vZ) = \big[\nabla_\vx f_1(\vz_1),\ldots, \nabla_\vx f_m(\vz_m) \big]^\top, \quad \nabla_\vy F(\vZ) = \big[\nabla_\vy f_1(\vz_1),\ldots, \nabla_\vy f_m(\vz_m) \big]^\top,\label{eq:nabla-F}
\end{align}
\end{subequations}
where $\nabla_\vx f_i$ takes the partial gradient about $\vx$ and $\nabla_\vy f_i$ about $\vy$.
Also, we define
\begin{equation}\label{eq:avg-perp-not}
\textstyle \vx_\avg = \frac{1}{m}\sum_{i=1}^m \vx_i,\ \vX_\avg = \vone \vx_\avg^\top, \ \vX_\perp = \vX - \vX_\avg.
\end{equation}
For a closed convex function $r$, $\prox_r(\vx):=\argmin_\vy \{r(\vy) + \frac{1}{2}\|\vy-\vx\|^2\}$ is the proximal mapping of $r$.

\subsection{Outline} The rest of the paper is organized as follows. Sect.~\ref{sec:alg} gives our algorithm. Convergence results are presented in Sect.~\ref{sec:analysis} and numerical results in Sect.~\ref{sec:numerical}. Finally, Sect.~\ref{sec:conclusion} concludes the paper.

\section{Decentralized formulations and proposed algorithm}\label{sec:alg}
With the introduction of local variables $\{(\vx_i, \vy_i)\}$, one can formulate \eqref{eq:min-max-prob} equivalently into a consensus-constrained problem as follows
\begin{equation}\label{eq:min-max-prob-dec}
\min_{\vx_1,\ldots,\vx_m} \max_{\vy_1, \ldots, \vy_m} \textstyle \frac{1}{m}\sum_{i=1}^m \big( f_i(\vx_i, \vy_i) + g(\vx_i) - h(\vy_i) \big), \st (\vx_i, \vy_i) = (\vx_j, \vy_j), \, \forall\, (i, j)\in \cE.
\end{equation}
To achieve the consensus, a mixing (or gossip) matrix $\vW\in\RR^{m\times m}$ is often used for neighbor (weighted) averaging. When $\vW$ satisfies certain conditions (see  Assumption~\ref{assump-W} below), the $\vy$-consensus 
can be expressed as $\vW\vY = \vY$. 
Hence, with a multiplier $\vLam\in\RR^{m\times n_2}$, we have from the strong duality \cite{rockafellar1997convex}, which holds if the relative interior of $\dom(h)$ is nonempty, that for each $\vX$, 
$\max_\vY\left\{\textstyle \frac{1}{m}\sum_{i=1}^m \big( f_i(\vx_i, \vy_i) - h(\vy_i)\big), \st \vW\vY = \vY\right\} = \min_\vLam \max_\vY \Phi(\vX, \vLam, \vY),$
where
\begin{equation}\label{eq:def-Phi}
\Phi(\vX, \vLam, \vY) := \textstyle \frac{1}{m}\sum_{i=1}^m \big( f_i(\vx_i, \vy_i) - h(\vy_i)\big) - \frac{L}{2\sqrt{m}} \big\langle \vLam, (\vW- \vI) \vY\big\rangle. 
\end{equation}
Therefore, we can reformulate \eqref{eq:min-max-prob-dec} equivalently into 
\begin{equation}\label{eq:min-max-Phi}
\min_{\vX, \vLam} \max_\vY \Phi(\vX, \vLam, \vY) + \textstyle \frac{1}{m}\sum_{i=1}^m g(\vx_i), \st \vx_i = \vx_j, \forall\, (i, j) \in \cE.
\end{equation}

Compared to \eqref{eq:min-max-prob-dec}, there is no $\vy$-consensus constraint in \eqref{eq:min-max-Phi}. This enables an aggressive way to update $\vY$, which can further lead to a larger stepsize for $\vx$-update; see Remark~\ref{rm:stepsize-set}. More precisely, with $\vX$ and $(\vW-\vI)^\top\vLam$ fixed, the $\vy$-subproblems on all agents are independent, and thus we can greedily update $\vY$ by (approximately) maximizing each local function, \emph{without the need of neighbor communication}. By the (approximate) $\vY$-maximizer, we can obtain $\nabla Q(\vX, \vLam)$ or its estimate, where $Q(\vX, \vLam):= \max_\vY \Phi(\vX,\vY, \vLam)$. 

Based on these observations, we propose a decentralized gradient descent maximization (D-GDMax) method for solving \eqref{eq:min-max-prob}. The pseudocode is shown in Algorithm~\ref{alg:dgdm}, where $d_i^t(\cdot)$ is defined as
\begin{equation}\label{eq:def-d-i}
d_i^t(\vy) := \textstyle f_i(\vx_i^t, \vy) - h(\vy) - \frac{L\sqrt{m}}{2} \big\langle \tilde\vlam_i^t, \vy\big\rangle, \forall\, i\in [m], \forall\, t \ge0.
\end{equation} 
At each iteration of D-GDMax, the agents, in parallel, first perform one round of neighbor communication of their local $\vx$-variables and do a local proximal gradient step to update $\vx$-variables along a tracked $\vx$-gradient, secondly update $\vlam$-variables by gradient descent with one round of neighbor communication of local $\vy$-variables, then update $\vy$-variables by (approximately) maximizing local functions, and finally with the (approximate) $\vy$-maximizers, compute local $\vx$-gradients and perform a gradient tracking step. 

\begin{algorithm}[h]
\caption{Decentralized Gradient Descent Maximization (D-GDMax) Method for \eqref{eq:min-max-prob}}\label{alg:dgdm}
\DontPrintSemicolon
{\bfseries Input:} $\vx^0 \in \dom(g)$, $\eta_x > 0$, $\eta_\lambda>0$, a nonnegative number sequence $\{\delta_t\}_{t\ge0}$. \;
{\bfseries Initialization:} for each $i\in [m]$, set $\vx_i^0 = \vx^0,\ \vlam_i^0 = \tilde\vlam_i^0=\vzero$, find $\vy_i^0$ such that $\dist\big(\vzero, \partial d_i^0(\vy_i^0)\big) \le \delta_0$, and let $\vv_i^0 = \nabla_\vx f_i(\vx_i^0, \vy_i^0), \forall \, i\in[m]$.\; 
\For{$t = 0, 1,\ldots,$}{
\For{$i=1,\ldots,m$ {\bfseries in parallel}}{
Let $\tilde\vx_i^t = \sum_{j=1}^m w_{ij} \vx_j^t$ and update $\vx_i^{t+1} = \prox_{\eta_x g}\big(\tilde\vx_i^t - \eta_x \vv_i^t\big)$.\;
Update $\vlam_i^{t+1} = \vlam_i^{t} + \frac{L\eta_\lambda}{2\sqrt{m}} \big(\sum_{j=1}^m w_{ij} \vy_j^t - \vy_i^t\big)$ and let $\tilde\vlam_i^{t+1} = \sum_{j=1}^m w_{ji} \vlam_j^{t+1} - \vlam_i^{t+1}$.\;
Find $\vy_i^{t+1}$ such that $\dist\big(\vzero, \partial d_i^{t+1}(\vy_i^{t+1})\big)\le \delta_{t+1}$, where $d_i^t$ is defined in \eqref{eq:def-d-i}.\;
Let $\vv_i^{t+1} = \sum_{j=1}^m w_{ij} \vv_j^t + \nabla_\vx f_i(\vx_i^{t+1}, \vy_i^{t+1}) - \nabla_\vx f_i(\vx_i^{t}, \vy_i^{t})$.
}
}
\end{algorithm}

A centralized version of our algorithm has been presented in \cite{jin2020local} for smooth minimax problems. However, \cite{jin2020local} does not exploit strong concavity, so its complexity is worse  than ours in the special centralized setting. In addition, though our reformulation in \eqref{eq:min-max-Phi} does not have a consensus constraint on $\vY$, the primal problem will not reduce to a standard decentralized minimization problem, because the term $\langle \vLam, (\vW-\vI)\vY\rangle$ will lead to coupling of the primal variables after $\vY$ is removed by maximization. Hence, our algorithm is fundamentally different from a decentralized method for solving a consensus minimization problem. 

\section{Convergence analysis}\label{sec:analysis}

In this section, we analyze the D-GDMax method in Algorithm~\ref{alg:dgdm}. The following (scalar or vector) functions will be used in our analysis or to state our results:
\begin{align}
& p(\vx) := \max_\vy \left\{\textstyle \frac{1}{m}\sum_{i=1}^m f_i(\vx, \vy) - h(\vy) \right\},\quad P(\vx, \vLam) := \max_\vY \Phi(\vone\vx^\top, \vY, \vLam), \label{eq:def-p}\\
& \phi(\vx, \vLam) := P(\vx, \vLam) + g(\vx), \quad S_\Phi(\vX,\vLam) := \argmax_\vY  \Phi(\vX, \vLam, \vY).\label{eq:def-P}
\end{align}
Here, $p$ is the smooth term for the primal problem of \eqref{eq:min-max-prob}; $P$ can be viewed as the smooth term for  the primal problem of \eqref{eq:min-max-Phi} and $\phi$ is the objective function. 

\subsection{Assumptions}
We make the following assumptions for the considered problem \eqref{eq:min-max-prob}.
\begin{assumption}\label{assump-func}
Problem \eqref{eq:min-max-prob} has a finite solution $\vx^*:=\argmin_\vx (p+g)(\vx)$. For each $i\in [m]$, $f_i$ is $L$-smooth in an open set $\cZ \supseteq \dom(g)\times \dom(h)$, and $f_i(\vx, \cdot)$ is $\mu$-strongly concave with $\mu>0$, 
i.e.,
\begin{align}
&\|\nabla f_i(\vz) - \nabla f_i(\tilde \vz)\| \le L\|\vz - \tilde\vz\|, \forall\, \vz, \tilde\vz \in \cZ,\label{eq:smooth-fi}\\
&\big\langle \vy -\tilde\vy, \nabla_\vy f_i(\vx, \vy) - \nabla_\vy f_i(\vx, \tilde\vy) \big\rangle \le -\mu \|\vy - \tilde \vy\|^2, \forall\, \vx \in \dom(g), \forall\, \vy, \tilde\vy \in \dom(h). 
\end{align}
\end{assumption}

The strong concavity condition in Assumption~\ref{assump-func} is also assumed in existing works (e.g., \cite{chen2022simple, zhang2021taming}) about decentralized methods for solving nonconvex minimax problems. Our setting is more general than that in those existing works, by including convex regularizers $g$ and $h$. 
In order to tighten the dependence of our complexity on certain constants, we let $L_y$ be the smoothness constant of $f_i(\vx, \cdot)$, namely, 
\begin{equation}\label{eq:smooth-y}
\|\nabla_\vy f_i(\vx, \vy) - \nabla_\vy f_i(\vx, \tilde\vy)\| \le L_y \|\vy - \tilde\vy\|, \forall \, \vy, \tilde\vy\in \dom(h),\ \forall\, \vx\in \dom(g),\ \forall\, i \in [m].
\end{equation}
The existence of $L_y$ is guaranteed by Assumption~\ref{assump-func}, and it can be much smaller than $L$. We define
\begin{equation}\label{eq:kappa}
\textstyle \kappa = \frac{L}{\mu}, \quad \kappa_y = \frac{L_y}{\mu}.
\end{equation}

\begin{assumption}\label{assump-relint}
The relative interior of $\dom(h)$ is nonempty.
\end{assumption}

 The condition above is mild and implies equivalence between \eqref{eq:min-max-prob-dec} and \eqref{eq:min-max-Phi} and $p(\vx) = \min_\vLam P(\vx,\vLam), \forall\, \vx$. 

\begin{assumption}\label{assump-W}
The mixing matrix $\vW\in\RR^{m\times m}$ satisfies the conditions: $\mathrm{(i)}$ $w_{ij} = 0$ if $j$ is not a neighbor of $i$; $\mathrm{(ii)}$ $\rho:= \|\vW-\frac{1}{m}\vone\vone^\top\|_2 < 1$; $\mathrm{(iii)}$ $\Null(\vW-\vI) = \Span\{\vone\}$ and $\vW^\top\vone = \vone$; $\mathrm{(iv)}$  $\|\vW-\vI\|_2 \le 2$.
\end{assumption}

In the above assumption, condition (i) indicates that $j$ can send message to $i$ \emph{only if} $(i, j)$ is an edge of  $\cG$; conditions (ii) and (iii) are essential for achieving consensus and can hold if $\cG$ is connected; condition (iv) is just for convenience, and our analysis can still go through as long as $\|\vW-\vI\|_2$ is bounded. 

The updates in Algorithm~\ref{alg:dgdm} are from the point view of each agent. It is easy to have the following more compact matrix format, and we will use it in our analysis for convenience and better readability: 
\begin{subequations}\label{eq:alg-update}
\begin{align}
& \vx^{t+1}_i  = \prox_{\eta_x g}(\tilde\vx_i^t - \eta_x \vv_i^t), \forall\, i\in [m], \text{ with }\widetilde\vX^t := \vW\vX^t \label{eq:alg-update-x}\\
& \vLam^{t+1} = \textstyle \vLam^t + \frac{L\eta_\lambda}{2\sqrt{m}} (\vW-\vI)\vY^t, \label{eq:alg-update-lam-v} \quad
 \vV^{t+1} = \vW\vV^{t} + \nabla_\vx F(\vX^{t+1}, \vY^{t+1}) - \nabla_\vx F(\vX^{t}, \vY^{t}). 
\end{align}
\end{subequations}

\noindent\textbf{Roadmap of analysis:} we will first show some important properties of $P$ and $S_\Phi$ and then build a key inequality after one iteration of updates based on the change of $\phi$. Moreover, we bound the consensus error of $\vX$ and $\vV$ iterates by using the contraction of $\vW-\frac{1}{m}\vone\vone^\top$. Finally, we combine the consensus error bounds with the one-iteration progress to show a square-summable result for the generated $\vX$ and $\vLam$ iterates, from which we will be able to show the global convergence and establish iteration complexity results. 

\subsection{Preparatory results}
The lemma below shows some important properties of $P$ and $S_\Phi$.
\begin{proposition}[Smoothness of $P$ and Lipschitz continuity of $S_\Phi$]\label{prop:P-S}
Let $P$ and $S_\Phi$ be defined in \eqref{eq:def-p} and \eqref{eq:def-P}. Then $P$ is $L_P$-smooth with $L_P=L\sqrt{4\kappa^2 + 1}$ and $S_\Phi$ is Lipschitz continuous. More precisely,   
\begin{subequations}
\begin{align}
&\|\nabla P(\vx, \vLam) - \nabla P(\tilde\vx, \tilde\vLam)\|_F^2 \le L_P^2\left(\|\vx-\tilde\vx\|^2 + \|\vLam - \tilde\vLam\|_F^2\right),\forall\, \vx,\tilde\vx\in \dom(g); \, \forall\,\vLam, \tilde\vLam, \label{eq:smooth-P} \\
\label{eq:lip-S}
&\|S_\Phi(\vX,\vLam) - S_\Phi(\widetilde\vX,\vLam)\|_F^2 \le \kappa^2 \|\vX - \widetilde\vX\|_F^2, \forall\, \vX, \widetilde\vX\in\dom(g)^m;\ \forall\, \vLam,\\
&\|S_\Phi(\vX,\vLam) - S_\Phi(\widetilde\vX,\tilde\vLam)\|_F^2 \le 2\kappa^2 \|\vX - \widetilde\vX\|_F^2 + 2m\kappa^2\|\vLam - \tilde\vLam\|_F^2, \forall\, \vX, \widetilde\vX\in\dom(g)^m;\ \forall\, \vLam, \tilde\vLam. \label{eq:lip-S2}
\end{align}
\end{subequations}
\end{proposition}

\begin{proof}
For any $(\vX,\vLam)$ and $(\widetilde\vX, \tilde\vLam)$ with $\vX, \widetilde\vX\in \dom(g)^m$, let $\vY = S_\Phi(\vX,\vLam)$  and $\widetilde\vY = S_\Phi(\widetilde\vX,\tilde\vLam)$. From the optimality condition, it holds that there exists $\vzeta_i\in\partial h(\vy_i), \forall\, i\in [m]$ such that
\begin{equation}\label{eq:opt-S-Phi}
\textstyle \frac{1}{m} \Big[ \nabla_\vy f_1(\vx_1,\vy_1) - \vzeta_1,  
\cdots,
\nabla_\vy f_m(\vx_m,\vy_m) - \vzeta_m
\Big]^\top - \frac{L}{2\sqrt{m}} (\vW- \vI)^\top \vLam = \vzero.
\end{equation}
Hence,
$\frac{1}{m} \sum_{i=1}^m\big\langle \tilde\vy_i - \vy_i, \nabla_\vy f_i(\vx_i,\vy_i) - \vzeta_i \big\rangle - \frac{L}{2\sqrt{m}} \big\langle \widetilde\vY - \vY, (\vW- \vI)^\top \vLam\big\rangle = 0.$
Similarly, there exists $\tilde\vzeta_i\in\partial h(\tilde\vy_i)$ for each $i$ such that
$\frac{1}{m} \sum_{i=1}^m\big\langle \vy_i - \tilde \vy_i, \nabla_\vy f_i(\tilde\vx_i,\tilde\vy_i) - \tilde\vzeta_i \big\rangle - \frac{L}{2\sqrt{m}} \big\langle \vY - \widetilde\vY, (\vW- \vI)^\top \tilde\vLam\big\rangle = 0.$
Adding the two equalities together, we obtain 
\begin{equation}\label{eq:add-y-opt}
\textstyle \frac{1}{m} \sum_{i=1}^m\big\langle \tilde\vy_i - \vy_i, \nabla_\vy f_i(\vx_i,\vy_i) - \nabla_\vy f_i(\tilde\vx_i,\tilde\vy_i) + \tilde\vzeta_i - \vzeta_i \big\rangle + \frac{L}{2\sqrt{m}} \big\langle \widetilde\vY - \vY, (\vW- \vI)^\top (\tilde\vLam-\vLam)\big\rangle = 0.
\end{equation}
The $\mu$-strong concavity of each $f_i(\vx_i, \cdot)$ implies 
$\textstyle \frac{1}{m} \sum_{i=1}^m\big\langle \tilde\vy_i - \vy_i, \nabla_\vy f_i(\vx_i,\vy_i) - \nabla_\vy f_i(\vx_i,\tilde\vy_i) \big\rangle \ge \frac{\mu}{m} \sum_{i=1}^m \|\tilde\vy_i - \vy_i\|^2.$
Also, from the convexity of $h$, it follows
$\frac{1}{m} \sum_{i=1}^m\big\langle \tilde\vy_i - \vy_i, \tilde\vzeta_i - \vzeta_i \big\rangle \ge 0.$
Substitute these two inequalities into \eqref{eq:add-y-opt} to have
\begin{equation}\label{eq:add-y-opt2}
\textstyle \frac{1}{m} \sum_{i=1}^m\big\langle \tilde\vy_i - \vy_i, \nabla_\vy f_i(\vx_i,\tilde\vy_i) - \nabla_\vy f_i(\tilde\vx_i,\tilde\vy_i) \big\rangle + \frac{L}{2\sqrt{m}} \big\langle \widetilde\vY - \vY, (\vW- \vI)^\top (\tilde\vLam-\vLam)\big\rangle \ge \frac{\mu}{m} \sum_{i=1}^m \|\tilde\vy_i - \vy_i\|^2.
\end{equation}

Moreover, by the smoothness of $f_i$, it holds $\|\nabla_\vy f_i(\vx_i,\tilde\vy_i) - \nabla_\vy f_i(\tilde\vx_i,\tilde\vy_i)\| \le L\|\vx_i - \tilde \vx_i\|$, and thus using the Young's inequality, we obtain $\big\langle \tilde\vy_i - \vy_i, \nabla_\vy f_i(\vx_i,\tilde\vy_i) - \nabla_\vy f_i(\tilde\vx_i,\tilde\vy_i) \big\rangle \le \frac{a}{4}\|\tilde\vy_i - \vy_i\|^2 + \frac{L^2}{a}\|\vx_i - \tilde \vx_i\|^2$ for any $a>0$. Hence, \eqref{eq:add-y-opt2} implies
\begin{equation}\label{eq:add-y-opt3}
\textstyle \frac{\mu}{m} \sum_{i=1}^m \|\tilde\vy_i - \vy_i\|^2 \le \frac{1}{m} \sum_{i=1}^m\left(\frac{a}{4}\|\tilde\vy_i - \vy_i\|^2 + \frac{L^2}{a}\|\vx_i - \tilde \vx_i\|^2 \right) +  \frac{L}{2\sqrt{m}} \big\langle\widetilde\vY - \vY, (\vW- \vI)^\top (\tilde\vLam-\vLam)\big\rangle. 
\end{equation}
Below we discuss two different cases to obtain the desired results.

\textbf{Case I.} When $\vLam = \tilde\vLam$, we take $a = 2\mu$ in \eqref{eq:add-y-opt3} and rearrange terms to have $\|\widetilde\vY - \vY\|_F^2 \le \frac{L^2}{\mu^2}\|\widetilde\vX - \vX\|_F^2$, which indicates \eqref{eq:lip-S} by the definition of $\kappa=\frac{L}{\mu}$ in \eqref{eq:kappa}. 

\textbf{Case II.} Generally, we first use the Young's inequality to have 
$$\textstyle \frac{L}{2\sqrt{m}} \big\langle\widetilde\vY - \vY, (\vW- \vI)^\top (\tilde\vLam-\vLam)\big\rangle \le \frac{\mu}{4m}\|\widetilde\vY - \vY\|_F^2 + \frac{L^2}{4\mu}\|(\vW- \vI)^\top (\tilde\vLam-\vLam)\|_F^2.$$
Then, we take $a=\mu$ in \eqref{eq:add-y-opt3} and obtain, by rearranging terms, that
$\frac{1}{2m}\|\widetilde\vY - \vY\|_F^2 \le \frac{L^2}{m\mu^2}\|\vX-\widetilde\vX\|_F^2 + \frac{L^2}{4\mu^2} \|(\vW- \vI)^\top (\tilde\vLam-\vLam)\|_F^2,$
which together with the fact $\|\vW- \vI\|_2\le 2$ gives
\begin{equation}\label{eq:add-y-opt4-1}
\textstyle \frac{1}{2m}\|\widetilde\vY - \vY\|_F^2 \le \frac{L^2}{m\mu^2}\|\vX-\widetilde\vX\|_F^2 + \frac{L^2}{\mu^2}\|\tilde\vLam-\vLam\|_F^2 = \frac{\kappa^2}{m}\|\vX-\widetilde\vX\|_F^2 + \kappa^2\|\tilde\vLam-\vLam\|_F^2.
\end{equation}
We thus obtain \eqref{eq:lip-S2} by applying $2m$ to both sides of \eqref{eq:add-y-opt4-1}.

Now with $\vX=\vone\vx^\top$ and $\widetilde\vX=\vone\tilde\vx^\top$, we notice $\nabla P(\vx, \vLam) = \left(\frac{1}{m}\sum_{i=1}^m \nabla_\vx f_i(\vx, \vy_i),\ -\frac{L}{2\sqrt{m}}(\vW-\vI)\vY\right)$ by the Danskin's Theorem \cite{danskin1966theory}.
Hence,
\begin{align*}
&\,\|\nabla P(\vx, \vLam) - \nabla P(\tilde\vx, \tilde\vLam)\|_F^2 =  \textstyle \frac{1}{m^2}\big\|\sum_{i=1}^m \nabla_\vx f_i(\vx, \vy_i) - \sum_{i=1}^m \nabla_\vx f_i(\tilde\vx, \tilde\vy_i)\big\|^2 + \frac{L^2}{4m}\big\|(\vW-\vI)(\vY -\widetilde\vY)\big\|_F^2\cr
\le &\, \textstyle \frac{L^2}{m}\sum_{i=1}^m\big(\|\vx-\tilde\vx\|^2 + \|\vy_i-\tilde\vy_i\|^2\big) + \frac{L^2}{m}\|\vY -\widetilde\vY\|_F^2
\overset{\eqref{eq:add-y-opt4-1}}\le  L^2 \|\vx-\tilde\vx\|^2 + 4L^2\kappa^2 \left(\|\vx-\tilde\vx\|^2 + \|\tilde\vLam-\vLam\|_F^2\right),
\end{align*}
which implies \eqref{eq:smooth-P}.
 This completes the proof.
\end{proof}

The proposition below bounds the approximate maximizer $\vY^t$ to the exact one.
\begin{proposition}\label{prop:error-bd-vYt}
Let $\{\vY^t\}$ be generated from Alg.~\ref{alg:dgdm}. Then $\big\|\vY^t - S_\Phi(\vX^t, \vLam^t)\big\|_F^2 \le \frac{m\delta_t^2}{\mu^2}$ for any $t\ge 0$.
\end{proposition}

\begin{proof}
Let ${\vY}^{t\star}= S_\Phi(\vX^t, \vLam^t)$, i.e., $\vy^{t\star}_i = \argmax_{\vy_i} d_i^t(\vy_i)$ for each $i\in [m]$. Then by the $\mu$-strong concavity of $d_i^t$, it follows that $\mu\|\vy^{t\star}_i - \vy_i^t\|^2 \le \big\langle \vy^{t\star}_i - \vy_i^t, \vxi_i^t\big\rangle$ for any $\vxi_i^t \in \partial d_i^t(\vy_i^t)$. Hence using the Cauchy-Schwarz inequality and by the condition $\dist\big(\vzero, \partial d_i^t(\vy_i^t)\big) \le \delta_t$, we obtain  $\mu\|\vy^{t\star}_i - \vy_i^t\|^2 \le \delta_t \|\vy^{t\star}_i - \vy_i^t\|$ and thus $\|\vy^{t\star}_i - \vy_i^t\|\le \frac{\delta_t}{\mu}$ for each $i\in [m]$. This implies the desired result.
\end{proof}

\subsection{Global convergence and convergence rate results}
We first establish a key one-iteration progress inequality by using the following two lemmas.
\begin{lemma}[{\cite[Lemma~C.3]{mancino2022proximal}}]\label{lem:g-step}
It holds that for any $t\ge0$ and each $i\in [m]$,
\begin{equation}\label{eq:g-step}
\textstyle g(\vx_i^{t+1}) - g(\vx_\avg^t) + \big\langle \vx_i^{t+1} - \vx_\avg^t, \vv_i^t\big\rangle \le -\frac{1}{2\eta_x}\left(\|\vx_i^{t+1} - \vx_\avg^t\|^2 + \|\vx_i^{t+1} - \tilde\vx_i^t\|^2 - \|\tilde\vx_i^t- \vx_\avg^t\|^2\right).
\end{equation}
\end{lemma}

\begin{lemma}\label{lem:one-step}
For any $t\ge0$, it holds with $L_P=L\sqrt{4\kappa^2 + 1}$ that
\begin{equation}\label{eq:split-P-v}
\begin{aligned}
&~\textstyle \big\langle \nabla P(\vx_\avg^{t}, \vLam^{t}), (\vx_\avg^{t+1}-\vx_\avg^{t}, \vLam^{t+1}-\vLam^{t}) \big\rangle - \frac{1}{m}\sum_{i=1}^m \big\langle \vx_i^{t+1} - \vx_\avg^t, \vv_i^t\big\rangle\\
\le  &~ \textstyle \frac{1}{2m} \big(L_P\|\vX^{t+1} - \vX_\avg^t\|_F^2 + \frac{1}{L_P}\|\vV_\perp^t\|_F^2\big)  + \frac{L}{2m} \big((2\kappa+1)\|\vX_\perp^{t}\|_F^2+(\kappa+1)\|\vX_\avg^{t+1}-\vX_\avg^{t}\|_F^2\big) + \frac{2\delta_t^2}{\mu} \\
& ~ \textstyle + \frac{L\kappa}{m}\|\vX_\perp^t\|_F^2 - \big(\frac{1}{\eta_\lambda}-\frac{L\kappa}{2}\big)\|\vLam^{t+1}-\vLam^{t}\|_F^2.
\end{aligned}
\end{equation}
\end{lemma}
 
\begin{proof}
Let $\widehat\vY^t:= S_\Phi\big(\vone(\vx_\avg^t)^\top, \vLam^t\big)$. Then from the Danskin's theorem \cite{danskin1966theory}, it holds
\begin{equation}\label{eq:grad-P}
\textstyle \nabla P(\vx_\avg^{t}, \vLam^{t}) = \left(\frac{1}{m}\sum_{i=1}^m \nabla_\vx f_i(\vx_\avg^{t}, \hat\vy_i^t),\ -\frac{L}{2\sqrt{m}}(\vW-\vI)\widehat\vY^t\right).
\end{equation}
Thus we have 
\begin{equation}\label{eq:cross-gradP}
\begin{aligned}
&~\textstyle \left\langle \nabla P(\vx_\avg^{t}, \vLam^{t}), (\vx_\avg^{t+1}-\vx_\avg^{t}, \vLam^{t+1}-\vLam^{t})\right\rangle \\
= & ~ \textstyle \big\langle \frac{1}{m}\sum_{i=1}^m \nabla_\vx f_i(\vx_\avg^{t}, \hat\vy_i^t),  \vx_\avg^{t+1}-\vx_\avg^{t}\big\rangle - \frac{L}{2\sqrt{m}}\big\langle(\vW-\vI)\widehat\vY^t, \vLam^{t+1}-\vLam^{t} \big\rangle.
\end{aligned}
\end{equation}
For the first inner product term on the right hand side of \eqref{eq:cross-gradP}, we split it as follows:
\begin{equation}\label{eq:split}
\begin{aligned}
\textstyle \left\langle \frac{1}{m}\sum_{i=1}^m \nabla_\vx f_i(\vx_\avg^{t}, \hat\vy_i^t),  \vx_\avg^{t+1}-\vx_\avg^{t}\right\rangle = &~ \textstyle \frac{1}{m}\sum_{i=1}^m \big\langle  \nabla_\vx f_i(\vx_i^{t}, \vy_i^t),  \vx_\avg^{t+1}-\vx_\avg^{t}\big\rangle \\
& \hspace{-2cm}+ \textstyle  \frac{1}{m}\sum_{i=1}^m \big\langle  \nabla_\vx f_i(\vx_\avg^{t}, \hat\vy_i^t) - \nabla_\vx f_i(\vx_i^{t}, \hat\vy_i^t),  \vx_\avg^{t+1}-\vx_\avg^{t}\big\rangle \\
& \hspace{-2cm} \textstyle + \frac{1}{m}\sum_{i=1}^m \big\langle \nabla_\vx f_i(\vx_i^{t}, \hat\vy_i^t) -  \nabla_\vx f_i(\vx_i^{t}, \vy_i^t),  \vx_\avg^{t+1}-\vx_\avg^{t}\big\rangle.
\end{aligned}
\end{equation}
By the smoothness of each $f_i$ and the Young's inequality, we have
\begin{equation}\label{eq:split1}
\begin{aligned}
&\textstyle  \frac{1}{m}\sum_{i=1}^m \big\langle  \nabla_\vx f_i(\vx_\avg^{t}, \hat\vy_i^t) - \nabla_\vx f_i(\vx_i^{t}, \hat\vy_i^t),  \vx_\avg^{t+1}-\vx_\avg^{t}\big\rangle \\
\le &~ \textstyle \frac{L}{2m} \sum_{i=1}^m\big(\|\vx_\avg^{t} - \vx_i^t\|^2+\|\vx_\avg^{t+1}-\vx_\avg^{t}\|^2\big) = \frac{L}{2m} \big(\|\vX_\perp^t\|_F^2 + \|\vX_\avg^{t+1} - \vX_\avg^t\|_F^2\big),
\end{aligned}
\end{equation}
where we have used the notation in \eqref{eq:avg-perp-not}.
Similarly, it holds
\begin{equation}\label{eq:split2}
\begin{aligned}
&\textstyle  \frac{1}{m}\sum_{i=1}^m \big\langle \nabla_\vx f_i(\vx_i^{t}, \hat\vy_i^t) - \nabla_\vx f_i(\vx_i^{t}, \vy_i^t),  \vx_\avg^{t+1}-\vx_\avg^{t}\big\rangle \\
\le &~ \textstyle \frac{L}{2m} \sum_{i=1}^m\big(\frac{1}{\kappa}\|\hat\vy_i^t - \vy_i^t\|^2+\kappa\|\vx_\avg^{t+1}-\vx_\avg^{t}\|^2\big) = \frac{L}{2m}\big(\frac{1}{\kappa}\|\widehat\vY^t - \vY^t\|_F^2+\kappa\|\vX_\avg^{t+1} - \vX_\avg^t\|_F^2\big).
\end{aligned}
\end{equation}
By the definition of $\widehat\vY^t$ and the choice of $\vY^t$, we have from Propositions~\ref{prop:P-S} and \ref{prop:error-bd-vYt} that 
\begin{equation}\label{eq:rel-y-x}
\begin{aligned}
\|\widehat\vY^t - \vY^t\|_F^2 \le & \, 2\|\widehat\vY^t - S_\Phi(\vX^t,  \vLam^t)\|_F^2 + 2 \|S_\Phi(\vX^t,  \vLam^t)-\vY^t\|_F^2 
\le  \textstyle 
2\kappa^2 \|\vX_\perp^t\|_F^2 + \frac{2m\delta_t^2}{\mu^2},
\end{aligned}
\end{equation}
 and thus \eqref{eq:split2}, together with the definition of $\kappa=\frac{L}{\mu}$, implies
\begin{equation}\label{eq:split3}
\begin{aligned}
&\textstyle  \frac{1}{m}\sum_{i=1}^m  \big\langle\nabla_\vx f_i(\vx_i^{t}, \hat\vy_i^t) - \nabla_\vx f_i(\vx_i^{t}, \vy_i^t),  \vx_\avg^{t+1}-\vx_\avg^{t}\big\rangle 
\le  \textstyle \frac{L}{2m} \big(2\kappa\|\vX_\perp^{t}\|_F^2+\kappa\|\vX_\avg^{t+1}-\vX_\avg^{t}\|_F^2\big) + \frac{\delta_t^2}{\mu}.
\end{aligned}
\end{equation}
Plugging \eqref{eq:split1} and \eqref{eq:split3} into \eqref{eq:split} gives
\begin{equation}\label{eq:split4}
\begin{aligned}
&\textstyle \big\langle \frac{1}{m}\sum_{i=1}^m \nabla_\vx f_i(\vx_\avg^{t}, \hat\vy_i^t),  \vx_\avg^{t+1}-\vx_\avg^{t}\big\rangle \\
\le &~ \textstyle \frac{1}{m}\sum_{i=1}^m  \big\langle \nabla_\vx f_i(\vx_i^{t}, \vy_i^t),  \vx_\avg^{t+1}-\vx_\avg^{t}\big\rangle  + \frac{L}{2m} \big((2\kappa+1)\|\vX_\perp^{t}\|_F^2+(\kappa+1)\|\vX_\avg^{t+1}-\vX_\avg^{t}\|_F^2\big) + \frac{\delta_t^2}{\mu}.
\end{aligned}
\end{equation}

In addition, notice $\vv_\avg^t = \frac{1}{m}\sum_{i=1}^m \nabla_\vx f_i(\vx_i^{t}, \vy_i^t)$.  
Hence,
\begin{align*}
&~\textstyle \frac{1}{m}\sum_{i=1}^m \big\langle  \nabla_\vx f_i(\vx_i^{t}, \vy_i^t),  \vx_\avg^{t+1}-\vx_\avg^{t}\big\rangle - \frac{1}{m}\sum_{i=1}^m \big\langle \vx_i^{t+1} - \vx_\avg^t, \vv_i^t\big\rangle\cr
= & ~ \textstyle \frac{1}{m} \sum_{i=1}^m \big\langle \vx_i^{t+1} - \vx_\avg^t, \vv_\avg^t - \vv_i^t\big\rangle
 \le \textstyle \frac{1}{2m} \sum_{i=1}^m \big(L_P\|\vx_i^{t+1} - \vx_\avg^t\|^2 + \frac{1}{L_P}\|\vv_\avg^t - \vv_i^t\|^2\big),
\end{align*}
where $L_P$ is given in Proposition~\ref{prop:P-S}.
The inequality above together with \eqref{eq:split4} gives
\begin{equation}\label{eq:split5}
\begin{aligned}
&\textstyle \big\langle \frac{1}{m}\sum_{i=1}^m \nabla_\vx f_i(\vx_\avg^{t}, \hat\vy_i^t),  \vx_\avg^{t+1}-\vx_\avg^{t}\big\rangle - \frac{1}{m}\sum_{i=1}^m \big\langle \vx_i^{t+1} - \vx_\avg^t, \vv_i^t\big\rangle\\
\le &~ \textstyle \frac{1}{2m} \big(L_P\|\vX^{t+1} - \vX_\avg^t\|_F^2 + \frac{1}{L_P}\|\vV_\perp^t\|_F^2\big)  + \frac{L}{2m} \big((2\kappa+1)\|\vX_\perp^{t}\|_F^2+(\kappa+1)\|\vX_\avg^{t+1}-\vX_\avg^{t}\|_F^2\big) + \frac{\delta_t^2}{\mu}.
\end{aligned}
\end{equation}
Moreover, by Young's inequality and the update formula of $\vLam$ in \eqref{eq:alg-update-lam-v}, it holds
\begin{align}\label{eq:split6}
&~\textstyle - \frac{L}{2\sqrt{m}}\big\langle(\vW-\vI)\widehat\vY^t, \vLam^{t+1}-\vLam^{t} \big\rangle\cr
=&~ \textstyle - \frac{L}{2\sqrt{m}}\big\langle(\vW-\vI)(\widehat\vY^t - \vY^t) , \vLam^{t+1}-\vLam^{t} \big\rangle - \frac{L}{2\sqrt{m}}\big\langle(\vW-\vI)\vY^t, \vLam^{t+1}-\vLam^{t} \big\rangle\cr 
\le & ~ \textstyle \frac{L}{2m\kappa}\|\widehat\vY^t - \vY^t\|_F^2 + \frac{L\kappa}{8}\|(\vW-\vI)^\top (\vLam^{t+1}-\vLam^{t})\|_F^2 - \frac{1}{\eta_\lambda}\|\vLam^{t+1}-\vLam^{t}\|_F^2.
\end{align}
Now add \eqref{eq:split5} and \eqref{eq:split6} to \eqref{eq:cross-gradP} and use the notations in \eqref{eq:avg-perp-not} to obtain  
\begin{equation}\label{eq:split7}
\begin{aligned}
&~\textstyle \big\langle \nabla P(\vx_\avg^{t}, \vLam^{t}), (\vx_\avg^{t+1}-\vx_\avg^{t}, \vLam^{t+1}-\vLam^{t})\rangle - \frac{1}{m}\sum_{i=1}^m \big\langle \vx_i^{t+1} - \vx_\avg^t, \vv_i^t\big\rangle\\
\le  &~ \textstyle \frac{1}{2m} \big(L_P\|\vX^{t+1} - \vX_\avg^t\|_F^2 + \frac{1}{L_P}\|\vV_\perp^t\|_F^2\big)  + \frac{L}{2m} \big((2\kappa+1)\|\vX_\perp^{t}\|_F^2+(\kappa+1)\|\vX_\avg^{t+1}-\vX_\avg^{t}\|_F^2\big) + \frac{\delta_t^2}{\mu} \\
& ~ \textstyle + \frac{L}{2m\kappa}\|\widehat\vY^t - \vY^t\|_F^2 + \frac{L\kappa}{8}\|(\vW-\vI)^\top (\vLam^{t+1}-\vLam^{t})\|_F^2 - \frac{1}{\eta_\lambda}\|\vLam^{t+1}-\vLam^{t}\|_F^2,
\end{aligned}
\end{equation}
which, together with $\|\vW-\vI\|_2\le 2$ and \eqref{eq:rel-y-x}, gives the desired result in \eqref{eq:split-P-v}.
\end{proof}

Now we are ready to show the one-iteration progress inequality.
\begin{theorem}\label{thm:one-iter}
Let $\{(\vX^t, \vLam^t, \vV^t)\}$ be generated from Alg.~\ref{alg:dgdm}. Then we have with $L_P=L\sqrt{4\kappa^2 + 1}$ that
\begin{equation}\label{eq:phi-diff3}
\begin{aligned}
& ~ \textstyle\frac{1}{2m}\big(\frac{1}{\eta_x}-2L_P - L(\kappa+1)\big)\|\vX^{t+1}-\vX_\avg^{t}\|^2 + \frac{1}{2m\eta_x}\|\vX^{t+1} - \widetilde\vX^t\|_F^2 + \big(\frac{1}{\eta_\lambda}- \frac{L_P+L\kappa}{2} \big) \|\vLam^{t+1}-\vLam^{t}\|_F^2 \\
 \le & ~ \textstyle \phi(\vx_\avg^{t}, \vLam^{t}) - \phi(\vx_\avg^{t+1}, \vLam^{t+1}) + \frac{1}{2m L_P} \|\vV_\perp^t\|_F^2 + \big(\frac{\rho^2}{2m\eta_x}+ \frac{L(4\kappa+1)}{2m} \big) \|\vX_\perp^{t}\|_F^2  +  \frac{2\delta_t^2}{\mu} .
\end{aligned}
\end{equation}
\end{theorem}

\begin{proof}
From the $L_P$-smoothness of $P$ in Proposition~\ref{prop:P-S}, it follows that
\begin{equation}\label{eq:phi-diff0}
\begin{aligned}
\textstyle \phi(\vx_\avg^{t+1}, \vLam^{t+1}) - \phi(\vx_\avg^{t}, \vLam^{t})
 \le & ~ \textstyle \frac{L_P}{2}\left(\|\vx_\avg^{t+1}-\vx_\avg^{t}\|^2 + \|\vLam^{t+1}-\vLam^{t}\|_F^2\right) \\
 &\hspace{-1cm}+ \big\langle \nabla P(\vx_\avg^{t}, \vLam^{t}), (\vx_\avg^{t+1}-\vx_\avg^{t}, \vLam^{t+1}-\vLam^{t})\big\rangle + g(\vx_\avg^{t+1}) - g(\vx_\avg^{t}).
\end{aligned}
\end{equation}
In addition, by the convexity of $g$, it holds $g(\vx_\avg^{t+1}) - g(\vx_\avg^{t}) \le \frac{1}{m}\sum_{i=1}^m \big(g(\vx_i^{t+1}) - g(\vx_\avg^{t})\big)$. Hence, we have from \eqref{eq:g-step} and \eqref{eq:phi-diff0} that
\begin{equation}\label{eq:phi-diff1}
\begin{aligned}
\phi(\vx_\avg^{t+1}, \vLam^{t+1}) - \phi(\vx_\avg^{t}, \vLam^{t})
 \le & ~ \textstyle \frac{L_P}{2}\left(\|\vx_\avg^{t+1}-\vx_\avg^{t}\|^2 + \|\vLam^{t+1}-\vLam^{t}\|_F^2\right) \\
 &\hspace{-2cm} \textstyle + \big\langle \nabla P(\vx_\avg^{t}, \vLam^{t}), (\vx_\avg^{t+1}-\vx_\avg^{t}, \vLam^{t+1}-\vLam^{t})\rangle - \frac{1}{m}\sum_{i=1}^m \big\langle \vx_i^{t+1} - \vx_\avg^t, \vv_i^t\big\rangle \\
 &\hspace{-2cm} \textstyle -\frac{1}{2m\eta_x}\sum_{i=1}^m\left(\|\vx_i^{t+1} - \vx_\avg^t\|^2 + \|\vx_i^{t+1} - \tilde\vx_i^t\|^2 - \|\tilde\vx_i^t- \vx_\avg^t\|^2\right).
\end{aligned}
\end{equation}
Adding \eqref{eq:split-P-v} to \eqref{eq:phi-diff1} and combining like terms yield
\begin{equation}\label{eq:phi-diff2}
\begin{aligned}
\phi(\vx_\avg^{t+1}, \vLam^{t+1}) - \phi(\vx_\avg^{t}, \vLam^{t})
 \le & ~ \textstyle \frac{L_P + L(\kappa+1)}{2m}\|\vX_\avg^{t+1}-\vX_\avg^{t}\|^2 - \big(\frac{1}{\eta_\lambda}- \frac{L\kappa}{2} - \frac{L_P}{2}\big) \|\vLam^{t+1}-\vLam^{t}\|_F^2 \\
 &\hspace{-2cm}+ \textstyle \frac{1}{2m} \big(L_P\|\vX^{t+1} - \vX_\avg^t\|_F^2 + \frac{1}{L_P}\|\vV_\perp^t\|_F^2\big) +  \frac{L(4\kappa+1)}{2m} \|\vX_\perp^{t}\|_F^2 +  \frac{2\delta_t^2}{\mu} \\
 &\hspace{-2cm} \textstyle -\frac{1}{2m\eta_x}\left(\|\vX^{t+1} - \vX_\avg^t\|_F^2 + \|\vX^{t+1} - \widetilde\vX^t\|_F^2 - \|\widetilde\vX^t- \vX_\avg^t\|_F^2\right).
\end{aligned}
\end{equation}
Since $\vW\vone = \vone$, it holds $\vW - \frac{1}{m}\vone\vone^\top=(\vW - \frac{1}{m}\vone\vone^\top)(\vI - \frac{1}{m}\vone\vone^\top)$. Hence, $\|\widetilde\vX^t- \vX_\avg^t\|_F=\|(\vW - \frac{1}{m}\vone\vone^\top)\vX_\perp^t\|_F \le \rho \|\vX_\perp^t\|_F$. Also, use the fact $\|\vX_\avg^{t+1}-\vX_\avg^{t}\|_F^2 \le \|\vX^{t+1}-\vX_\avg^{t}\|_F^2$ and rearrange terms in \eqref{eq:phi-diff2} to complete the proof.
\end{proof}

Below, we bound $\{\vX^t_\perp\}$ and $\{\vV^t_\perp\}$ and then combine with \eqref{eq:phi-diff3} to show a square-summable result on the iterates. The following results are from Lemma~C.7 and the end of the proof of Lemma~C.13 in \cite{mancino2022proximal}:
\begin{equation}\label{eq:x-perp}
\textstyle \|\vX_\perp^{t+1}\|_F^2 \le \rho\|\vX_\perp^t\|_F^2 + \frac{\eta_x^2}{1-\rho}\|\vV_\perp^t\|_F^2, \quad \textstyle \|\vX^{t+1}-\vX^t\|_F^2 \le 2\|\vX^{t+1}-\widetilde\vX^t\|_F^2 + 8\|\vX_\perp^t\|_F^2.
\end{equation}
With the two inequalities above, we are ready to bound the consensus errors.
\begin{lemma}\label{lem:consensus}
For any positive integer $T$, it holds
\begin{equation}\label{eq:v-perp4}
\begin{aligned}
\textstyle (1-\rho)\sum_{t=0}^{T-1}\|\vV_\perp^t\|_F^2 \le & ~ \textstyle \big\|\vV_\perp^{0}\big\|_F^2 + \frac{6m L^2\kappa^2}{1-\rho} \sum_{t=1}^{T-1} \|\vLam^t - \vLam^{t-1}\|_F^2\\
 & ~ \textstyle  \hspace{-2cm}+ \frac{2L^2(1+ 6\kappa^2)}{1-\rho} \sum_{t=1}^{T-1} \big(\|\vX^{t}-\widetilde\vX^{t-1}\|_F^2 + 4\|\vX_\perp^{t-1}\|_F^2\big) + \frac{3m\kappa^2}{1-\rho}\sum_{t=1}^{T-1}(\delta_t^2 + \delta_{t-1}^2),
\end{aligned}
\end{equation}
and
\begin{align}\label{eq:x-perp2}
&~\textstyle \left(1-\rho - \frac{8\eta_x^2 L^2(1+ 6\kappa^2)}{(1-\rho)^3} \right) \sum_{t=0}^{T} \|\vX_\perp^t\|_F^2 \le  \|\vX_\perp^0\|_F^2  + \frac{3m\eta_x^2\kappa^2}{(1-\rho)^3}\sum_{t=1}^{T-1}(\delta_t^2 + \delta_{t-1}^2) \cr
& ~ \textstyle  \hspace{1cm}+ \frac{\eta_x^2}{(1-\rho)^2}\Big( \big\|\vV_\perp^{0}\big\|_F^2 + \frac{6m L^2\kappa^2}{1-\rho} \sum_{t=1}^{T-1} \|\vLam^t - \vLam^{t-1}\|_F^2 + \frac{2L^2(1+ 6\kappa^2)}{1-\rho} \sum_{t=1}^{T-1} \|\vX^{t}-\widetilde\vX^{t-1}\|_F^2\Big).
\end{align}
\end{lemma}

\begin{proof}
By $\vone^\top\vW = \vone^\top$ and $\vW\vone = \vone$, it holds $(\vI-\frac{1}{m}\vone\vone^\top)\vW = \vW - \frac{1}{m}\vone\vone^\top = (\vW - \frac{1}{m}\vone\vone^\top)(\vI - \frac{1}{m}\vone\vone^\top)$. Hence, $(\vI-\frac{1}{m}\vone\vone^\top)\vW \vV^{t-1} =(\vW - \frac{1}{m}\vone\vone^\top) \vV_\perp^{t-1}$, and thus  from \eqref{eq:alg-update-lam-v}, 
it follows that for any $a>0$,
\begin{align}\label{eq:v-perp1}
\|\vV_\perp^t\|_F^2 = & ~\left\|\textstyle (\vW - \frac{1}{m}\vone\vone^\top) \vV_\perp^{t-1} + (\vI-\frac{1}{m}\vone\vone^\top) \big(\nabla_\vx F(\vX^t, \vY^t) - \nabla_\vx F(\vX^{t-1}, \vY^{t-1})\big)\right\|_F^2\cr
\le & ~ \textstyle(1+a)\big\|(\vW - \frac{1}{m}\vone\vone^\top) \vV_\perp^{t-1}\big\|_F^2 + (1+\frac{1}{a})\big\|\nabla_\vx F(\vX^t, \vY^t) - \nabla_\vx F(\vX^{t-1}, \vY^{t-1})\big\|_F^2,
\end{align}
where the inequality is obtained by using the Young's inequality and the fact $\|\vI-\frac{1}{m}\vone\vone^\top\|_2\le 1$. Taking $a=\frac{1}{\rho}-1$ in \eqref{eq:v-perp1}, by $\|\vW - \frac{1}{m}\vone\vone^\top\|_2\le \rho$, and using the $L$-smoothness of each $f_i$, we have
\begin{align}\label{eq:v-perp2}
\|\vV_\perp^t\|_F^2 \le & ~ \textstyle \rho \big\|\vV_\perp^{t-1}\big\|_F^2 + \frac{1}{1-\rho}\big\|\nabla_\vx F(\vX^t, \vY^t) - \nabla_\vx F(\vX^{t-1}, \vY^{t-1})\big\|_F^2\cr
\le & ~ \textstyle \rho \big\|\vV_\perp^{t-1}\big\|_F^2 + \frac{L^2}{1-\rho}\big(\|\vX^t - \vX^{t-1}\|_F^2 + \|\vY^t - \vY^{t-1}\|_F^2 \big).
\end{align}
In addition, by \eqref{eq:lip-S2}, Proposition~\ref{prop:error-bd-vYt}, and the Young's inequality, it holds 
\begin{align}
&\,\|\vY^t - \vY^{t-1}\|_F^2 \cr
 = &\, \|\vY^t - S_\Phi(\vX^t, \vLam^t) + S_\Phi(\vX^t, \vLam^t) - S_\Phi(\vX^{t-1}, \vLam^{t-1}) + S_\Phi(\vX^{t-1}, \vLam^{t-1}) - \vY^{t-1}\|_F^2\cr
\le &\, 3\|\vY^t - S_\Phi(\vX^t, \vLam^t)\|_F^2 + 3\|S_\Phi(\vX^t, \vLam^t) - S_\Phi(\vX^{t-1}, \vLam^{t-1})\|_F^2 + 3\|S_\Phi(\vX^{t-1}, \vLam^{t-1}) - \vY^{t-1}\|_F^2 \cr
 \le &\, \textstyle \frac{3m\delta_t^2}{\mu^2} + \frac{3m\delta_{t-1}^2}{\mu^2} +  6\kappa^2\|\vX^t - \vX^{t-1}\|_F^2 + 6m\kappa^2 \|\vLam^t - \vLam^{t-1}\|_F^2.
\end{align}
 Hence, \eqref{eq:v-perp2} implies
\begin{align}\label{eq:v-perp3}
& ~\|\vV_\perp^t\|_F^2 \le  \textstyle \rho \big\|\vV_\perp^{t-1}\big\|_F^2 + \frac{L^2}{1-\rho}\left((1+ 6\kappa^2)\|\vX^t - \vX^{t-1}\|_F^2 + 6m\kappa^2\|\vLam^t - \vLam^{t-1}\|_F^2 + \frac{3m\delta_t^2}{\mu^2} + \frac{3m\delta_{t-1}^2}{\mu^2}\right)\\
\overset{\eqref{eq:x-perp}}\le & ~ \textstyle \rho \big\|\vV_\perp^{t-1}\big\|_F^2 + \frac{6m L^2\kappa^2}{1-\rho} \|\vLam^t - \vLam^{t-1}\|_F^2 + \frac{2L^2(1+ 6\kappa^2)}{1-\rho}\big(\|\vX^{t}-\widetilde\vX^{t-1}\|_F^2 + 4\|\vX_\perp^{t-1}\|_F^2\big)
\textstyle + \frac{3m\kappa^2}{1-\rho}(\delta_t^2 + \delta_{t-1}^2) \nonumber.
\end{align}

Summing up \eqref{eq:v-perp3} over $t=1$ to $T-1$, we obtain \eqref{eq:v-perp4}.
Now summing up the first inequality in \eqref{eq:x-perp} over $t=0$ to $T-1$ and using \eqref{eq:v-perp4} yield
\begin{align}\label{eq:x-perp1}
& ~\textstyle \|\vX_\perp^{T}\|_F^2 + (1-\rho) \sum_{t=0}^{T-1} \|\vX_\perp^t\|_F^2 \le \|\vX_\perp^0\|_F^2 + \frac{\eta_x^2}{1-\rho} \sum_{t=0}^{T-1}\|\vV_\perp^t\|_F^2\cr
\le & ~ \textstyle \|\vX_\perp^0\|_F^2 + \frac{\eta_x^2}{(1-\rho)^2}\Big( \big\|\vV_\perp^{0}\big\|_F^2 + \frac{6m L^2\kappa^2}{1-\rho} \sum_{t=1}^{T-1} \|\vLam^t - \vLam^{t-1}\|_F^2 \\
& ~\textstyle \hspace{1.5cm}+ \frac{2L^2(1+ 6\kappa^2)}{1-\rho} \sum_{t=1}^{T-1} \big(\|\vX^{t}-\widetilde\vX^{t-1}\|_F^2 + 4\|\vX_\perp^{t-1}\|_F^2\big) + \frac{3m\kappa^2}{1-\rho}\sum_{t=1}^{T-1}(\delta_t^2 + \delta_{t-1}^2) \Big),\nonumber
\end{align}
which apparently indicates \eqref{eq:x-perp2}. This completes the proof.
\end{proof}

By Theorem~\ref{thm:one-iter} and Lemma~\ref{lem:consensus}, we can easily show the following square-summable result.
\begin{theorem}
Let $\{(\vX^t, \vLam^t, \vV^t)\}$ be generated from Alg.~\ref{alg:dgdm}. Then for any positive integer $T$, it holds
\begin{align}\label{eq:phi-diff5}
& ~ \textstyle \alpha_x\sum_{t=0}^{T-1}\|\vX^{t+1}-\vX_\avg^{t}\|^2 + \tilde\alpha_x\sum_{t=0}^{T-1}\|\vX^{t+1} - \widetilde\vX^t\|_F^2 + \alpha_\lambda \sum_{t=0}^{T-1}\|\vLam^{t+1}-\vLam^{t}\|_F^2 \cr
\le & ~  \textstyle \phi(\vx_\avg^{0}, \vLam^{0}) - \phi(\vx_\avg^{T}, \vLam^{T}) + c \|\vX_\perp^0\|_F^2 + \Big(\frac{1}{2m L_P(1-\rho)}   +\frac{c\eta_x^2}{(1-\rho)^2} \Big)\|\vV_\perp^{0}\|_F^2 \\
& ~ \textstyle + \left(\frac{3\kappa^2}{2L_P(1-\rho)^2} + \frac{3c m\eta_x^2\kappa^2}{(1-\rho)^3}\right)\sum_{t=1}^{T-1}(\delta_t^2 + \delta_{t-1}^2)  +\sum_{t=0}^{T-1}\frac{2\delta_t^2}{\mu}  \nonumber,
\end{align}
where the constants $\alpha_x,\tilde\alpha_x, \alpha_\lambda$ and $c$ are defined as
\begin{subequations}\label{eq:def-c-alphas}
\begin{align}
&\textstyle \alpha_x = \frac{1}{2m}\big(\frac{1}{\eta_x}-2L_P - L(\kappa+1)\big), \quad \tilde\alpha_x = \frac{1}{2m\eta_x} - \frac{L^2(1+ 6\kappa^2)}{m L_P(1-\rho)^2} - \frac{2cL^2(1+6\kappa^2)\eta_x^2}{(1-\rho)^3}\\
& \textstyle \alpha_\lambda = \big(\frac{1}{\eta_\lambda}- \frac{L\kappa}{2} - \frac{L_P}{2}\big) - \frac{3L^2\kappa^2}{L_P(1-\rho)^2} - \frac{6cmL^2\kappa^2\eta_x^2}{(1-\rho)^3}, \quad c = \dfrac{\frac{\rho^2}{2m\eta_x}+ \frac{L(4\kappa+1)}{2m} +\frac{4L^2(1+ 6\kappa^2)}{m L_P(1-\rho)^2}}{1-\rho - \frac{8\eta_x^2 L^2(1+ 6\kappa^2)}{(1-\rho)^3}}\label{eq:def-little-c}.
\end{align}
\end{subequations}
\end{theorem}

\begin{proof}
Sum up \eqref{eq:phi-diff3} over $t=0$ to $T-1$ and use \eqref{eq:v-perp4}. We have
\begin{align}\label{eq:phi-diff4}
& ~ \textstyle\frac{1}{2m}\big(\frac{1}{\eta_x}-2L_P - L(\kappa+1)\big)\sum_{t=0}^{T-1}\|\vX^{t+1}-\vX_\avg^{t}\|^2 + \big(\frac{1}{\eta_\lambda}- \frac{L\kappa}{2} - \frac{L_P}{2}\big) \sum_{t=0}^{T-1}\|\vLam^{t+1}-\vLam^{t}\|_F^2 \cr
 \le & ~ \textstyle \phi(\vx_\avg^{0}, \vLam^{0}) - \phi(\vx_\avg^{T}, \vLam^{T}) + \big(\frac{\rho^2}{2m\eta_x}+ \frac{L(4\kappa+1)}{2m} \big) \sum_{t=0}^{T-1} \|\vX_\perp^{t}\|_F^2  -\frac{1}{2m\eta_x} \sum_{t=0}^{T-1}\|\vX^{t+1} - \widetilde\vX^t\|_F^2 \cr
 & ~ \textstyle   +\sum_{t=0}^{T-1}\frac{2\delta_t^2}{\mu} + \frac{1}{2m L_P(1-\rho)} \Big(\big\|\vV_\perp^{0}\big\|_F^2 + \frac{6m L^2\kappa^2}{1-\rho} \sum_{t=1}^{T-1} \|\vLam^t - \vLam^{t-1}\|_F^2  \cr
& ~ \textstyle \hspace{2cm}+ \frac{2L^2(1+ 6\kappa^2)}{1-\rho} \sum_{t=1}^{T-1} \big(\|\vX^{t}-\widetilde\vX^{t-1}\|_F^2 + 4\|\vX_\perp^{t-1}\|_F^2\big) + \frac{3m\kappa^2}{1-\rho}\sum_{t=1}^{T-1}(\delta_t^2 + \delta_{t-1}^2)\Big) \cr
\le & ~  \textstyle \phi(\vx_\avg^{0}, \vLam^{0}) - \phi(\vx_\avg^{T}, \vLam^{T}) + \frac{1}{2m L_P(1-\rho)} \big\|\vV_\perp^{0}\big\|_F^2 - \left(\frac{1}{2m\eta_x} - \frac{L^2(1+ 6\kappa^2)}{m L_P(1-\rho)^2}\right)\sum_{t=0}^{T-1}\|\vX^{t+1} - \widetilde\vX^t\|_F^2\\
&~ \textstyle + \frac{3L^2\kappa^2}{L_P(1-\rho)^2}\sum_{t=1}^{T-1} \|\vLam^t - \vLam^{t-1}\|_F^2 + \left(\frac{\rho^2}{2m\eta_x}+ \frac{L(4\kappa+1)}{2m} +\frac{4L^2(1+ 6\kappa^2)}{m L_P(1-\rho)^2}\right) \sum_{t=0}^{T-1} \|\vX_\perp^{t}\|_F^2 \cr
&~ \textstyle + \frac{3\kappa^2}{2L_P(1-\rho)^2}\sum_{t=1}^{T-1}(\delta_t^2 + \delta_{t-1}^2)  +\sum_{t=0}^{T-1}\frac{2\delta_t^2}{\mu} .  \nonumber
\end{align}
Now use \eqref{eq:x-perp2} to bound $\sum_{t=0}^{T-1} \|\vX_\perp^{t}\|_F^2$ and substitute it into \eqref{eq:phi-diff4}. We obtain
\begin{align*}
& ~ \textstyle\frac{1}{2m}\big(\frac{1}{\eta_x}-2L_P - L(\kappa+1)\big)\sum_{t=0}^{T-1}\|\vX^{t+1}-\vX_\avg^{t}\|^2 + \big(\frac{1}{\eta_\lambda}- \frac{L\kappa}{2} - \frac{L_P}{2}\big) \sum_{t=0}^{T-1}\|\vLam^{t+1}-\vLam^{t}\|_F^2 \cr
\le & ~  \textstyle \phi(\vx_\avg^{0}, \vLam^{0}) - \phi(\vx_\avg^{T}, \vLam^{T}) + \frac{1}{2m L_P(1-\rho)} \big\|\vV_\perp^{0}\big\|_F^2 - \left(\frac{1}{2m\eta_x} - \frac{L^2(1+ 6\kappa^2)}{m L_P(1-\rho)^2}\right)\sum_{t=0}^{T-1}\|\vX^{t+1} - \widetilde\vX^t\|_F^2\\
&~ \textstyle + \frac{3L^2\kappa^2}{L_P(1-\rho)^2}\sum_{t=1}^{T-1} \|\vLam^t - \vLam^{t-1}\|_F^2  +\sum_{t=0}^{T-1}\frac{2\delta_t^2}{\mu}  + \left(\frac{3\kappa^2}{2L_P(1-\rho)^2} + \frac{3c m\eta_x^2\kappa^2}{(1-\rho)^3}\right)\sum_{t=1}^{T-1}(\delta_t^2 + \delta_{t-1}^2) + c \|\vX_\perp^0\|_F^2\cr
&~ +  \textstyle \frac{c\eta_x^2}{(1-\rho)^2}\Big(\|\vV_\perp^{0}\|_F^2 + \frac{6m L^2\kappa^2}{1-\rho} \sum_{t=1}^{T-1} \|\vLam^t - \vLam^{t-1}\|_F^2 + \frac{2L^2(1+ 6\kappa^2)}{1-\rho} \sum_{t=1}^{T-1} \|\vX^{t}-\widetilde\vX^{t-1}\|_F^2\Big),  \nonumber
\end{align*}
where $c$ is defined in \eqref{eq:def-little-c}. Combining like terms in the inequality above gives the desired result.
\end{proof}

In the lemma below, we specify the stepsizes, which lead to positive coefficients on the left hand side of \eqref{eq:phi-diff5}. Its proof is about basic (but tedious) algebra, so we put it in the appendix for better readability. 

\begin{lemma}\label{lem:stepsize-choice}
Suppose the stepsize parameters $\eta_x$ and $\eta_\lambda$ are taken as
\begin{equation}\label{eq:cond-stepsize}
\textstyle \eta_x = \frac{(1-\rho)^2}{5L\sqrt{1+6\kappa^2}},\quad  \eta_\lambda = \frac{(1-\rho)^2}{L(9\kappa+2)} .
\end{equation}
Let $c,\alpha_x,\tilde\alpha_x$ and $\alpha_\lambda$ be defined in \eqref{eq:def-c-alphas}. Then it holds that
\begin{subequations}
\begin{align}
&\textstyle \alpha_x \ge  \frac{L(4\kappa+1)}{2m(1-\rho)^2},\ \tilde \alpha_x \ge  \frac{3L(1+4\kappa)}{25m (1-\rho)^2} ,\  \alpha_\lambda \ge  \frac{L(4\kappa+1)}{(1-\rho)^2} ,\label{eq:stepsize-ineq2}\\ 
&\textstyle \frac{3\kappa^2}{2L_P(1-\rho)^2} + \frac{3c m\eta_x^2\kappa^2}{(1-\rho)^3} \le \frac{3\kappa}{2L(1-\rho)^2}, \ \frac{1}{2m L_P(1-\rho)}   +\frac{c\eta_x^2}{(1-\rho)^2} \le \frac{1}{2m L \kappa (1-\rho)}. \label{eq:stepsize-ineq2-2}
\end{align}
\end{subequations}
\end{lemma}

\begin{remark}\label{rm:stepsize-set}
The stepsizes in \eqref{eq:cond-stepsize} are larger than those of the GDA method in \cite{lin2020gradient} and decentralized gradient-type methods in \cite{zhang2021taming, chen2022simple, gao2022decentralized}, by a factor of $\kappa$ or $\kappa^2$. This will lead our method to have better numerical performance and lower complexity, especially when $\kappa$ is big; see numerical results in Section~\ref{sec:numerical}.
\end{remark}

Below, we show a global convergence result when the error tolerance $\{\delta_t\}$ is square summable.
\begin{theorem}[Global convergence]
Under Assumptions~\ref{assump-func}--\ref{assump-W}, let $\{(\vX^t, \vLam^t)\}_{t\ge 0}$ be generated from Alg.~\ref{alg:dgdm} with $\eta_x$ and $\eta_\lambda$ set to those in \eqref{eq:cond-stepsize}. Suppose $\sum_{t=0}^\infty \delta_t^2 < +\infty$. Then any limit point $(\bar\vX, \bar\vLam)$ of the sequence $\{(\vX^t, \vLam^t)\}_{t\ge 0}$ satisfies: (i) $\bar\vX_\perp = \vzero$; (ii) $(\bar\vx_\avg, \bar\vLam)$ is a stationary solution of $\phi$.
\end{theorem}

\begin{proof}
When $\eta_x$ and $\eta_\lambda$ are set to those in \eqref{eq:cond-stepsize}, we have from \eqref{eq:stepsize-ineq2} that $\alpha_x, \tilde\alpha_x$ and $\alpha_\lambda$ are all positive. Hence, taking $T\to\infty$ in \eqref{eq:phi-diff5} and using the lower boundedness of $\phi$ gives
\begin{equation}\label{eq:convg-series}  \textstyle \sum_{t=0}^{\infty}\|\vX^{t+1}-\vX_\avg^{t}\|_F^2 + \sum_{t=0}^{\infty}\|\vX^{t+1} - \widetilde\vX^t\|_F^2 + \sum_{t=0}^{\infty}\|\vLam^{t+1}-\vLam^{t}\|_F^2 < \infty,
\end{equation}
which together with \eqref{eq:v-perp4} and \eqref{eq:x-perp2} implies $\sum_{t=0}^{\infty}\|\vV_\perp^t\|_F^2 < \infty$ and $\sum_{t=0}^{\infty}\|\vX_\perp^t\|_F^2 < \infty$. Hence, $\vX_\perp^t\to \vzero$ as $t\to\infty$. Therefore, the limit point must satisfy $\bar\vX_\perp = \vzero$. 

In addition, we have $-\frac{L}{2\sqrt{m}}(\vW-\vI)\vY^t=\vLam^{t+1}-\vLam^{t}\to \vzero$ as $t\to\infty$ from \eqref{eq:convg-series}. 
From \eqref{eq:rel-y-x}, it holds 
\begin{equation}\label{eq:limt-y-hat-y}
\textstyle \|\widehat\vY^t - \vY^t\|_F \le \sqrt{2}\kappa \|\vX_\perp^t\|_F + \frac{\sqrt{2m}\delta_t}{\mu} \to 0.
\end{equation} Hence, 
\begin{equation}\label{eq:grad-lam-to-0}
\textstyle \nabla_\vLam P(\vx_\avg^t, \vLam^t) = -\frac{L}{2\sqrt{m}}(\vW-\vI)\widehat\vY^t = -\frac{L}{2\sqrt{m}}(\vW-\vI)\vY^t + \frac{L}{2\sqrt{m}}(\vW-\vI)(\vY^t-\widehat\vY^t) \to \vzero, \text{ as }t\to\infty.
\end{equation}
Also, it follows from \eqref{eq:convg-series} that $\vX^{t+1}-\vX_\avg^{t}\to\vzero$ and $\vX^{t+1} - \widetilde\vX^t\to\vzero$, as $t\to\infty$. Thus by $\vX_\perp^t\to \vzero$,  we have $\vx_i^{t+1} - \vx_\avg^t \to\vzero$ and $\tilde\vx_i^t - \vx_\avg^t \to\vzero$ for any $i\in [m]$. Therefore, it holds from \eqref{eq:alg-update-x} that
\begin{equation}\label{eq:limit-update-x}
\vx_\avg^t - \prox_{\eta_x g}\big(\vx_\avg^t - \eta_x \vv_i^t\big)\to \vzero, \text{ as }t\to\infty, \forall\, i \in [m].
\end{equation}

Moreover, $\vV_\perp^t\to\vzero$ and thus $\vv_i^t - \vv_\avg^t\to\vzero$ as $t\to\infty$. Now notice $\vv_\avg^t = \frac{1}{m}\sum_{i=1}^m \nabla_\vx f_i(\vx_i^t, \vy_i^t)$ and
\begin{align}\label{eq:bd-sum-y-grad-t}
& \, \textstyle \|\frac{1}{m}\sum_{i=1}^m \nabla_\vx f_i(\vx_i^t, \vy_i^t) - \frac{1}{m}\sum_{i=1}^m \nabla_\vx f_i(\vx_\avg^t, \hat\vy_i^t)\|^2 \le \frac{L^2}{m}\sum_{i=1}^m\big(\|\vx_i^t -\vx_\avg^t\|^2 + \|\vy_i^t - \hat\vy_i^t\|^2\big)\\
= &\, \textstyle \frac{L^2}{m}\big(\|\vX_\perp^t\|_F^2 + \|\vY^t -\widehat\vY^t\|_F^2\big) \overset{\eqref{eq:limt-y-hat-y}} {\longrightarrow} 0, \text{ as }t\to\infty. \nonumber
\end{align}
Hence, $\vv_\avg^t - \nabla_\vx P(\vx_\avg^t, \vLam^t) = \vv_\avg^t - \frac{1}{m}\sum_{i=1}^m \nabla_\vx f_i(\vx_\avg^t, \hat\vy_i^t) \to \vzero$, and thus $\vv_i^t - \nabla_\vx P(\vx_\avg^t, \vLam^t) \to \vzero$ for any $i$. Therefore, \eqref{eq:limit-update-x} indicates
$\vx_\avg^t - \prox_{\eta_x g}\big(\vx_\avg^t - \eta_x \nabla_\vx P(\vx_\avg^t, \vLam^t)\big)\to \vzero$, which together with \eqref{eq:grad-lam-to-0} implies the stationarity of $(\bar\vx_\avg, \bar\vLam)$. This completes the proof.
\end{proof}

In the follows, we bound the finite square-sum of each sequence. 
\begin{lemma}\label{lem:bd-sum-x-lam}
Let $\{(\vX^t, \vLam^t, \vV^t)\}$ be generated from Alg.~\ref{alg:dgdm} with $\eta_x$ and $\eta_\lambda$ set to those in \eqref{eq:cond-stepsize}. Then 
\begin{subequations}\label{eq:bd-sum-x-lam}
\begin{align}
&\textstyle \sum_{t=0}^{T-1}\|\vX^{t+1} - \vX_\avg^t\|_F^2 \le \frac{2mC_{0,T}(1-\rho)^2}{L(4\kappa+1)} \label{eq:bd-sum-x-lam-1}\\
&\textstyle \sum_{t=0}^{T-1}\|\vX^{t+1} - \widetilde\vX^t\|_F^2 \le  \frac{25 mC_{0,T} (1-\rho)^2}{3L(1+4\kappa)} ,\label{eq:bd-sum-x-lam-2}\\
&\textstyle \sum_{t=0}^{T-1}\|\vLam^{t+1}-\vLam^{t}\|_F^2 \le  \frac{C_{0,T}(1-\rho)^2}{L(4\kappa+1)} , \label{eq:bd-sum-x-lam-3}\\
& \textstyle \sum_{t=0}^{T-1} \|\vX_\perp^t\|_F^2 \le \frac{3(1-\rho)}{50 L^2(1+6\kappa^2)} \big\|\vV_\perp^{0}\big\|_F^2 +  \frac{53m C_{0,T} (1-\rho)^2}{50L(4\kappa+1)} + \frac{3m}{50 L^2} \Delta_T, \label{eq:x-perp2-1}\\
& \textstyle \sum_{t=0}^{T-1}\|\vV_\perp^t\|_F^2 \le \frac{3}{2(1-\rho)}\big\|\vV_\perp^{0}\big\|_F^2  + \frac{m(1+9\kappa^2)}{(1-\rho)^2} \Delta_T +  40mC_{0,T} L\kappa , \label{eq:v-perp4-1}
\end{align}
\end{subequations}
where $\Delta_T=\sum_{t=0}^{T-1} \delta_t^2$ and $C_{0,T} =    \textstyle \phi(\vx_\avg^{0}, \vLam^{0}) -  (p+g)(\vx^*) +  \frac{1}{2m L \kappa (1-\rho)}\|\vV_\perp^{0}\|_F^2 + \frac{5\kappa}{L(1-\rho)^2} \Delta_T$. 
\end{lemma}

\begin{proof}
Notice that $\min_\vLam P(\vx, \vLam) = p(\vx)$. Hence, $\phi(\vx, \vLam) \ge (p+g)(\vx^*)$, and thus from \eqref{eq:phi-diff5}, the two results in \eqref{eq:stepsize-ineq2-2}, the setting $\vx_i^0 = \vx^0,\forall\, i\in[m]$, and $\frac{1}{\mu} = \frac{\kappa}{L}$, it follows
\begin{align}\label{eq:phi-diff5-1}
& ~ \textstyle \alpha_x\sum_{t=0}^{T-1}\|\vX^{t+1}-\vX_\avg^{t}\|^2 + \tilde\alpha_x\sum_{t=0}^{T-1}\|\vX^{t+1} - \widetilde\vX^t\|_F^2 + \alpha_\lambda \sum_{t=0}^{T-1}\|\vLam^{t+1}-\vLam^{t}\|_F^2 \le C_{0,T}.
\end{align}
By the bounds of $\alpha_x$, $\tilde \alpha_x$ and  $\alpha_\lambda$ in \eqref{eq:stepsize-ineq2}, we immediately have \eqref{eq:bd-sum-x-lam-1} -- \eqref{eq:bd-sum-x-lam-3} from the inequality above.   

Using the three inequalities in \eqref{eq:stepsize-ineq1}, we have from \eqref{eq:x-perp2} that 
$$\textstyle \sum_{t=0}^{T} \|\vX_\perp^t\|_F^2 
\le  \textstyle  \frac{3\eta_x^2}{2(1-\rho)^3}\big\|\vV_\perp^{0}\big\|_F^2  + \frac{9m\eta_x^2\kappa^2}{2(1-\rho)^4}\sum_{t=1}^{T-1}(\delta_t^2 + \delta_{t-1}^2) 
 \textstyle + \frac{3m}{50} \sum_{t=1}^{T-1} \|\vLam^t - \vLam^{t-1}\|_F^2 + \frac{3}{25} \sum_{t=1}^{T-1} \|\vX^{t}-\widetilde\vX^{t-1}\|_F^2.$$
Plugging \eqref{eq:bd-sum-x-lam-2} and \eqref{eq:bd-sum-x-lam-3} into the inequality above gives 
\begin{align*}
\textstyle \sum_{t=0}^{T} \|\vX_\perp^t\|_F^2 
\le &~  \textstyle  \frac{3\eta_x^2}{2(1-\rho)^3}\big\|\vV_\perp^{0}\big\|_F^2 + \frac{9m\eta_x^2\kappa^2}{2(1-\rho)^4}\sum_{t=1}^{T-1}(\delta_t^2 + \delta_{t-1}^2) + \frac{3mC_{0,T}(1-\rho)^2}{50L(4\kappa+1)} + \frac{mC_{0,T} (1-\rho)^2}{L(1+4\kappa)},
\end{align*}
which implies \eqref{eq:x-perp2-1} by the choice of $\eta_x$.
Hence, by \eqref{eq:bd-sum-x-lam-2} -- \eqref{eq:x-perp2-1}, we have from \eqref{eq:v-perp4} that
\begin{align*}
 \textstyle \sum_{t=0}^{T-1}\|\vV_\perp^t\|_F^2 \le & \, \textstyle \frac{1}{1-\rho} \Big( \textstyle \big\|\vV_\perp^{0}\big\|_F^2 + \frac{6m L^2\kappa^2}{1-\rho} \sum_{t=1}^{T-1} \|\vLam^t - \vLam^{t-1}\|_F^2\\
 & ~ \textstyle  \hspace{1cm}+ \frac{2L^2(1+ 6\kappa^2)}{1-\rho} \sum_{t=1}^{T-1} \big(\|\vX^{t}-\widetilde\vX^{t-1}\|_F^2 + 4\|\vX_\perp^{t-1}\|_F^2\big) + \frac{6m\kappa^2}{1-\rho} \Delta_T\Big)\\
 \le &\, \textstyle \frac{1}{1-\rho} \big\|\vV_\perp^{0}\big\|_F^2 + \frac{6m L^2\kappa^2}{(1-\rho)^2} \frac{C_{0,T}(1-\rho)^2}{L(4\kappa+1)} + \frac{2L^2(1+ 6\kappa^2)}{(1-\rho)^2} \frac{25 mC_{0,T} (1-\rho)^2}{3L(1+4\kappa)} + \frac{6m\kappa^2}{(1-\rho)^2} \Delta_T \\
 &\, \textstyle + \frac{8L^2(1+ 6\kappa^2)}{(1-\rho)^2}\left(\frac{3(1-\rho)}{50 L^2(1+6\kappa^2)} \big\|\vV_\perp^{0}\big\|_F^2 + \frac{53m C_{0,T} (1-\rho)^2}{50L(4\kappa+1)} + \frac{3m}{50 L^2} \Delta_T\right)\\
 \le &\, \textstyle \frac{3}{2(1-\rho)}\big\|\vV_\perp^{0}\big\|_F^2  + \frac{m(1+9\kappa^2)}{(1-\rho)^2} \Delta_T + \frac{6m C_{0,T} L\kappa^2}{4\kappa+1} + \frac{50mC_{0,T}L(1+ 6\kappa^2)}{3(1+4\kappa)} + \frac{9m C_{0,T} L(1+6\kappa^2)}{1+4\kappa}, 
\end{align*}
which implies \eqref{eq:v-perp4-1}.  Thus we complete the proof.
\end{proof}

Now we are ready to show the convergence rate result based on the stationarity violation.
\begin{theorem}
Under Assumptions~\ref{assump-func}--\ref{assump-W}, let $\{(\vX^t, \vLam^t)\}_{t= 0}^T$ be generated from Alg.~\ref{alg:dgdm} with $\eta_x$ and $\eta_\lambda$ set to those in \eqref{eq:cond-stepsize}, where $T$ is a positive integer. Choose $\tau$ from $\{0,\ldots, T-1\}$ uniformly at random. Then
\begin{subequations}\label{eq:opt-cond-bd-tau}
\begin{align}
&\EE_\tau \|\nabla_\vLam P(\vx_\avg^\tau, \vLam^\tau)\|_F \le \textstyle \frac{7\sqrt{L \kappa C_{0,T} }}{(1-\rho)\sqrt{T}} + \frac{\sqrt{2(1-\rho)}}{10\sqrt{mT}}\big\|\vV_\perp^{0}\big\|_F + \frac{2\kappa}{\sqrt{T}}\sqrt{\Delta_T}, \label{eq:opt-cond-bd-tau1}\\
&\textstyle \frac{1}{\eta_x}\EE_\tau \big\|\vx_\avg^\tau - \prox_{\eta_x g}\big(\vx_\avg^\tau - \eta_x \nabla_\vx P(\vx_\avg^\tau, \vLam^\tau)\big)\big\| \le \frac{2}{\sqrt{T}}\sqrt{ C_{1,T} +  \frac{156 C_{0,T} L \kappa}{(1-\rho)^2} }, \label{eq:opt-cond-bd-tau2} 
\end{align}
\end{subequations}
where $C_{0,T}$ is defined in Lemma~\ref{lem:bd-sum-x-lam}, and 
\begin{equation}\label{eq:C1}
C_{1,T} = \textstyle \frac{1}{m} \left( \frac{3}{2(1-\rho)} + \frac{3(1-\rho)(1+2\kappa^2)}{50(1+6\kappa^2)} + \frac{3}{2(1-\rho)^2}\right)\big\|\vV_\perp^{0}\big\|_F^2 + \left( \frac{3(1+2\kappa^2)}{50}   + \frac{3(1+6\kappa^2)}{2(1-\rho)^4}  + \frac{1+9\kappa^2}{(1-\rho)^2} + 2\kappa^2\right) \Delta_T.
\end{equation}
\end{theorem}
\begin{proof}
By the definition of $\widehat\vY^t$ in the proof of Lemma~\ref{lem:one-step} and \eqref{eq:grad-P}, we use the triangle inequality to have
\begin{align}\label{eq:bd-grad-lam-tau0}
\|\nabla_\vLam P(\vx_\avg^\tau, \vLam^\tau)\|_F \le &\, \textstyle \frac{L}{2\sqrt{m}}\|(\vW-\vI)\vY^\tau\|_F + \frac{L}{2\sqrt{m}}\|(\vW-\vI)(\vY^\tau-\widehat\vY^\tau)\|_F\nonumber\\
= &\, \textstyle \frac{1}{\eta_\lambda}\|\vLam^{\tau+1} - \vLam^\tau\|_F + \frac{L}{2\sqrt{m}}\|(\vW-\vI)(\vY^\tau-\widehat\vY^\tau)\|_F.
\end{align} 
In addition, it follows from \eqref{eq:rel-y-x} and $\|\vW-\vI\|_2 \le 2$ that 
$ \|(\vW-\vI)(\vY^\tau-\widehat\vY^\tau)\|_F \le 2\sqrt{2} \kappa\|\vX_\perp^\tau\|_F + \frac{2\sqrt{2m}\delta_\tau}{\mu}.$ From the selection of $\tau$ and \eqref{eq:bd-sum-x-lam-3}, it holds $\EE_\tau \|\vLam^{\tau+1} - \vLam^\tau\|_F \le \sqrt{\EE_\tau \big[\|\vLam^{\tau+1} - \vLam^\tau\|_F^2\big]} \le \sqrt{\frac{C_{0,T}(1-\rho)^2}{T L(4\kappa+1)}}$, and similarly from \eqref{eq:x-perp2-1}, it holds $\EE_\tau \|\vX_\perp^\tau\|_F \le \frac{1}{\sqrt{T}}\sqrt{ \frac{3(1-\rho)}{50 L^2(1+6\kappa^2)} \big\|\vV_\perp^{0}\big\|_F^2 + \frac{53m C_{0,T} (1-\rho)^2}{50L(4\kappa+1)} + \frac{3m}{50 L^2} \Delta_T }$.
Moreover, $\EE_\tau[\delta_\tau] \le \sqrt{\EE_\tau[\delta_\tau^2]} = \frac{1}{\sqrt{T}}\sqrt{\Delta_T}$.
Hence, with $\eta_\lambda$ given in \eqref{eq:cond-stepsize}, the inequality in \eqref{eq:bd-grad-lam-tau0} indicates 
\begin{align}\label{eq:bd-grad-lam-tau}
&\,\EE_\tau \|\nabla_\vLam P(\vx_\avg^\tau, \vLam^\tau)\|_F \nonumber \\
 \le &\, \textstyle \frac{L(9\kappa+2)}{(1-\rho)^2} \sqrt{\frac{C_{0,T}(1-\rho)^2}{T L(4\kappa+1)}} + \frac{\sqrt{2}L\kappa}{\sqrt{mT}} \sqrt{ \frac{3(1-\rho)}{50 L^2(1+6\kappa^2)} \big\|\vV_\perp^{0}\big\|_F^2 + \frac{53m C_{0,T} (1-\rho)^2}{50L(4\kappa+1)} + \frac{3m}{50 L^2} \Delta_T } + \frac{\sqrt{2}\kappa}{\sqrt{T}}\sqrt{\Delta_T} \nonumber \\
\le &\, \textstyle \frac{9\sqrt{LC_{0,T}(4\kappa+1)}}{4(1-\rho)\sqrt{T}} + \frac{\sqrt{2(1-\rho)}}{10\sqrt{mT}}\big\|\vV_\perp^{0}\big\|_F +  \frac{(1-\rho)\sqrt{L\kappa}}{\sqrt{T}} \sqrt{C_{0,T} } + \frac{\kappa\sqrt{3\Delta_T}}{5\sqrt{T}} + \frac{\sqrt{2}\kappa}{\sqrt{T}}\sqrt{\Delta_T},
\end{align}
where the last inequality holds by $\frac{9\kappa+2}{\sqrt{4\kappa+1}} \le \frac{9}{4}\sqrt{4\kappa+1}$, $\sqrt{a+b} \le \sqrt{a} + \sqrt{b}, \forall\, a, b\in \RR_+$, $\frac{\kappa}{\sqrt{1+6\kappa^2}} \le \frac{1}{\sqrt{6}}$, and $\frac{\sqrt{53}\kappa}{5\sqrt{4\kappa+1}}\le \sqrt\kappa$. 
Now we obtain \eqref{eq:opt-cond-bd-tau1} from \eqref{eq:bd-grad-lam-tau} by noticing $\frac{9}{4}\sqrt{4\kappa+1}+\sqrt{\kappa} \le 7\sqrt{\kappa}$ and $\frac{\sqrt{3}}{5} + \sqrt{2} \le 2$.

Furthermore, by \eqref{eq:grad-P}, it holds 
$$\left\|\vx_\avg^\tau - \prox_{\eta_x g}\big(\vx_\avg^\tau - \eta_x \nabla_\vx P(\vx_\avg^\tau, \vLam^\tau)\big)\right\| =  \left\|\textstyle \vx_\avg^\tau - \prox_{\eta_x g}\big(\vx_\avg^\tau -  \frac{\eta_x}{m}\sum_{i=1}^m \nabla_\vx f_i(\vx_\avg^{\tau}, \hat\vy_i^\tau)\big)\right\|.$$
Hence, by the triangle inequality and the nonexpansiveness of $\prox_{\eta_x g}$, we have
\begin{align*}
&\,\left\|\vx_\avg^\tau - \prox_{\eta_x g}\big(\vx_\avg^\tau - \eta_x \nabla_\vx P(\vx_\avg^\tau, \vLam^\tau)\big)\right\| \nonumber \\
\le &\,  \left\|\textstyle \vx_\avg^\tau - \prox_{\eta_x g}\big(\tilde\vx_i^\tau - \eta_x\vv_i^\tau\big)\right\| + \left\| \textstyle \vx_\avg^\tau - \frac{\eta_x}{m}\sum_{i=1}^m \nabla_\vx f_i(\vx_\avg^{\tau}, \hat\vy_i^\tau) - \big(\tilde\vx_i^\tau - \eta_x\vv_i^\tau\big)\right\| \nonumber \\
\le &\, \left\|\textstyle \vx_\avg^\tau - \prox_{\eta_x g}\big(\tilde\vx_i^\tau - \eta_x\vv_i^\tau\big)\right\| + \eta_x \left\| \textstyle  \frac{1}{m}\sum_{i=1}^m \nabla_\vx f_i(\vx_i^{\tau}, \vy_i^\tau) - \frac{1}{m}\sum_{i=1}^m \nabla_\vx f_i(\vx_\avg^{\tau}, \hat\vy_i^\tau)\right\| \\
&\, + \big\|\vx_\avg^\tau - \tilde\vx_i^\tau\big\|  + \eta_x\left\| \textstyle \frac{1}{m}\sum_{i=1}^m \nabla_\vx f_i(\vx_i^{\tau}, \vy_i^\tau) - \vv_i^\tau\right\|. \nonumber
\end{align*}
Taking square of both sides of the inequality above and using the Young's inequality gives
\begin{align}\label{eq:bd-grad-Px}
&\,\left\|\vx_\avg^\tau - \prox_{\eta_x g}\big(\vx_\avg^\tau - \eta_x \nabla_\vx P(\vx_\avg^\tau, \vLam^\tau)\big)\right\|^2\nonumber \\
\le &\, 4\left\|\textstyle \vx_\avg^\tau - \prox_{\eta_x g}\big(\tilde\vx_i^\tau - \eta_x\vv_i^\tau\big)\right\|^2 + 4\eta_x^2 \left\| \textstyle  \frac{1}{m}\sum_{i=1}^m \nabla_\vx f_i(\vx_i^{\tau}, \vy_i^\tau) - \frac{1}{m}\sum_{i=1}^m \nabla_\vx f_i(\vx_\avg^{\tau}, \hat\vy_i^\tau)\right\|^2 \\
&\, + 4\big\|\vx_\avg^\tau - \tilde\vx_i^\tau\big\|^2  + 4\eta_x^2\left\| \textstyle \frac{1}{m}\sum_{i=1}^m \nabla_\vx f_i(\vx_i^{\tau}, \vy_i^\tau) - \vv_i^\tau\right\|^2. \nonumber
\end{align}
Substituting \eqref{eq:bd-sum-y-grad-t} with $t=\tau$ into \eqref{eq:bd-grad-Px} and summing it up over $i=1,\ldots,m$ give
\begin{align}\label{eq:bd-grad-Px2}
&\, m\left\|\vx_\avg^\tau - \prox_{\eta_x g}\big(\vx_\avg^\tau - \eta_x \nabla_\vx P(\vx_\avg^\tau, \vLam^\tau)\big)\right\|^2\nonumber \\
\le &\, 4\left\|\textstyle \vX_\avg^\tau - \vX^{\tau + 1}\right\|_F^2 +  4\eta_x^2 L^2\big(\|\vX^{\tau}_\perp \|_F^2 + \|\vY^\tau - \widehat\vY^\tau\|_F^2\big) + 4\big\|\vX_\avg^\tau - \widetilde\vX^\tau\big\|_F^2  + 4\eta_x^2\left\| \vV_\perp^\tau\right\|_F^2, 
\end{align}
where we have used \eqref{eq:alg-update-x} and the fact $\vv_\avg^t = \frac{1}{m}\sum_{i=1}^m \nabla_\vx f_i(\vx_i^{t}, \vy_i^t)$ for any $t$. Use \eqref{eq:rel-y-x} and notice $\big\|\vX_\avg^\tau - \widetilde\vX^\tau\big\|_F^2 \le \|\vX_\perp^\tau\|_F^2$. We have from \eqref{eq:bd-grad-Px2} that
\begin{align}\label{eq:bd-grad-Px3}
&\, m\left\|\vx_\avg^\tau - \prox_{\eta_x g}\big(\vx_\avg^\tau - \eta_x \nabla_\vx P(\vx_\avg^\tau, \vLam^\tau)\big)\right\|^2\nonumber \\
\le &\, 4\left\|\textstyle \vX_\avg^\tau - \vX^{\tau + 1}\right\|_F^2 +  4\eta_x^2 L^2(1+2\kappa^2)\big\|\vX_\perp^\tau \big\|_F^2 + 8m \eta_x^2 \kappa^2 \delta_\tau^2 + 4\big\|\vX_\perp^\tau \big\|_F^2  + 4\eta_x^2\left\| \vV_\perp^\tau\right\|_F^2. 
\end{align}
Now taking expectation about $\tau$ on both sides of \eqref{eq:bd-grad-Px3} and using \eqref{eq:bd-sum-x-lam-1}, \eqref{eq:x-perp2-1}, and \eqref{eq:v-perp4-1}, we have
\begin{align*}
&\, m \EE_\tau \left[ \big\|\vx_\avg^\tau - \prox_{\eta_x g}\big(\vx_\avg^\tau - \eta_x \nabla_\vx P(\vx_\avg^\tau, \vLam^\tau)\big)\big\|^2\right]\nonumber \\
\le &\, \frac{4}{T} \bigg[ \textstyle \frac{2mC_{0,T}(1-\rho)^2}{L(4\kappa+1)} +  \big( \eta_x^2 L^2(1+2\kappa^2) + 1 \big) \left(\frac{3(1-\rho)}{50 L^2(1+6\kappa^2)} \big\|\vV_\perp^{0}\big\|_F^2 + \frac{53m C_{0,T} (1-\rho)^2}{50L(4\kappa+1)} + \frac{3m}{50 L^2} \Delta_T\right) \\
 &\, \hspace{1cm} + \eta_x^2 \left( \textstyle \frac{3}{2(1-\rho)}\big\|\vV_\perp^{0}\big\|_F^2  + \frac{m(1+9\kappa^2)}{(1-\rho)^2} \Delta_T + 40mC_{0,T} L\kappa \right) + 2m \eta_x^2 \kappa^2 \Delta_T \bigg]. \nonumber
\end{align*}
Hence, dividing by $m \eta_x^2$ both sides of the inequality above yields
\begin{align}\label{eq:bd-grad-Px4}
&\, \textstyle \frac{1}{\eta_x^2}\EE_\tau \left[ \big\|\vx_\avg^\tau - \prox_{\eta_x g}\big(\vx_\avg^\tau - \eta_x \nabla_\vx P(\vx_\avg^\tau, \vLam^\tau)\big)\big\|^2\right]\nonumber \\
\le  &\, \textstyle \frac{4}{\eta_x^2 mT} \bigg[ \textstyle \frac{2mC_{0,T}(1-\rho)^2}{L(4\kappa+1)} +  \big( \eta_x^2 L^2(1+2\kappa^2) + 1 \big) \left(\frac{3(1-\rho)}{50 L^2(1+6\kappa^2)} \big\|\vV_\perp^{0}\big\|_F^2 + \frac{53m C_{0,T} (1-\rho)^2}{50L(4\kappa+1)} + \frac{3m}{50 L^2} \Delta_T\right) \nonumber\\
 &\,  \hspace{1.5cm}+ \eta_x^2 \left( \textstyle \frac{3}{2(1-\rho)}\big\|\vV_\perp^{0}\big\|_F^2  + \frac{m(1+9\kappa^2)}{(1-\rho)^2} \Delta_T +  40mC_{0,T} L\kappa  \right) + 2m \eta_x^2 \kappa^2 \Delta_T \bigg]\nonumber \\
 = &\, \textstyle \frac{4}{T}\Big[\textstyle C_{1,T} + 40C_{0,T} L\kappa +  \frac{ 53C_{0,T} L(1-\rho)^2(1+2\kappa^2)}{50(4\kappa+1)} + \frac{50 L C_{0,T}(1+6\kappa^2)}{(4\kappa +1)(1-\rho)^2} +  \frac{53L C_{0,T} (1+6\kappa^2)}{2(1-\rho)^2(4\kappa+1)} \Big] \nonumber\\
 \le &\, \textstyle \frac{4}{T}\Big[\textstyle C_{1,T} + \frac{156 C_{0,T} L \kappa}{(1-\rho)^2} \Big],
\end{align}
where the last inequality follows from $(1-\rho)^2 \le 1$ and $\frac{53(1+2\kappa^2)}{50(4\kappa+1)} + \frac{50(1+6\kappa^2)}{4\kappa+1}+\frac{53(1+6\kappa^2)}{2(4\kappa+1)} \le 116 \kappa$. Now we obtain \eqref{eq:opt-cond-bd-tau2} from \eqref{eq:bd-grad-Px4} by using the Jensen's inequality and complete the proof.
\end{proof}

\subsection{Complexity results}
In this subsection, we establish the complexity result of Algorithm~\ref{alg:dgdm} to produce a near-stationary point of \eqref{eq:min-max-Phi}.
\begin{definition}\label{def:opt-cond-phi-eps}
Given $\vareps>0$, a point $(\vX, \vLam)$ is an $\vareps$-stationary point of \eqref{eq:min-max-Phi} if for a certain $\eta>0$,
\begin{equation}\label{eq:opt-cond-phi-eps}
\textstyle \frac{1}{\eta}\left \| \vx_\avg - \prox_{\eta g}\big(\vx_\avg - \eta\nabla_\vx P(\vx_\avg,\vLam)\big)\right\| \le \vareps, \quad \frac{L}{\sqrt{m}} \|\vX_\perp\|_F \le \vareps,\quad \|\nabla_\vLam P(\vx_\avg,\vLam)\|_F \le \vareps.
\end{equation}
\end{definition}

\begin{theorem}\label{thm:up-bd-T}
Let $\vareps>0$ be given. Suppose  $\delta_t \le \frac{(1-\rho)^2}{8\kappa(1+t)}, \forall\, t \ge 0$ and $\|\vV^0_\perp\|_F^2 \le 2m L\kappa(1-\rho)$. Set
\begin{equation}\label{eq:upper-bd-T}
 T = \left\lceil \frac{1}{\vareps^2}\max\left\{\frac{64\big(10 L \kappa (\phi_0+1) + 1\big)}{(1-\rho)^2}, \ \frac{4096 L\kappa}{1-\rho},\ 800 \right\} \right\rceil, 
\end{equation}
where $\phi_0 = \textstyle \phi(\vx_\avg^{0}, \vLam^{0}) -  (p+g)(\vx^*)$. Let $\{(\vX^t, \vLam^t)\}_{t= 0}^T$ be generated from Alg.~\ref{alg:dgdm} with $\eta_x$ and $\eta_\lambda$ set to those in \eqref{eq:cond-stepsize}. Choose $\tau$ from $\{0,1,\ldots, T-1\}$ uniformly at random. Then under Assumptions~\ref{assump-func}--\ref{assump-W}, $(\vX^\tau, \vLam^\tau)$ is an $\vareps$-stationary point of \eqref{eq:min-max-Phi} in expectation.
\end{theorem}

\begin{proof}
From the definition of $\Delta_T$ in Lemma~\ref{lem:bd-sum-x-lam}, it follows that $\Delta_T \le \sum_{t=0}^\infty \frac{(1-\rho)^4}{64\kappa^2(1+t)^2} \le \frac{(1-\rho)^4}{32\kappa^2}$. Hence, it holds $C_{0,T} \le \phi_0 + 1 +  \frac{1}{6 L\kappa} $ and $C_{1,T} \le \frac{8 L\kappa}{1-\rho} + 1$ by the assumption on $\|\vV^0_\perp\|_F^2$, the definitions of $C_{0,T}$ and $C_{1,T}$ in Lemma~\ref{lem:bd-sum-x-lam} and Eqn. \eqref{eq:C1}, and the definition of $\phi_0$. With these, we have from \eqref{eq:opt-cond-bd-tau} that
\begin{subequations}\label{eq:opt-cond-bd-tau-2}
\begin{align}
&\EE_\tau \|\nabla_\vLam P(\vx_\avg^\tau, \vLam^\tau)\|_F \le  \textstyle \frac{7\sqrt{L \kappa (\phi_{0} + 1) + \frac{1}{6} }}{(1-\rho)\sqrt{T}} + \frac{(1-\rho)\sqrt{L\kappa}}{5\sqrt{T}} + \frac{(1-\rho)^2}{\sqrt{8T}}, \label{eq:opt-cond-bd-tau1-2}\\[0.1cm]
&\textstyle \frac{1}{\eta_x}\EE_\tau \big\|\vx_\avg^\tau - \prox_{\eta_x g}\big(\vx_\avg^\tau - \eta_x \nabla_\vx P(\vx_\avg^\tau, \vLam^\tau)\big)\big\| \le  \textstyle \frac{2}{\sqrt{T}}\sqrt{ \frac{8 L\kappa}{1-\rho} + 1 + \frac{156 L \kappa(\phi_0+1)+26}{(1-\rho)^2}} . \label{eq:opt-cond-bd-tau2-2} 
\end{align}
In addition, \eqref{eq:x-perp2-1} indicates $\sum_{t=0}^{T-1} \|\vX_\perp^t\|_F^2 \le \frac{3m\kappa(1-\rho)^2}{25 L(1+6\kappa^2)} + \frac{53m (1-\rho)^2 (\phi_0 + 1 + \frac{1}{6 L\kappa})}{50L(4\kappa+1)} + \frac{3m}{50 L^2} \frac{(1-\rho)^4}{32\kappa^2}$, and thus
\begin{equation}
\textstyle \EE_\tau \|\vX^\tau_\perp\|_F \le \sqrt{\frac{m}{L T}} \left(\frac{\sqrt{3\kappa}(1-\rho)}{5\sqrt{1+6\kappa^2}}+ \frac{\sqrt{53(\phi_0 + 1 + \frac{1}{6 L\kappa})}(1-\rho)}{5\sqrt{2+8\kappa}} + \frac{\sqrt{3}(1-\rho)^2}{40\kappa\sqrt{L}}\right). \label{eq:x-perp2-1-2}
\end{equation}
\end{subequations}
Now by the choice of $T$, it is straightforward to verify from \eqref{eq:opt-cond-bd-tau1-2} -- \eqref{eq:x-perp2-1-2} that 
$$\textstyle \frac{1}{\eta_x}\EE_\tau \big\|\vx_\avg^\tau - \prox_{\eta_x g}\big(\vx_\avg^\tau - \eta_x \nabla_\vx P(\vx_\avg^\tau, \vLam^\tau)\big)\big\| \le\vareps, \ \EE_\tau \|\nabla_\vLam P(\vx_\avg^\tau, \vLam^\tau)\|_F \le \vareps, \ \frac{L}{\sqrt{m}}\EE_\tau \|\vX^\tau_\perp\|_F \le \vareps,$$ namely, $(\vX^\tau, \vLam^\tau)$ is an $\vareps$-stationary point of \eqref{eq:min-max-Phi} in expectation.
\end{proof}

\begin{remark}\label{rm:complexity}
We make a few remarks about Theorem~\ref{thm:up-bd-T}. First, the complexity has the optimal dependence on $\vareps$. Also, the dependence on $\rho$ is the best known for a decentralized method that does not perform multi-communication per iteration on solving composite nonconvex problems, such as SONATA \cite{scutari2019distributed}. Second, the assumed condition $\|\vV^0_\perp\|_F^2 \le 2m L\kappa(1-\rho)$ can hold if $\kappa$ is big. Otherwise, we can perform multi-communication in the initial step. Third, the theorem gives the number of communication rounds and $\vx$-gradient evaluations in order to produce an $\vareps$-stationary point of \eqref{eq:min-max-Phi} in expectation. If each local $\vy$-subproblem has a closed-form solution, such as for the distributionally robust logistic regression that we will test in the next section,  then the total number of $\vy$-gradient evaluations is also $T$. In general, we need to apply an iterative solver to find each local approximate solution. Below, we estimate the total number of $\vy$-gradients for the general case.
\end{remark}

An accelerated proximal gradient method (APG) that has accelerated convergence for strongly-convex problems, e.g., the methods in \cite{nesterov2013gradient, lin2015adaptive}, can be applied to find $\vy_i^t$ for each $t\ge0$ and each $i\in [m]$. To guarantee the stationarity condition required in Line~7 of Algorithm~\ref{alg:dgdm}, we apply the modified version of the APG in \cite{lin2015adaptive}, which is given in  \cite[Algorithm~2]{xu2022first-O1}. For simplicity, we use the non-adaptive version (i.e., no line search by assuming the knowledge of the smoothness constant) to approximately solve $\max_{\vy} d_i^t (\vy)$, or equivalently $\min_\vy (-d_i^t) (\vy)$, where $d_i^t$ is defined in \eqref{eq:def-d-i}. Let $\vy_i^{t\star}$ be the unique maximizer of $d_i^t$. Notice $d_i^t(\vy_i^{t\star}) - d_i^t(\vy_i) \ge \frac{\mu}{2}\|\vy_i^{t\star} - \vy_i\|^2, \forall\, \vy_i \in \dom(h)$ by the $\mu$-strong concavity of $d_i^t$. Hence, if the non-adaptive version of Algorithm~2 in \cite{xu2022first-O1} is applied to solve $\min_\vy (-d_i^t) (\vy)$ and starts from $\vy_i^{t-1}$, then from \cite[Theorem~2.2]{xu2022first-O1}, it can produce a point $\vy_i^{t-1, s}$ after $s$ iterations such that
\begin{equation}\label{eq:stat-cond-viol-y}
\dist\big(\vzero, \partial d_i^t(\vy_i^{t-1, s}) \big) \le  4\sqrt{L_y} \sqrt{d_i^t(\vy_i^{t\star}) - d_i^t(\vy_i^{t-1})} \textstyle \left(1-\frac{1}{\sqrt{\kappa_y}}\right)^{\frac{s}{2}},
\end{equation} 
where $L_y$ and $\kappa_y$ are respectively given in \eqref{eq:smooth-y} and \eqref{eq:kappa}. Therefore, to produce $\vy_i^t$ such that $\dist\big(\vzero, \partial d_i^t(\vy_i^{t}) \big) \le \delta_t$, it is sufficient to run $s_t$ APG iterations, if
$
s_t \ge {\ln \frac{16 L_y \big(d_i^t(\vy_i^{t\star}) - d_i^t(\vy_i^{t-1})\big)}{\delta_t^2}} \Big / { \ln \frac{1}{1-{1}/{\sqrt{\kappa_y}}}}. 
$
Hence, by the fact $\ln\frac{1}{1-x} \ge x, \forall x \in (0, 1)$, we can set
\begin{equation}\label{eq:APG-s-t}
s_t = \left\lceil \sqrt{\kappa_y} \ln \frac{16 L_y \big(d_i^t(\vy_i^{t\star}) - d_i^t(\vy_i^{t-1})\big)}{\delta_t^2}\right\rceil. 
\end{equation}

We show an upper bound of $d_i^t(\vy_i^{t\star}) - d_i^t(\vy_i^{t-1})$ as follows. Its proof is given in the appendix.

\begin{lemma}\label{lem:bd-dit}
Let $\vY^{t\star} = \argmax_\vY \Phi(\vX^t, \vLam^t, \vY)$, i.e., $\vy_i^{t\star} = \argmax_{\vy} d_i^t(\vy), \forall\, t\ge0, \, i\in [m]$. Then
\begin{equation}\label{eq:bd-d-i-t-star}
\begin{aligned}
&\, d_i^t(\vy_i^{t\star}) - d_i^t(\vy_i^{t-1}) 
\le 
\textstyle \left(L(11+\sqrt{m}) + 22 m \kappa^2 (2L + L\sqrt{m} + 2)\right) \frac{C_{0,t-1}(1-\rho)^2}{L(4\kappa+1)}  \\ 
&\, \textstyle \hspace{0.5cm}+ \frac{(1-\rho)\big(\frac{L}{2}+\kappa^2(2L + L\sqrt{m} + 2)\big)}{2L^2(1+6\kappa^2)} \big\|\vV_\perp^{0}\big\|_F^2 
\textstyle  + \frac{m\big(\frac{L}{2}+\kappa^2(2L + L\sqrt{m} + 2)\big)}{2L^2} \Delta_{t-1}  +  \left(L + \frac{L\sqrt{m}}{2} + 1\right) \frac{\delta_{t-1}^2}{\mu^2} + \frac{\delta_{t-1}^2}{2},
\end{aligned}
\end{equation}
for all $t\ge1$, where $C_{0, t}$ and $\Delta_t$ are defined in Lemma~\ref{lem:bd-sum-x-lam}.
\end{lemma}

With \eqref{eq:APG-s-t} and \eqref{eq:bd-d-i-t-star}, we are ready to bound the total number of $\vy$-gradients.
\begin{theorem}[Total number of $\vy$-gradients]\label{thm:y-grad}
Under the same assumptions of Theorem~\ref{thm:up-bd-T}, set $\delta_t = \frac{(1-\rho)^2}{8\kappa(1+t)}, \forall\, t \ge 0$. 
Then to produce $(\vX^\tau, \vLam^\tau)$, the total number of $\vy$-gradient evaluations satisfies
$$T_y \le \textstyle \left\lceil \sqrt{\kappa_y} \ln \frac{1024 \kappa L_y \big(d_i^{0\star} - d_i^0(\vy_i^{-1})\big)}{(1-\rho)^4} + T \sqrt{\kappa_y} \ln \frac{1024\kappa L_y D (T+1)^2}{(1-\rho)^4} + T + 1 \right\rceil, \forall\, i\in [m],$$ 
where $\vy_i^{-1}$ is the initial point that is used to solve $d_i^{0\star}:=\max_{\vy} d_i^0(\vy)$, and 
\begin{equation}\label{eq:def-D}
\begin{aligned}
D:= &\, \textstyle \left(L(11+\sqrt{m}) + 22 m \kappa^2 (2L + L\sqrt{m} + 2)\right) \frac{(\phi_0 + 1 + \frac{1}{6L\kappa})(1-\rho)^2}{L(4\kappa+1)}   \\
& \, \textstyle + \frac{m\kappa(1-\rho)^2\big(\frac{L}{2}+\kappa^2(2L + L\sqrt{m} + 2)\big)}{L(1+6\kappa^2)} + \frac{m\big(\frac{L}{2}+\kappa^2(2L + L\sqrt{m} + 2)\big)(1-\rho)^4}{64L^2 \kappa^2}   + \frac{ (L + \frac{L\sqrt{m}}{2} + 1)(1-\rho)^4}{64\mu^2 \kappa^2} + \frac{(1-\rho)^4}{128\kappa^2}.
\end{aligned}
\end{equation}
\end{theorem}

\begin{proof}
From the proof of Theorem~\ref{thm:up-bd-T}, it holds $\Delta_t \le \frac{(1-\rho)^4}{32\kappa^2},\forall\, t\ge 0$ and $C_{0,t} \le \phi_0 + 1 + \frac{1}{6L\kappa}$. In addition, it is obvious that $\delta_t^2 \le  \frac{(1-\rho)^4}{64\kappa^2}$. Hence, substituting these upper bounds into \eqref{eq:bd-d-i-t-star}, we can easily obtain $d_i^t(\vy_i^{t\star}) - d_i^t(\vy_i^{t-1}) \le D, \forall\, t\ge 1, i\in [m]$. Hence by \eqref{eq:APG-s-t}, for each $i$, the total number of $\vy$-gradient evaluations is upper bounded by 
\begin{align*}
T_y = \textstyle \sum_{t = 0}^T s_t \le &\, \textstyle \sqrt{\kappa_y} \ln \frac{16 L_y \big(d_i^{0\star} - d_i^0(\vy_i^{-1})\big)}{\delta_0^2} + \sum_{t=1}^T \sqrt{\kappa_y} \ln \frac{16L_y D}{\delta_t^2} + T + 1\cr
\le &\, \textstyle \sqrt{\kappa_y} \ln \frac{16 L_y \big(d_i^{0\star} - d_i^0(\vy_i^{-1})\big)}{\delta_0^2} + T\sqrt{\kappa_y} \ln \frac{16L_y D}{\delta_T^2} + T + 1\cr
= &\, \textstyle \sqrt{\kappa_y} \ln \frac{1024 \kappa L_y \big(d_i^{0\star} - d_i^0(\vy_i^{-1})\big)}{(1-\rho)^4} + T \sqrt{\kappa_y} \ln \frac{1024\kappa L_y D (T+1)^2}{(1-\rho)^4} + T + 1,
\end{align*}
which completes the proof.
\end{proof}

\begin{remark}\label{rm:complexity-result}
By Theorems~\ref{thm:up-bd-T} and \ref{thm:y-grad}, Algorithm~\ref{alg:dgdm} can produce an $\vareps$-stationary point of \eqref{eq:min-max-Phi} in expectation,  by $O\big(\frac{L\kappa}{\vareps^2(1-\rho)^2}\big)$ communication rounds  and $\vx$-gradient evaluations and  $O\big(\frac{L\kappa\sqrt{\kappa_y}}{\vareps^2(1-\rho)^2} \ln \frac{\kappa}{\vareps}\big)$ $\vy$-gradient evaluations.
\end{remark}

\subsection{Relation between the stationarity for \eqref{eq:min-max-prob-dec} and \eqref{eq:min-max-Phi}}
In this subsection, we show that a (near) stationary point of $\phi$ is also a (near) stationary point of $p+g$, where $p(\cdot)$ is defined in \eqref{eq:def-p} and $\phi(\cdot)$ defined in \eqref{eq:def-P}. With these, we are able to establish the relation between the stationarity for \eqref{eq:min-max-Phi} and the original formulation \eqref{eq:min-max-prob-dec}. Thus the results in the previous subsection will translate to complexity results to produce a near-stationary point of \eqref{eq:min-max-prob-dec}. The lemma below establishes the relation for the exact case.


\begin{lemma}
If $(\bar\vx, \bar\vLam)$ is a stationary point of $\phi$ defined in \eqref{eq:def-P}, i.e., for a certain $\eta>0$,
\begin{equation}\label{eq:opt-cond-phi}
\bar\vx = \prox_{\eta g}\big(\bar\vx - \eta\nabla_\vx P(\bar\vx,\bar\vLam)\big),\quad \nabla_\vLam P(\bar\vx,\bar\vLam) = \vzero,
\end{equation}
then $\bar\vx$ is also a stationary point of $p+g$, i.e., $\bar\vx = \prox_{\eta g}\big(\bar\vx - \eta\nabla p(\bar\vx)\big)$, where $p(\cdot)$ is defined in \eqref{eq:def-p}.
\end{lemma}

\begin{proof}
Let $\bar\vY=S_\Phi(\vone\bar\vx^\top, \bar\vLam)$. Then $\nabla P(\bar\vx, \bar\vLam) = \left(\frac{1}{m}\sum_{i=1}^m \nabla_\vx f_i(\bar\vx, \bar\vy_i),\ -\frac{L}{2\sqrt{m}}(\vW-\vI)\bar\vY\right)$.
Hence, we have from $\nabla_\vLam P(\bar\vx,\bar\vLam) = \vzero$ that $\bar\vY = \frac{1}{m}\vone\vone^\top \bar\vY = \vone\bar\vy_\avg^\top$, and by the definition of $\bar\vY$, it holds
$\Phi(\vone\bar\vx^\top, \bar\vLam, \vone\bar\vy_\avg^\top) \ge \Phi(\vone\bar\vx^\top, \bar\vLam, \vone\vy^\top), \forall\, \vy\in\dom(h)$. Notice $(\vW-\vI) \vone \vy^\top = \vzero$. We have $\bar\vy_\avg = \argmax_\vy \frac{1}{m}\sum_{i=1}^m f_i(\bar\vx, \vy) - h(\vy)$. Thus $\nabla p(\bar\vx) = \frac{1}{m}\sum_{i=1}^m \nabla_\vx f_i(\bar\vx, \bar\vy_\avg) = \nabla_\vx P(\bar\vx,\bar\vLam)$. Therefore, the first condition in \eqref{eq:opt-cond-phi} reduces to $\bar\vx = \prox_{\eta g}\big(\bar\vx - \eta\nabla p(\bar\vx)\big)$, and this completes the proof.
\end{proof}


The following lemma will be used to show the relation for the case of near stationarity.
\begin{lemma}\label{lem:lam_min}
Under Assumption~\ref{assump-W}, we have that the eigenvalues of $(\vW-\vI)^\top(\vW-\vI)$ restricted on $\Span\{\vone\}^\perp$ are positive, and we denote the smallest one as $\lambda_{\min}^+ > 0$.
\end{lemma}
\begin{proof}
Let $\vx\in \Span\{\vone\}^\perp$ and $\vx\neq\vzero$. Suppose $\vx^\top (\vW-\vI)^\top(\vW-\vI) \vx = 0$, i.e., $\|(\vW-\vI) \vx\|^2 = 0$, so $(\vW-\vI) \vx = \vzero$. Hence, $\vx \in \Null(\vW-\vI)$. By Assumption~\ref{assump-W}(iii), it follows $\vx\in \Span\{\vone\}$. Because  $\vx\in \Span\{\vone\}^\perp$ also holds, it can only be $\vx=\vzero$, which contradicts to $\vx\neq \vzero$. Therefore, $\vx^\top (\vW-\vI)^\top(\vW-\vI) \vx > 0$, i.e., $(\vW-\vI)^\top(\vW-\vI)$ is positive definite on $\Span\{\vone\}^\perp$. This completes the proof.
\end{proof}

The next lemma shows a near-stationary point of $\phi$ is also a near-stationary solution of $p+g$ under different conditions on $h$.
\begin{lemma}\label{lem:rel-stat}
Suppose that $(\bar\vx, \bar\vLam)$ is an $\vareps$-stationary point of $\phi$, i.e., for a certain $\eta>0$, 
\begin{equation}\label{eq:opt-cond-phi-eps-bar}
\textstyle \frac{1}{\eta}\left\|\bar\vx - \prox_{\eta g}\big(\bar\vx - \eta\nabla_\vx P(\bar\vx,\bar\vLam)\big)\right\|\le \vareps,\quad \|\nabla_\vLam P(\bar\vx,\bar\vLam)\|_F \le \vareps.
\end{equation}
 We have the follows: 
\begin{enumerate}
\item[(i)] If $h$ is $L_h$-smooth, then $\bar\vx$ is an $\vareps_1$-stationary point of $p+g$, where $\vareps_1= \vareps + \vareps\sqrt{\frac{8(\mu + L + L_h)}{\mu \lambda_{\min}^+}}$ and $\lambda_{\min}^+$ is given in Lemma~\ref{lem:lam_min}, namely,  $\frac{1}{\eta}\left\| \bar\vx - \prox_{\eta g}\big(\bar\vx - \eta\nabla p(\bar\vx)\big)\right\| \le \vareps_1$.
\item[(ii)] If $\|\vxi\| \le M_h$ for any $\vxi\in \partial h(\vy)$ and for any $\vy\in \dom(h)$, 
then $\bar\vx$ is an $\vareps_2$-stationary point of $p+g$, where $\vareps_2=\vareps + \vareps\sqrt{\frac{8(\mu + L)}{\mu \lambda_{\min}^+}} + \sqrt{\frac{16L \vareps}{\mu  \lambda_{\min}^+} + \frac{4 L M_h^2 \vareps}{\mu }}$.
\end{enumerate}
\end{lemma}

\begin{proof}
For the given $(\bar\vx, \bar\vLam)$, we define 
$$\psi(\vY) := \Phi(\vone\bar\vx^\top, \bar\vLam, \vY),\quad \bar\vY = \argmax_\vY \psi(\vY), \quad \hat\vy=\argmax_\vy f(\bar\vx, \vy) - h(\vy).$$ 
It holds $\nabla_\vLam P(\bar\vx,\bar\vLam) = -\frac{L}{2\sqrt{m}} (\vW-\vI)\bar\vY$ and thus
\begin{equation}\label{eq:bd-Y-perp}
\textstyle\vareps^2\ge \|\nabla_\vLam P(\bar\vx,\bar\vLam)\|_F^2 = \frac{L^2}{4m}\|(\vW-\vI)\bar\vY\|_F^2 = \frac{L^2}{4m}\|(\vW-\vI)\bar\vY_\perp\|_F^2 \ge \frac{\lambda_{\min}^+ L^2}{4m}\|\bar\vY_\perp\|_F^2,
\end{equation} 
where the second equality holds because $\vW-\vI = (\vW-\vI)(\vI - \frac{1}{m}\vone\vone^\top)$, and the second inequality follows from Lemma~\ref{lem:lam_min}. 
Also, we have
\begin{equation*}
 \begin{aligned}
& ~\|\nabla p(\bar\vx) - \nabla_\vx P(\bar\vx,\bar\vLam)\|^2 = \textstyle \big\|\frac{1}{m}\sum_{i=1}^m \nabla_\vx f_i(\bar\vx, \hat\vy) - \frac{1}{m}\sum_{i=1}^m \nabla_\vx f_i(\bar\vx, \bar\vy_i)\big\|^2 \le \frac{L^2}{m} \sum_{i=1}^m\|\hat\vy - \bar\vy_i\|^2 \\
\le & ~ \textstyle \frac{2L^2}{m} \sum_{i=1}^m\left(\|\bar\vy_\avg - \bar\vy_i\|^2+ \|\hat\vy - \bar\vy_\avg\|^2\right) = \frac{2L^2}{m}\|\bar\vY_\perp\|_F^2 + 2L^2\|\hat\vy - \bar\vy_\avg\|^2,
\end{aligned}
\end{equation*}
and thus by the nonexpansiveness of $\prox_{\eta g}$, the first condition in \eqref{eq:opt-cond-phi-eps-bar}, and \eqref{eq:bd-Y-perp}, it follows
\begin{equation}\label{eq:gap-nab-p}
\textstyle \frac{1}{\eta}\left\| \bar\vx - \prox_{\eta g}\big(\bar\vx - \eta\nabla p(\bar\vx)\big)\right\| \le \vareps + \big\|\nabla p(\bar\vx) - \nabla_\vx P(\bar\vx,\bar\vLam)\big\| \le \vareps + \sqrt{ \frac{8\vareps^2}{\lambda_{\min}^+} + 2L^2\|\hat\vy - \bar\vy_\avg\|^2}.
\end{equation}
Hence, to show the near-stationarity of $\bar\vx$ for $p + g$, it is sufficient to bound $\|\hat\vy - \bar\vy_\avg\|^2$.

By the $L$-smoothness of each $f_i$, it follows
\begin{equation}\label{eq:stationary-ineq1}
\textstyle \frac{1}{m}\sum_{i=1}^m f_i(\bar\vx, \bar\vy_\avg) \ge \frac{1}{m}\sum_{i=1}^m \left( f_i(\bar\vx, \bar\vy_i) + \big\langle \nabla_\vy f_i(\bar\vx, \bar\vy_i),  \bar\vy_\avg - \bar\vy_i \big\rangle - \frac{L}{2}\|\bar\vy_\avg - \bar\vy_i\|^2\|\right).
\end{equation}
Also, by the optimality condition at $\bar\vY$, there exists a subgradient $ \bar\vxi_i\in\partial h(\bar\vy_i)$ for each $i$ such that
$$\vzero = \textstyle \frac{1}{m} \big[ \nabla_\vy f_1(\bar\vx, \bar\vy_1) - \bar\vxi_1, \ldots, \nabla_\vy f_m(\bar\vx, \bar\vy_m) - \bar\vxi_m \big]^\top - \frac{L}{2\sqrt{m}} (\vW-\vI)^\top \bar\vLam.$$
The equation above together with \eqref{eq:stationary-ineq1} gives
\begin{align}\label{eq:stationary-ineq2}
&\, \textstyle \frac{1}{m}\sum_{i=1}^m f_i(\bar\vx, \bar\vy_\avg) \nonumber\\ 
\ge &\, \textstyle \frac{1}{m}\sum_{i=1}^m \left(f_i(\bar\vx, \bar\vy_i) + \big\langle \bar\vxi_i,  \bar\vy_\avg - \bar\vy_i \big\rangle - \frac{L}{2}\|\bar\vy_\avg - \bar\vy_i\|^2\|\right) + \frac{L}{2\sqrt{m}} \big\langle (\vW-\vI)^\top \bar\vLam, \vone\bar\vy_\avg^\top - \bar\vY\big\rangle \nonumber\\ 
= &\, \textstyle \frac{1}{m}\sum_{i=1}^m \left(f_i(\bar\vx, \bar\vy_i) + \big\langle \bar\vxi_i,  \bar\vy_\avg - \bar\vy_i \big\rangle - \frac{L}{2}\|\bar\vy_\avg - \bar\vy_i\|^2\|\right) - \frac{L}{2\sqrt{m}} \big\langle (\vW-\vI)^\top \bar\vLam, \bar\vY\big\rangle \nonumber\\
 = &\, \textstyle \psi(\bar\vY) - \frac{L}{2m} \|\bar\vY_\perp\|_F^2 + \frac{1}{m}\sum_{i=1}^m \left(h(\bar\vy_i) +\big\langle \bar\vxi_i,  \bar\vy_\avg - \bar\vy_i \big\rangle \right),
\end{align}
where the first equality follows from $\vW\vone = \vone$ and the second equality uses the definition of $\psi$. Subtract both sides of \eqref{eq:stationary-ineq2} by $h(\bar\vy_\avg)$ and use the definition of $f$ to have 
\begin{equation}\label{eq:stationary-ineq3}
f(\bar\vx, \bar\vy_\avg) - h(\bar\vy_\avg) \ge \textstyle \psi(\bar\vY)  - \frac{L}{2m} \|\bar\vY_\perp\|_F^2 + \frac{1}{m}\sum_{i=1}^m \left(h(\bar\vy_i) - h(\bar\vy_\avg) +\big\langle \bar\vxi_i,  \bar\vy_\avg - \bar\vy_i \big\rangle \right).
\end{equation}
Moreover, by the definition of $\bar\vY$, it holds
$\psi(\bar\vY) \ge \psi(\vone\hat\vy^\top) =  f(\bar\vx, \hat\vy) - h(\hat\vy)$, and by the $\mu$-strong concavity of $f(\bar\vx, \,\cdot\,) - h(\cdot)$, we have $f(\bar\vx, \bar\vy_\avg) - h(\bar\vy_\avg) \le f(\bar\vx, \hat\vy) - h(\hat\vy) - \frac{\mu}{2}\|\hat\vy - \bar\vy_\avg\|^2$. Combining these two inequalities with \eqref{eq:stationary-ineq3} gives
\begin{equation}\label{eq:stationary-ineq4}
\textstyle  \frac{\mu}{2}\|\hat\vy - \bar\vy_\avg\|^2 \le    \frac{L}{2m} \|\bar\vY_\perp\|_F^2 + \frac{1}{m}\sum_{i=1}^m \left(h(\bar\vy_\avg) - h(\bar\vy_i) -  \big\langle \bar\vxi_i,  \bar\vy_\avg - \bar\vy_i \big\rangle \right).
\end{equation}

Now if $h$ is $L_h$-smooth, the subgradient $\bar\vxi_i$ reduces to the gradient $\nabla h(\bar\vy_i)$, and $  h(\bar\vy_\avg) - h(\bar\vy_i) - \big\langle \nabla h(\bar\vy_i),  \bar\vy_\avg - \bar\vy_i \big\rangle \le \frac{L_h}{2}\|\bar\vy_\avg - \bar\vy_i\|^2$ for each $i$.
 Hence, \eqref{eq:stationary-ineq4} implies $\mu \|\hat\vy - \bar\vy_\avg\|^2 \le \frac{L+L_h}{m} \|\bar\vY_\perp\|_F^2$, which together with \eqref{eq:bd-Y-perp} gives $\|\hat\vy - \bar\vy_\avg\|^2 \le \frac{4 (L + L_h) \vareps^2}{\mu L^2 \lambda_{\min}^+}$. Thus by \eqref{eq:gap-nab-p}, we obtain 
\begin{equation*}
\textstyle \frac{1}{\eta}\left\| \bar\vx - \prox_{\eta g}\big(\bar\vx - \eta\nabla p(\bar\vx)\big)\right\| \le \vareps + \sqrt{ \frac{8\vareps^2}{\lambda_{\min}^+} + \frac{8 (L + L_h) \vareps^2}{\mu  \lambda_{\min}^+}} = \vareps + \vareps\sqrt{\frac{8(\mu + L + L_h)}{\mu \lambda_{\min}^+}},
\end{equation*}
which proves the first claim of the lemma.

When $\|\bar\vxi_i\| \le M_h$,  it holds 
$\textstyle h(\bar\vy_\avg) - h(\bar\vy_i)  - \big\langle \bar\vxi_i,  \bar\vy_\avg - \bar\vy_i \big\rangle \le 2 M_h\|\bar\vy_\avg - \bar\vy_i\| \le \frac{M_h^2 \vareps}{L} + \frac{L \|\bar\vy_\avg - \bar\vy_i\|^2}{\vareps}$ for each $i$.
 Hence, \eqref{eq:stationary-ineq4} implies $\frac{\mu}{2}\|\hat\vy - \bar\vy_\avg\|^2 \le    \frac{L}{2m} \|\bar\vY_\perp\|_F^2 + \frac{L}{m\vareps} \|\bar\vY_\perp\|_F^2 + \frac{M_h^2 \vareps}{L}$, which together with \eqref{eq:bd-Y-perp} gives $\|\hat\vy - \bar\vy_\avg\|^2 \le (1 + \frac{2}{\vareps})\frac{4\vareps^2}{\mu  \lambda_{\min}^+ L} + \frac{2 M_h^2 \vareps}{\mu L}$. Thus by \eqref{eq:gap-nab-p}, we obtain 
\begin{equation*}
\textstyle \frac{1}{\eta}\left\| \bar\vx - \prox_{\eta g}\big(\bar\vx - \eta\nabla p(\bar\vx)\big)\right\| \le \vareps + \sqrt{ \frac{8\vareps^2}{\lambda_{\min}^+} + (1 + \frac{2}{\vareps})\frac{8L \vareps^2 }{\mu  \lambda_{\min}^+ } + \frac{4 L M_h^2 \vareps}{\mu }} \le \vareps + \vareps\sqrt{\frac{8(\mu + L)}{\mu \lambda_{\min}^+}} + \sqrt{\frac{16L \vareps}{\mu  \lambda_{\min}^+} + \frac{4 L M_h^2 \vareps}{\mu }},
\end{equation*}
which completes the proof.
\end{proof}

\begin{remark}\label{rm:stat-original}
By Lemma~\ref{lem:rel-stat}, to produce an $\vareps$-solution $\vX$ of the original formulation \eqref{eq:min-max-prob-dec}, i.e., $\frac{L}{\sqrt{m}}\|\vX_\perp\|_F \le \vareps$ and $\frac{1}{\eta}\left\| \vx_\avg - \prox_{\eta g}\big(\vx_\avg - \eta\nabla p(\vx_\avg)\big)\right\| \le \vareps$ for some $\eta>0$, the complexity will have an additional factor of $\kappa$ on top of those in Remark~\ref{rm:complexity-result} when $h$ is smooth, and an additional factor of $\frac{\kappa^2}{\vareps^2}$ when $h$ is nonsmooth but has bounded subgradients. The latter one gives higher complexity, compared to a centralized GDMax method. We conjecture that the worse result is caused by our possibly non-tight analysis. 
\end{remark}

\section{Numerical Experiments}\label{sec:numerical}
In this section,  we test the proposed algorithm on solving the following distributionally robust logistic regression (DRLR):
\begin{equation}\label{eq:drlr}
\min_{\vx\in\RR^n} \max_{\vy\in\cY} \textstyle \sum_{j=1}^N y_j \ell(\vx; \va_j, b_j) + V_x(\vx) - V_y(\vy),
\end{equation}
where $\{(\va_j, b_j)\}_{j=1}^N$ is the dataset with each $\va_j\in\RR^n$ a feature vector and $b_j\in \{-1, +1\}$ the corresponding label, $\cY=\{\vy \in \RR^N_+: \vone^\top \vy = 1\}$ is the $N$-dimensional simplex, $\ell(\vx; \va, b)=\log\big(1+\exp(-b\cdot\va^\top\vx)\big)$ is the logistic loss, $V_x(\vx) = \beta_x \sum_{i=1}^n \frac{\alpha x_i^2}{1+\alpha x_i^2}$, and $V_y(\vy) = \frac{\beta_y}{2}\|\vy - \frac{\vone}{N}\|^2$ is a strongly-convex term that controls the distance of the probability vector $\vy$ to the uniform distribution vector $\frac{\vone}{N}$. When $\vy= \frac{\vone}{N}$ is enforced, \eqref{eq:drlr} reduces to the standard (non-robust) LR. Suppose the dataset is partitioned into $m$ subsets $\{\cD_i\}_{i=1}^m$. Let $\cJ_i\subset[N]$ be index set corresponding to the sub-dataset $\cD_i$ for each $i\in [m]$. Define
$$\textstyle f_i(\vx, \vy) = m\sum_{j\in \cJ_i} y_j \ell(\vx; \va_j, b_j) + V_x(\vx) - V_y(\vy),\forall\, i\in [m], \quad g(\vx) = 0, \quad h(\vy) = \iota_\cY.$$
Then \eqref{eq:drlr} can be formulated to the form of \eqref{eq:min-max-prob} and satisfies the conditions in Assumptions~\ref{assump-func} and \ref{assump-relint}.

\subsection{Comparison in a centralized setting}
A centralized (or non-distributed) version of our method can be easily obtained if $m=1$ in Algorithm~\ref{alg:dgdm}, and we name it as GDMax. For the purpose of demonstrating the advantage by dual maximization, we only compare our method to the GDA method in \cite{lin2020gradient}. Three LIBSVM \cite{chang2011libsvm} datasets are used: \verb|a9a|, \verb|gisette|, and \verb|rcv1|. For each dataset, we fix $\alpha =10, \beta_x = 10^{-3}$ in \eqref{eq:drlr} by following \cite{yan2019stochastic} and vary $\beta_y$ from $\{0.01, 0.1, 1\}$, with a smaller $\beta_y$ indicating a harder instance. Since the $\vy$-subproblem can be exactly solved by projecting a certain point onto the simplex, no iterative subroutine is needed, and we only need to tune the $\vx$-stepsize for our method. GDA has both $\vx$ and $\vy$-stepsizes to tune. For each dataset, we grid-search $\eta_x$ from $\{0.5, 0.1, 0.05, 0.01, 0.005, 0.001\}$ for our method and $(\eta_x, \eta_y)$ from  $\{0.5, 0.1, 0.05, 0.01, 0.005, 0.001\}^2$ for GDA. The best results are reported. Figure~\ref{fig:cen-alg} shows the results, where the best stepsize is included in the table. We see that a smaller $\beta_y$ (that corresponds to a larger $\kappa$) slows down the convergence of both GDMax and GDA, but GDMax performs significantly better than GDA, especially when $\beta_y$ is small. This demonstrates the advantage of performing dual maximization.

\begin{figure}[h]
\begin{center}
\begin{tabular}{ccc}
{\small $\beta_y = 0.01$} & {\small $\beta_y = 0.1$} & {\small $\beta_y = 1$} \\
\includegraphics[width = 0.25\textwidth]{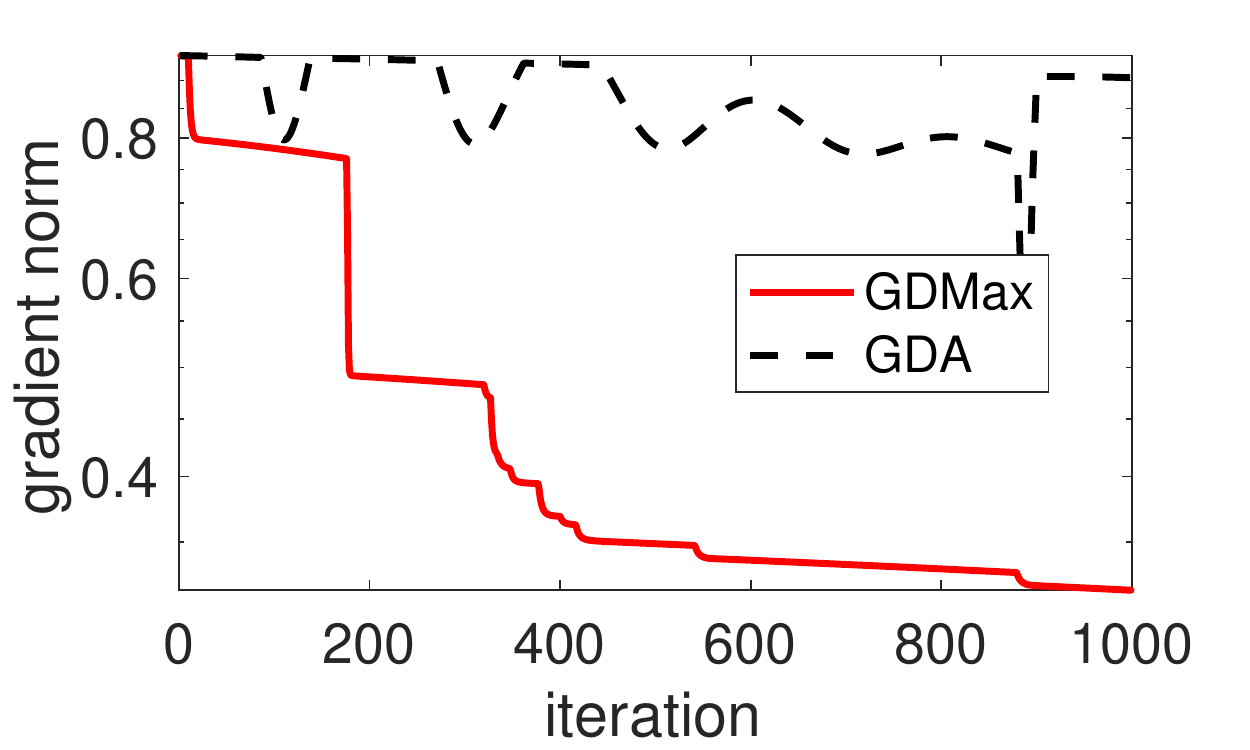} &
\includegraphics[width = 0.25\textwidth]{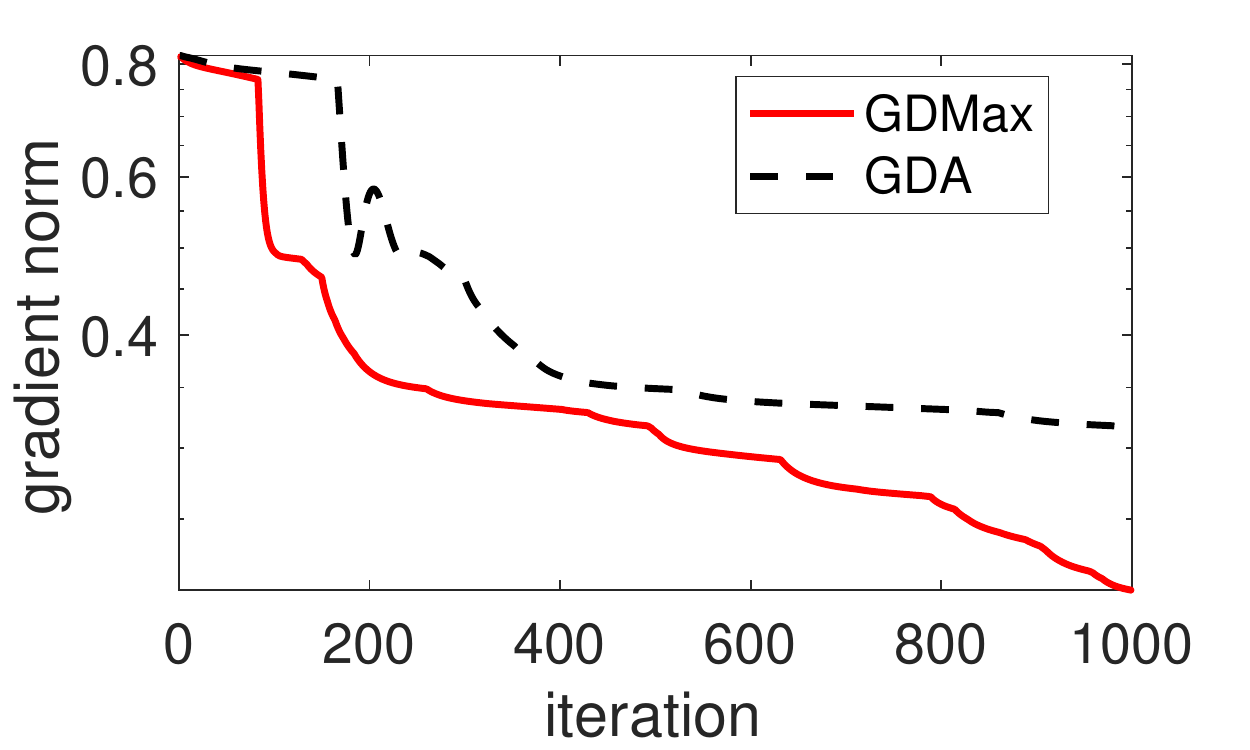} &
\includegraphics[width = 0.25\textwidth]{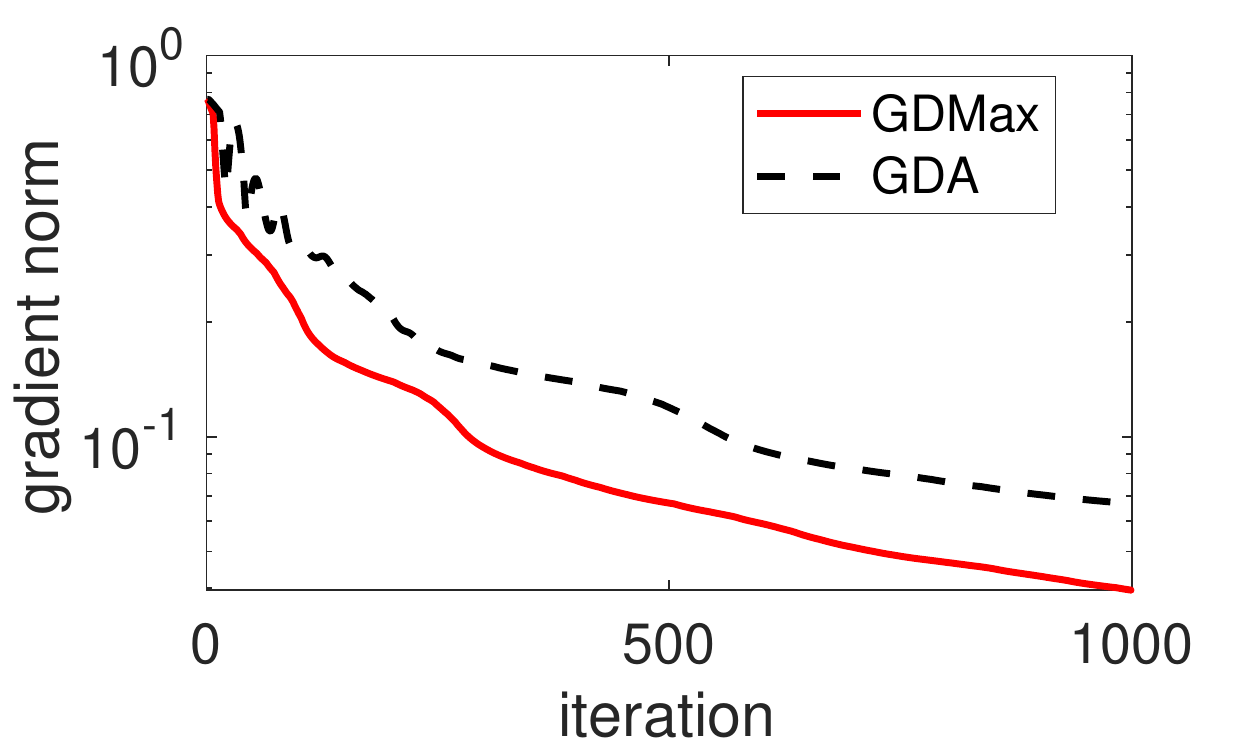} \\
\includegraphics[width = 0.25\textwidth]{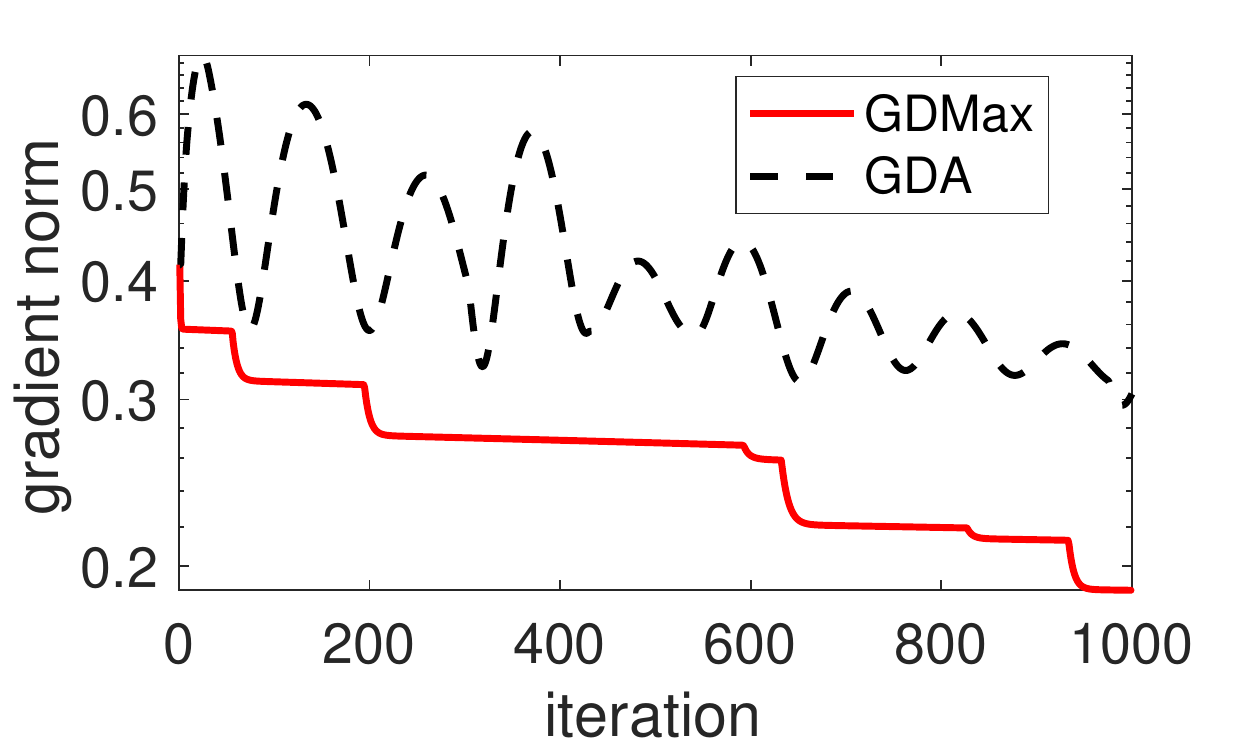} &
\includegraphics[width = 0.25\textwidth]{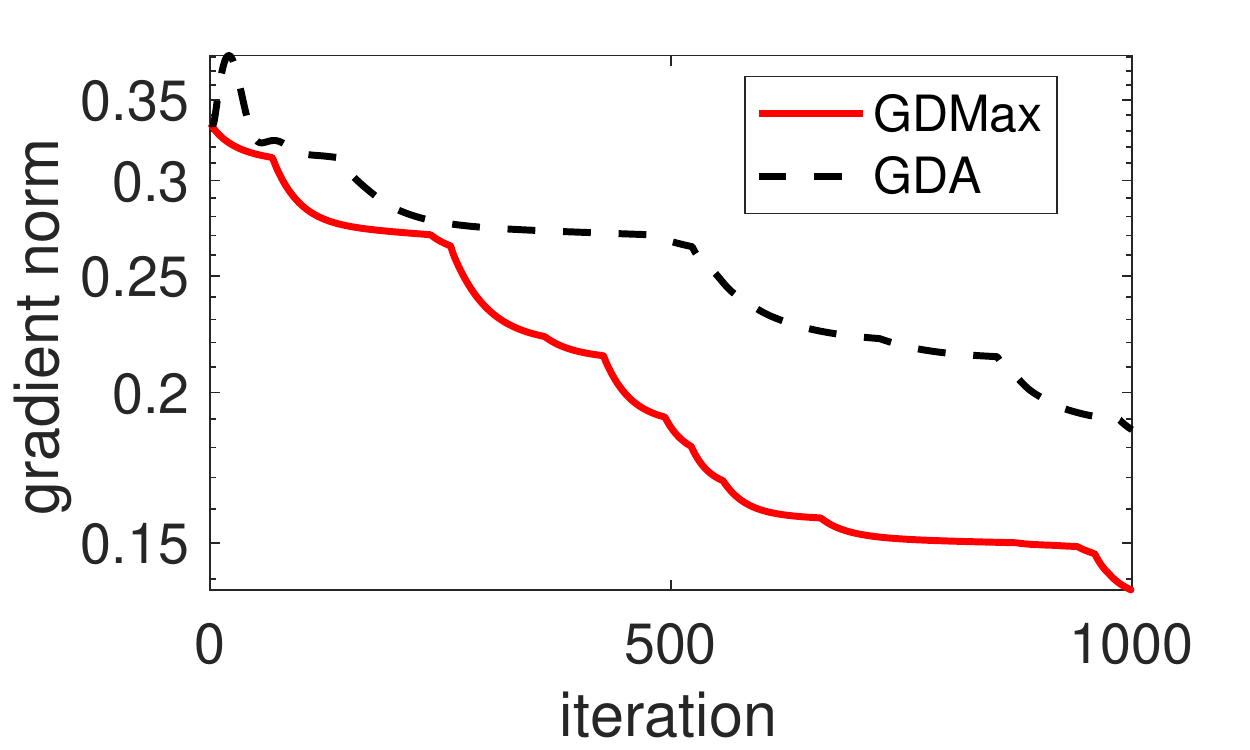} &
\includegraphics[width = 0.25\textwidth]{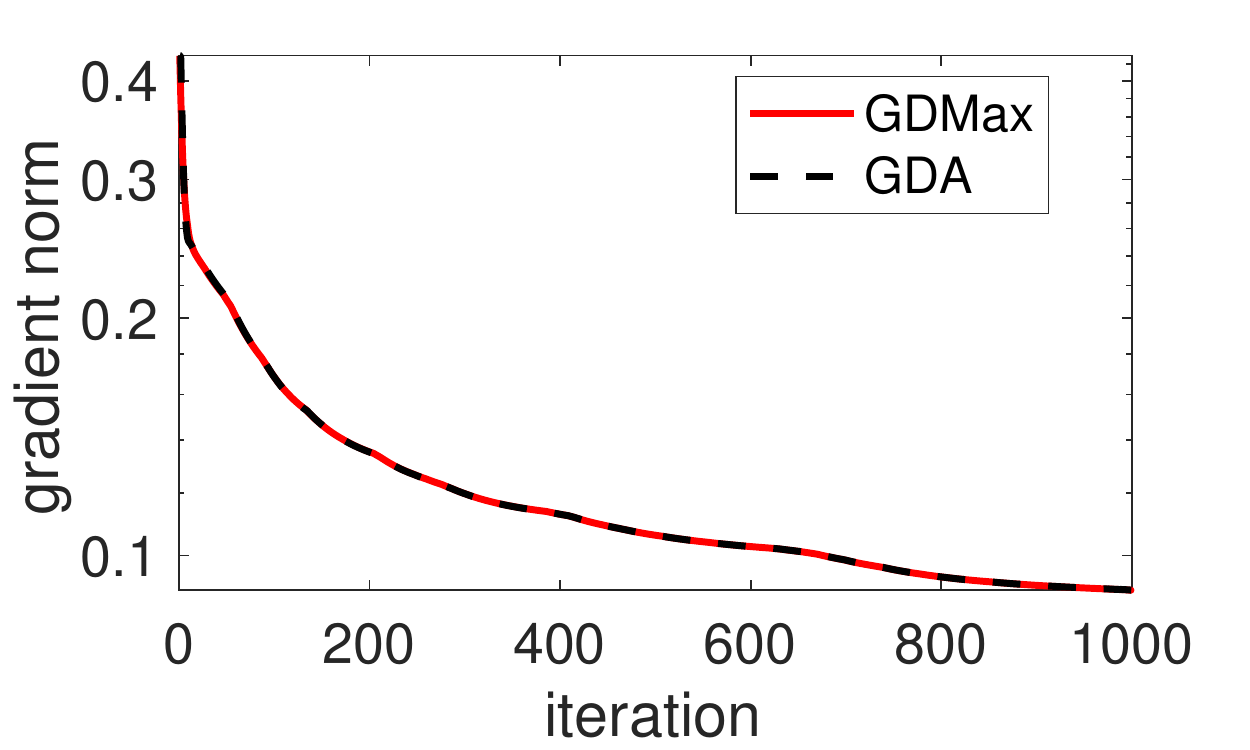} \\
\includegraphics[width = 0.25\textwidth]{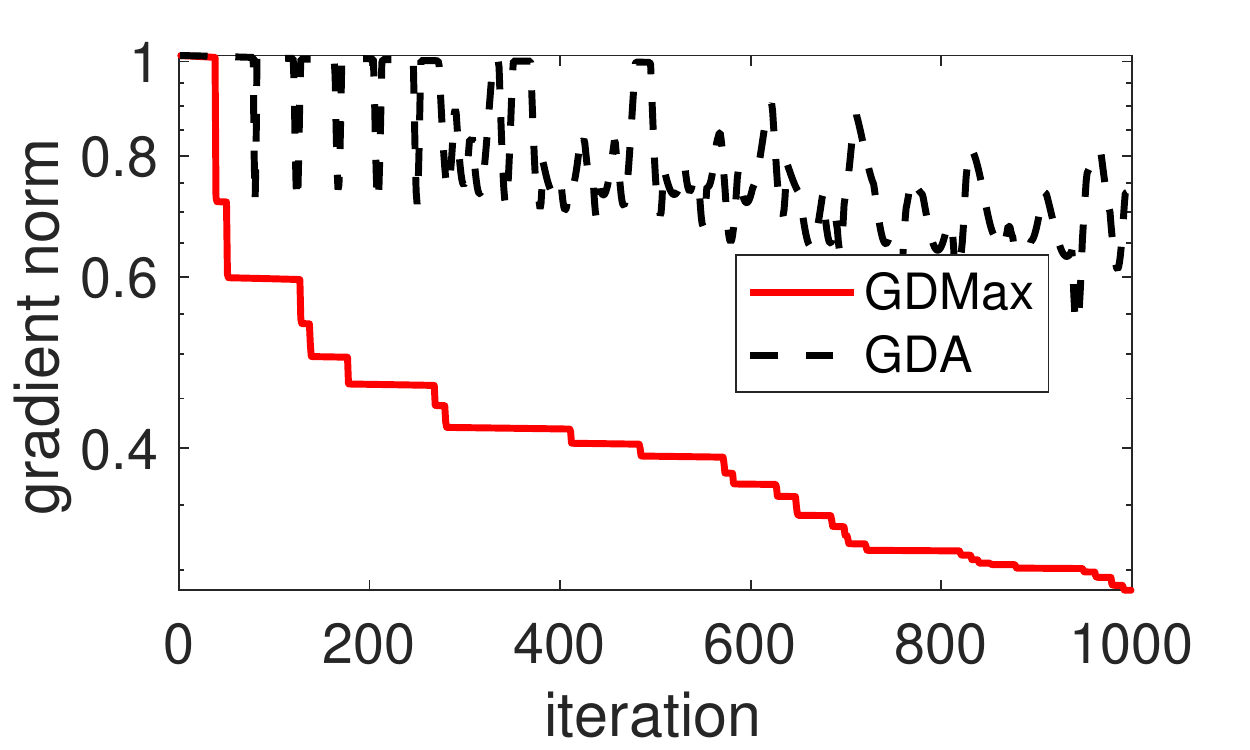} &
\includegraphics[width = 0.25\textwidth]{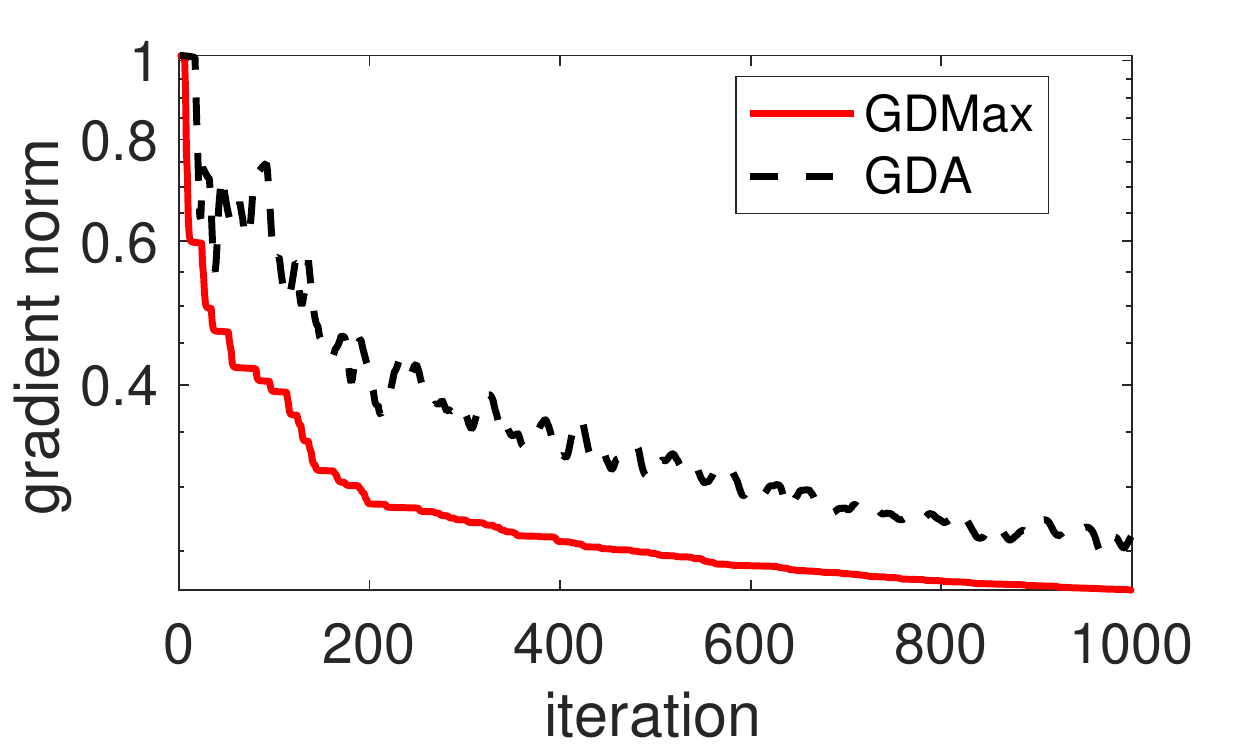} &
\includegraphics[width = 0.25\textwidth]{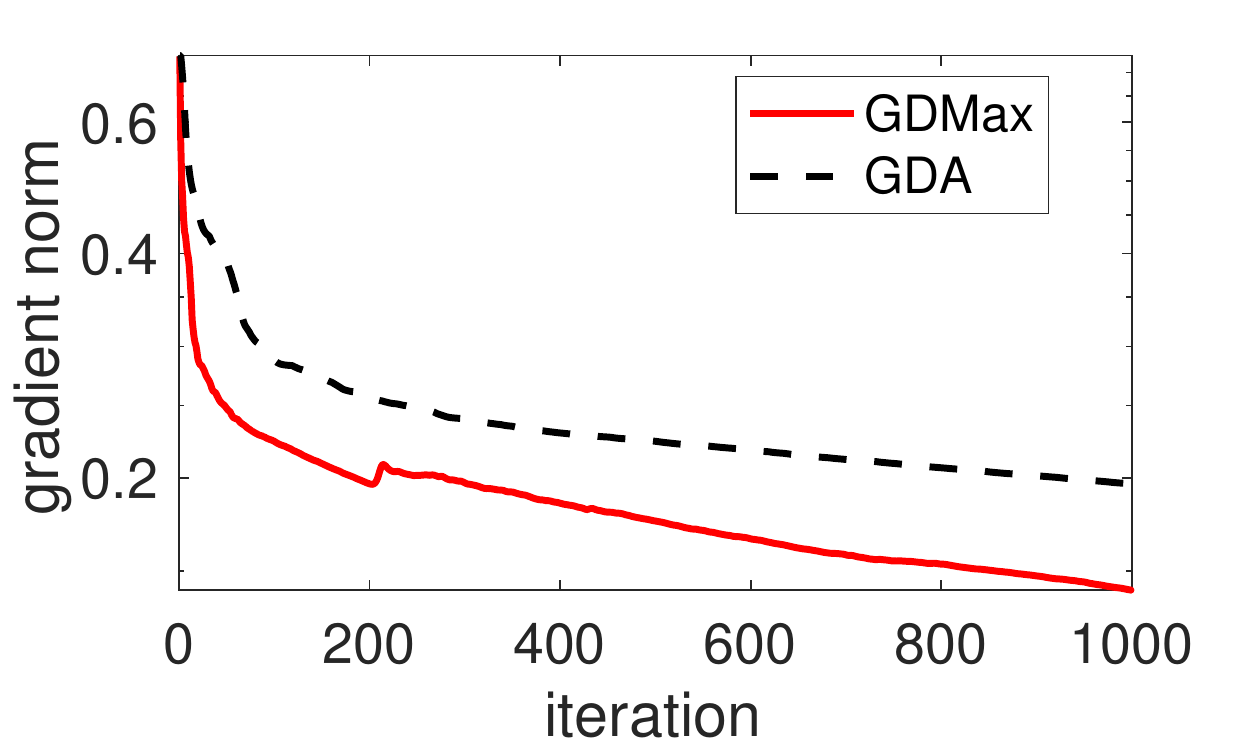} 
\end{tabular}
\resizebox{0.85\textwidth}{!}{
\begin{tabular}{cc|c|c|c|c|c}
\cline{2-7}
&\multicolumn{2}{c|}{\small $\beta_y = 0.01$} & \multicolumn{2}{|c|}{\small $\beta_y = 0.1$} & \multicolumn{2}{|c}{\small $\beta_y = 1$}\\\cline{2-7}
& GDMax & GDA & GDMax & GDA & GDMax & GDA \\\hline
\multicolumn{1}{c|}{a9a} & $\eta_x = 0.005$ & $\eta_x = 0.001, \eta_y = 0.5$ & $\eta_x = 0.01$ & $\eta_x = 0.005, \eta_y = 0.5$ & $\eta_x = 0.1$ & $\eta_x = 0.05, \eta_y = 0.05$\\
\multicolumn{1}{c|}{gisette} & $\eta_x = 0.005$ & $\eta_x = 0.001, \eta_y = 0.5$ & $\eta_x = 0.01$ & $\eta_x = 0.005, \eta_y = 0.5$ & $\eta_x = 0.05$ & $\eta_x = 0.01, \eta_y = 0.5$\\
\multicolumn{1}{c|}{rcv1} & $\eta_x = 0.01$ & $\eta_x = 0.005, \eta_y = 0.5$ & $\eta_x = 0.01$ & $\eta_x = 0.01, \eta_y = 0.1$ & $\eta_x = 0.5$ & $\eta_x = 0.1, \eta_y = 0.5$\\\hline
\end{tabular}
}
\end{center}
\caption{Primal gradient norm at each iterate by the centralized version of our method (called GDMax) and GDA in \cite{lin2020gradient} for solving the DRLR problem in \eqref{eq:drlr} on a9a, gisette, and rcv1 (from top to bottom) with $\beta_y$ varying from $\{0.01,0.1,1\}$ and $\alpha=10, \beta_x=10^{-3}$ fixed. The corresponding best stepsizes are listed in the table.}\label{fig:cen-alg}
\end{figure}

\subsection{Comparison to a decentralized method} We further compare the proposed method D-GDMax to the DREAM method in \cite{chen2022simple}, which appears to be the only existing decentralized method with guaranteed convergence for solving NCSC minimax problem with a $\vy$-constraint, by performing multiple communication every update. Again we use \verb|a9a|, \verb|gisette|, and \verb|rcv1|, and each of them is partitioned uniformly at random into $m=20$ subsets. We generate an Erd\H{o}s-R\'enyi random connected graph with $m$ nodes. The mixing matrix $\vW$ is set to $\vI - 0.8\vL / \lambda_{\max}(\vL)$, where $\vL$ is the graph Laplacian matrix. The generated $\vW$ has $\rho = \|\vW - \frac{1}{m}\vone\vone^\top\|\approx 0.9271$. The same $\vW$ is used by D-GDMax and DREAM. Again we fix $\alpha =10, \beta_x = 10^{-3}$ in \eqref{eq:drlr} but vary $\beta_y$ from $\{0.1,1,10\}$. For DREAM, we simply perform 10 rounds of communication per iteration. Both methods have two stepsizes, and we grid-search them from $\{0.5, 0.1, 0.05, 0.01, 0.005, 0.001\}^2$. The best results are plotted in Figure~\ref{fig:dec-alg} and the corresponding stepsizes listed in the table. The consensus error for DREAM is not plotted, as it stands in the order of $10^{-4}$ at each iterate due to the multiple communication. We see that, except for \verb|rcv1| with $\beta_y=0.1$, the gradient error, which is measured based on the original formulation \eqref{eq:min-max-prob-dec}, by DREAM is significantly larger than that by D-GDMax even though the former performs multiple communication at each iteration. The advantage of D-GDMax should still be attribute to the dual maximization that is enabled by our reformulation.

\begin{figure}[h]
\begin{center}
\begin{tabular}{ccc}
{\small $\beta_y = 0.1$} & {\small $\beta_y = 1$} & {\small $\beta_y = 10$} \\
\includegraphics[width = 0.3\textwidth]{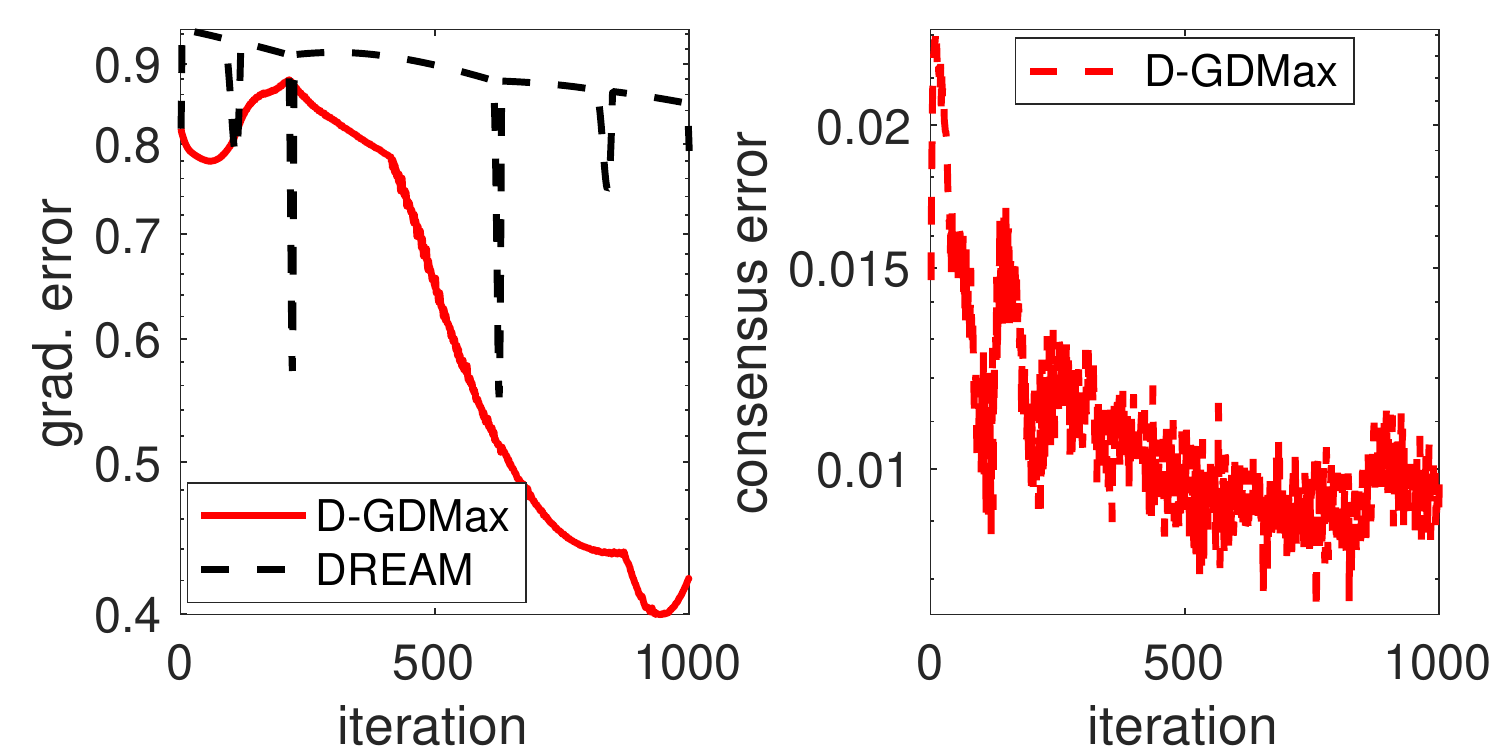} &
\includegraphics[width = 0.3\textwidth]{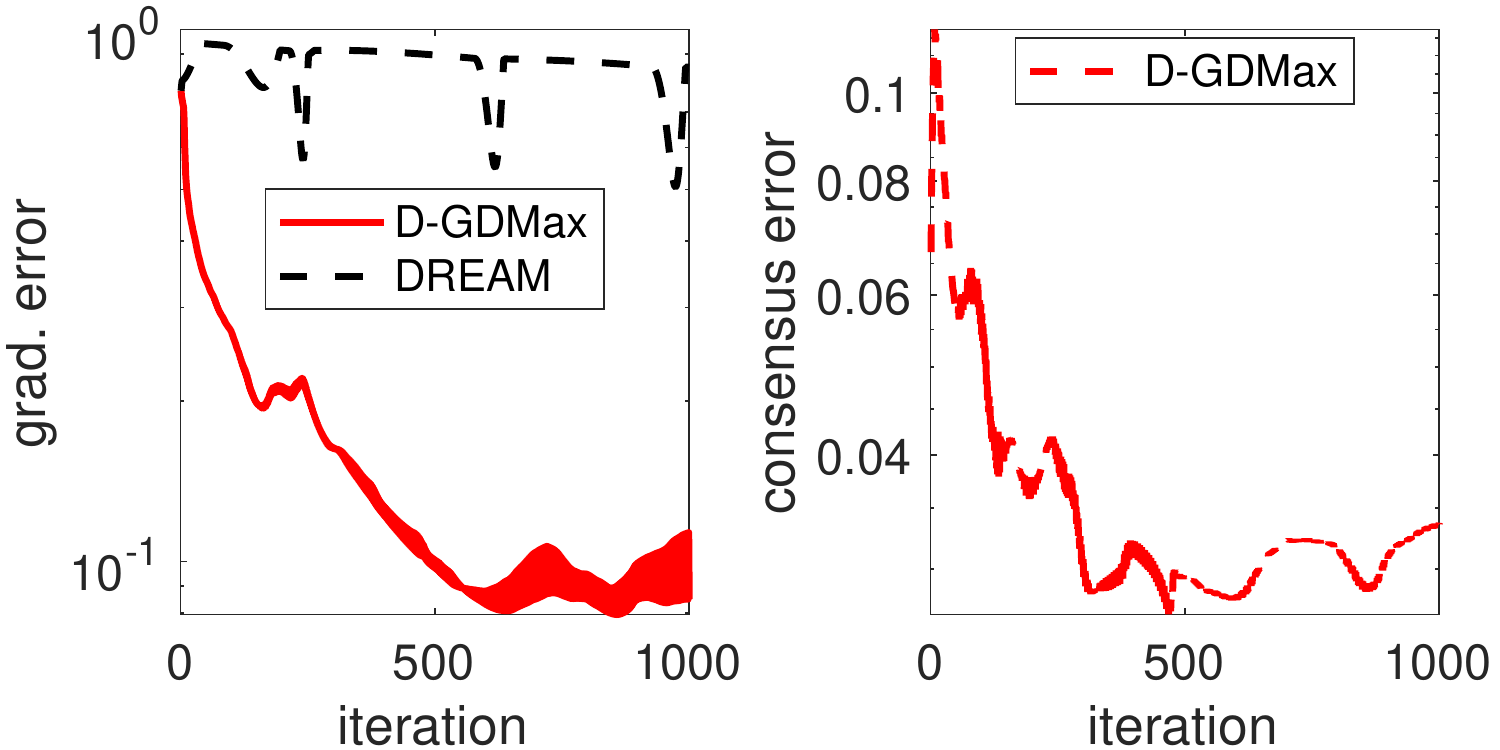} &
\includegraphics[width = 0.3\textwidth]{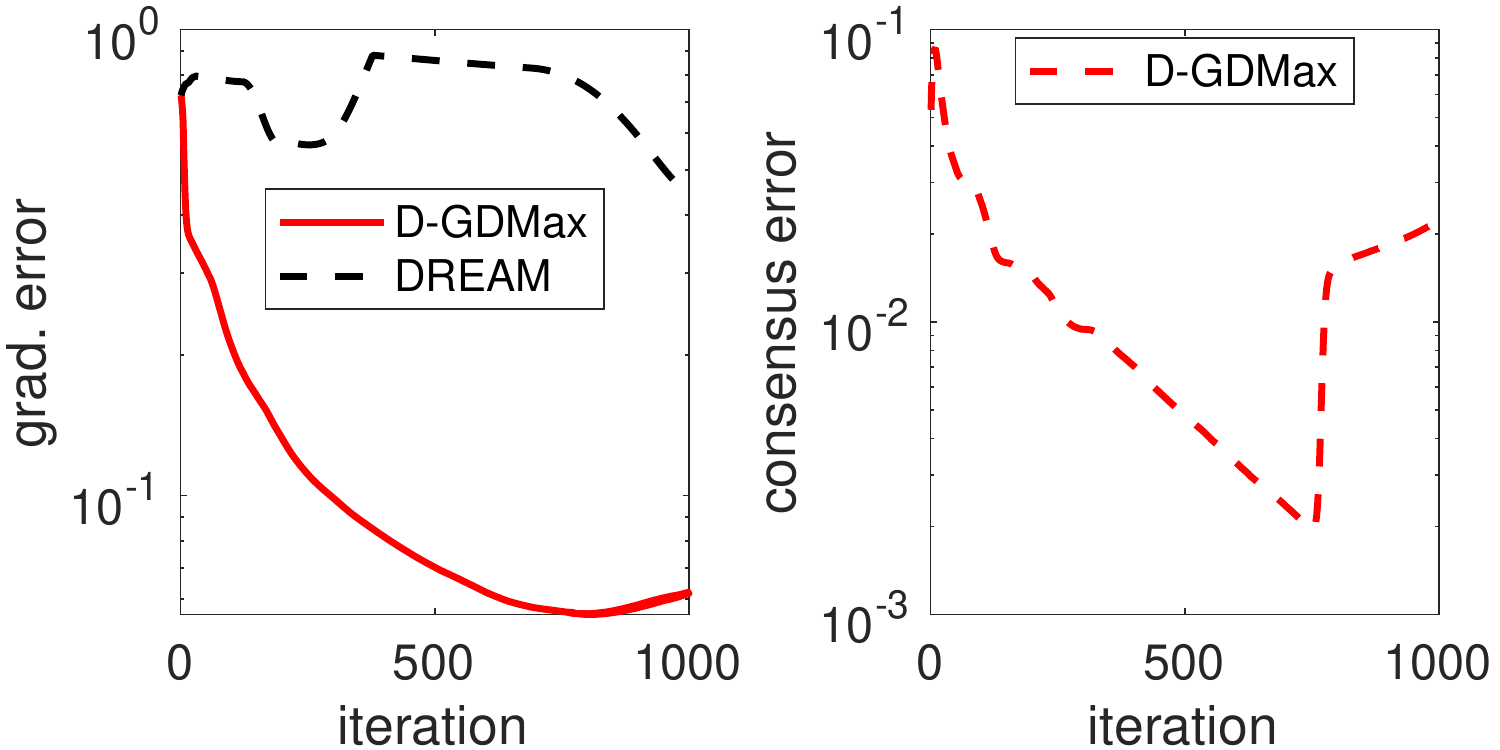} \\
\includegraphics[width = 0.3\textwidth]{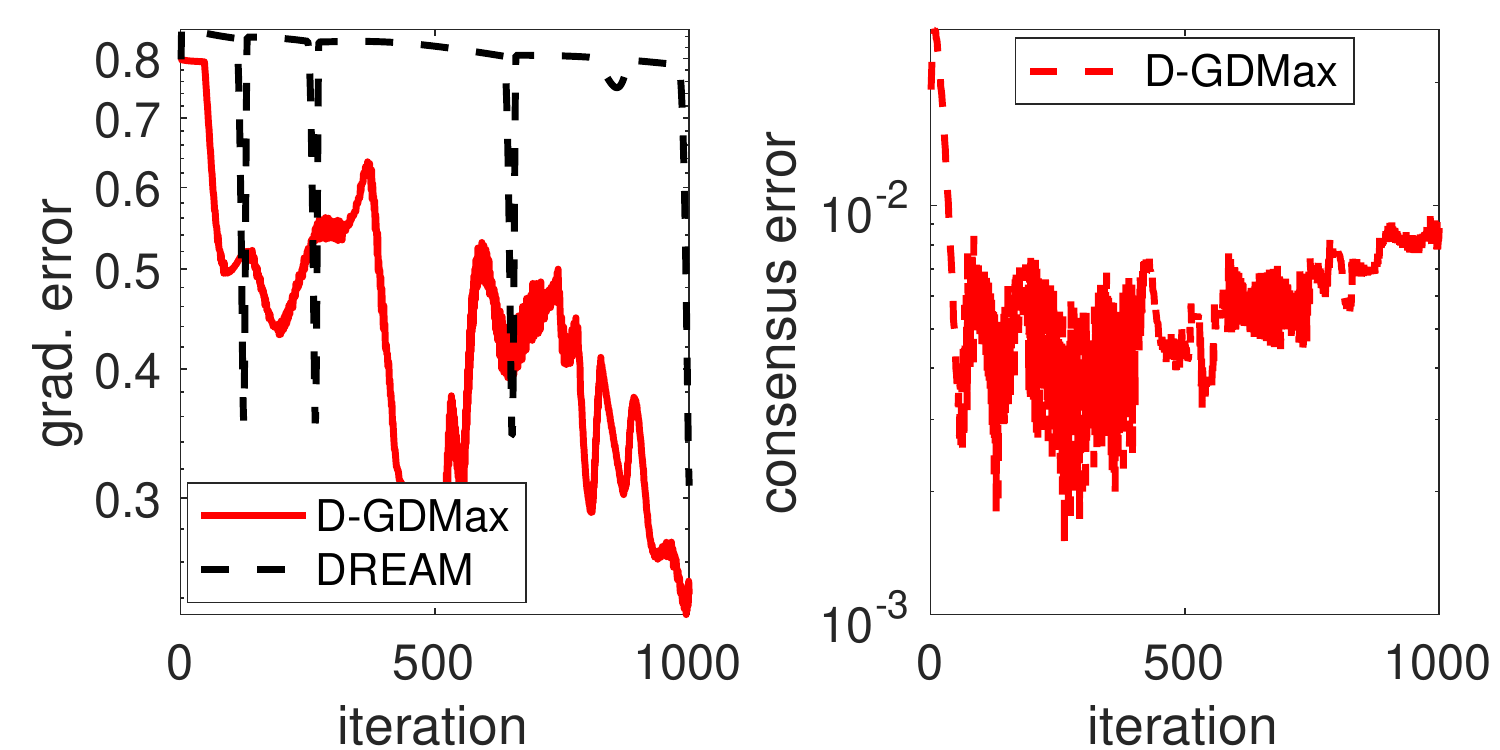} &
\includegraphics[width = 0.3\textwidth]{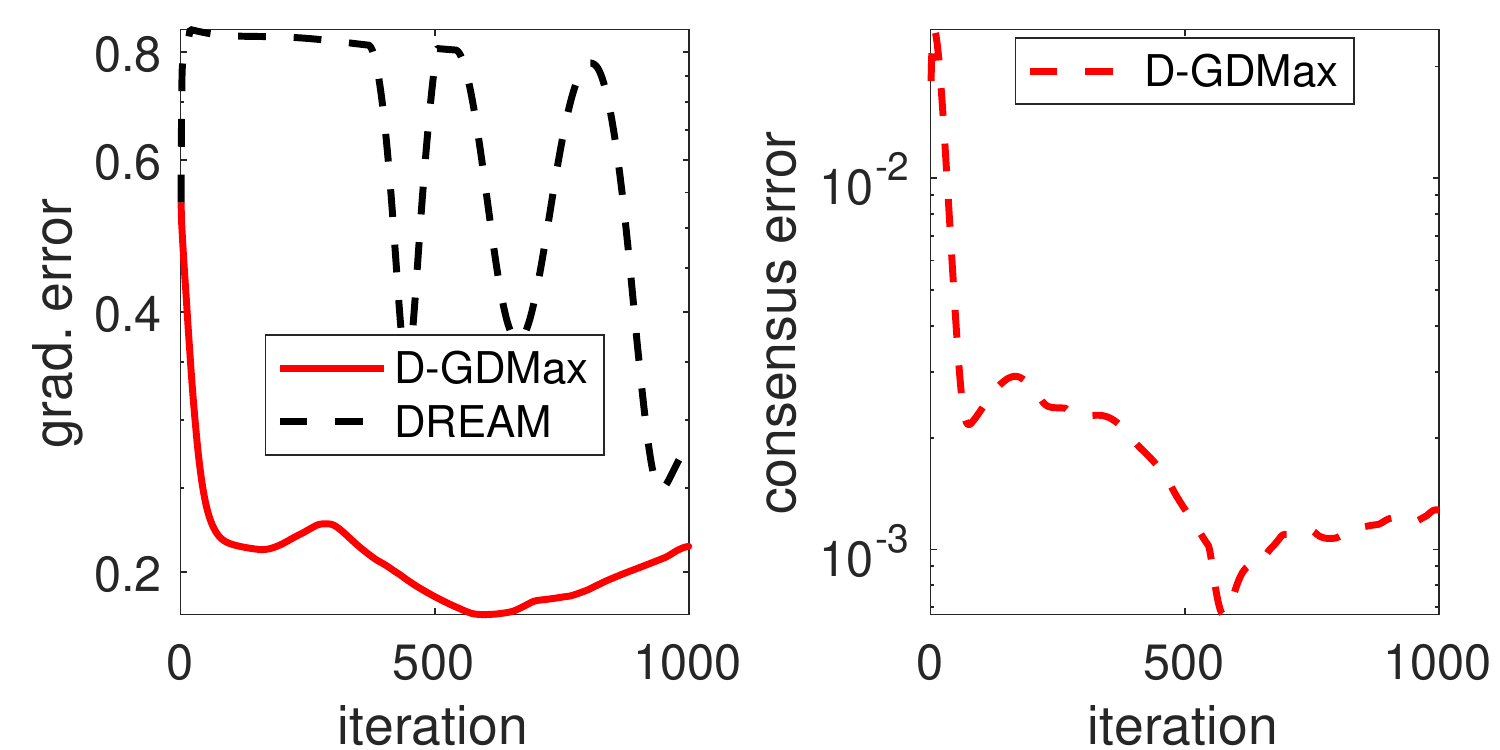} &
\includegraphics[width = 0.3\textwidth]{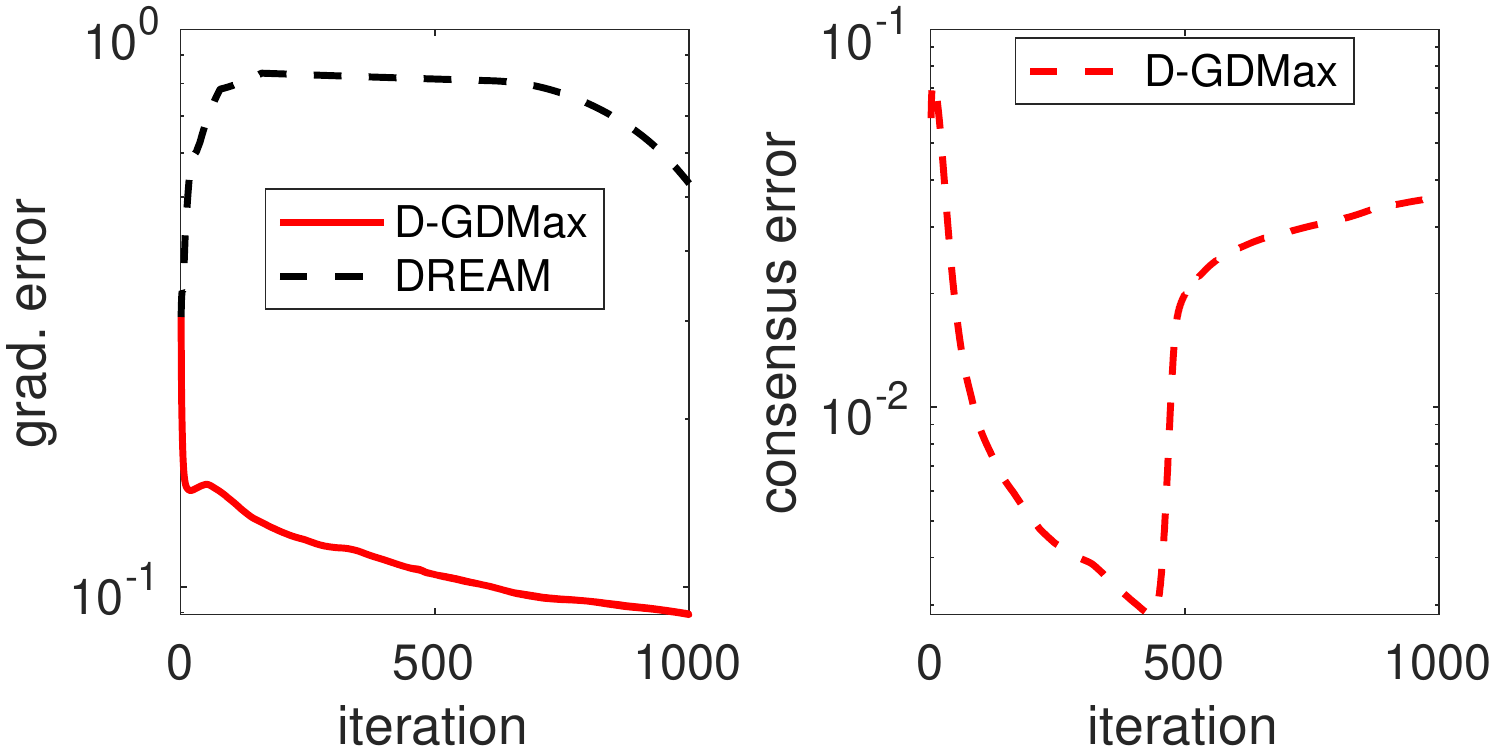} \\
\includegraphics[width = 0.3\textwidth]{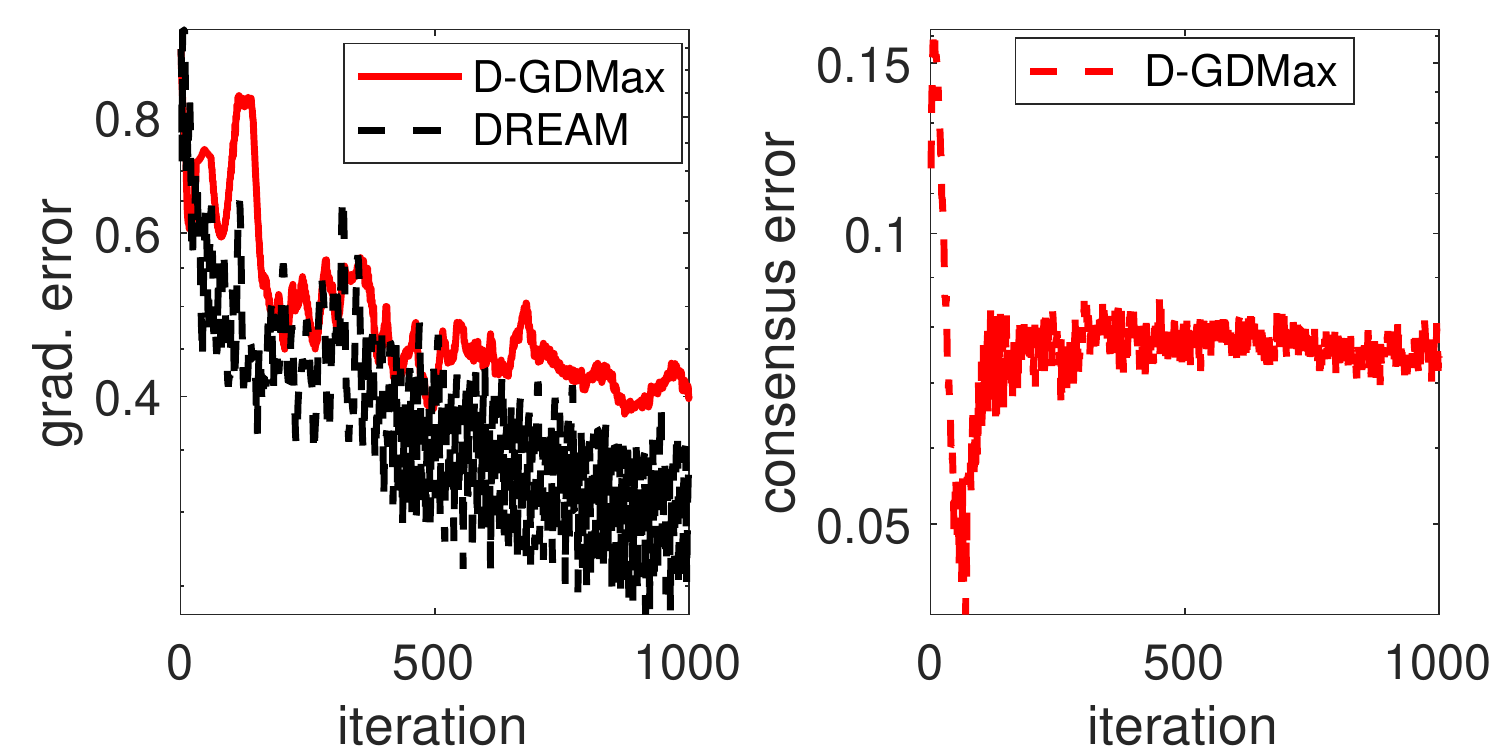} &
\includegraphics[width = 0.3\textwidth]{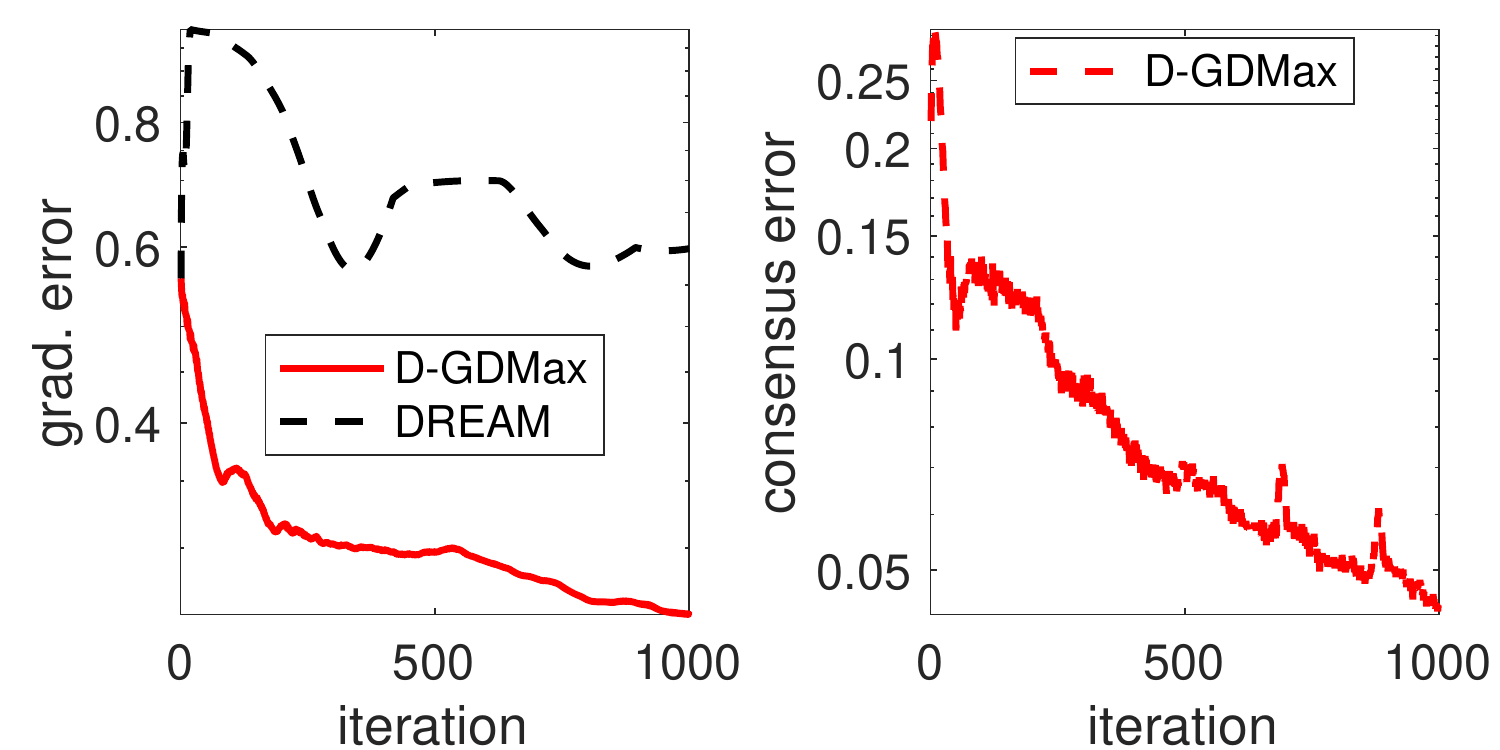} &
\includegraphics[width = 0.3\textwidth]{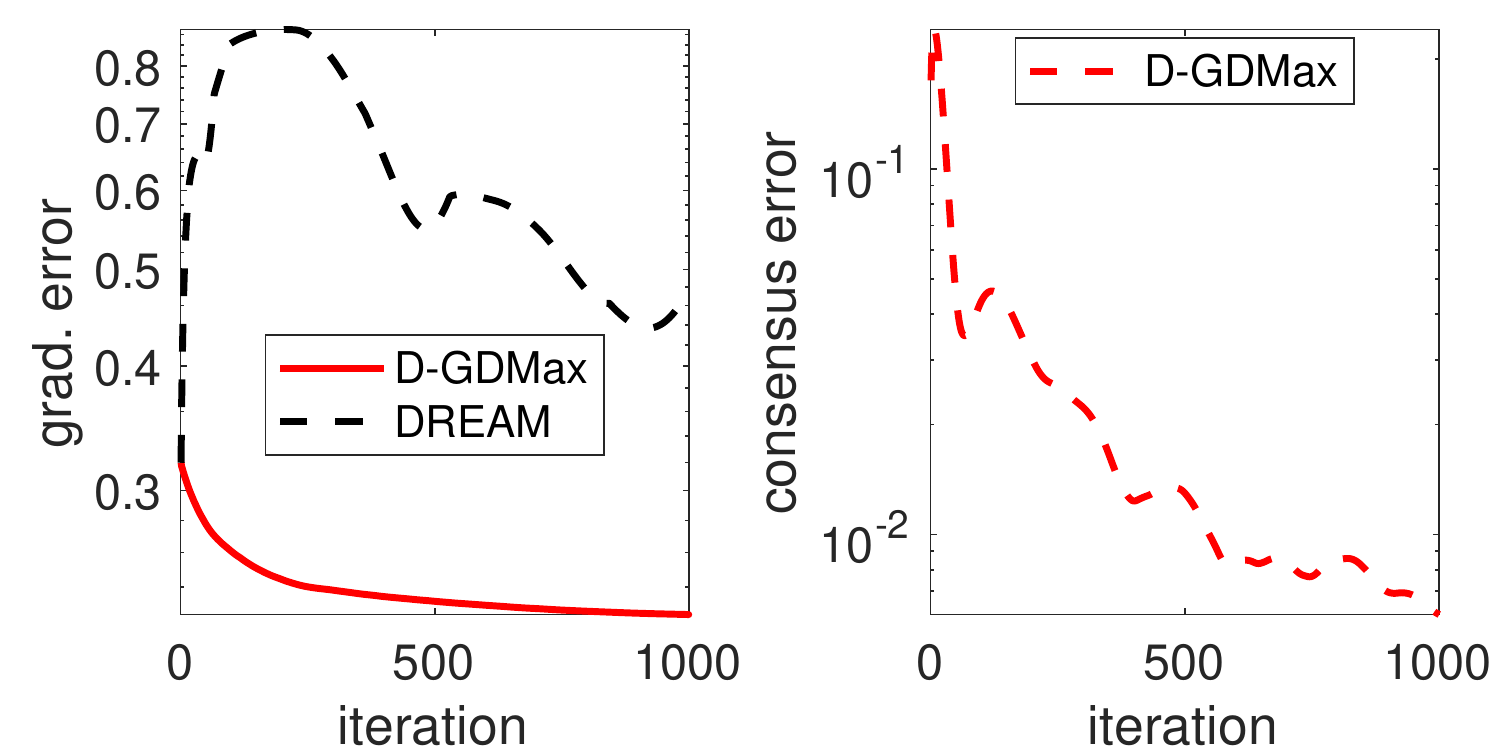} 
\end{tabular}
\resizebox{0.95\textwidth}{!}{
\begin{tabular}{cc|c|c|c|c|c}
\cline{2-7}
&\multicolumn{2}{c|}{\small $\beta_y = 0.01$} & \multicolumn{2}{|c|}{\small $\beta_y = 0.1$} & \multicolumn{2}{|c}{\small $\beta_y = 1$}\\\cline{2-7}
& D-GDMax & DREAM \cite{chen2022simple} & D-GDMax & DREAM \cite{chen2022simple} & D-GDMax & DREAM \cite{chen2022simple} \\\hline
\multicolumn{1}{c|}{a9a} & $\eta_x = 10^{-3},\eta_\lambda = 0.05$ & $\eta_x = 0.005, \eta_y = 10^{-3}$ & $\eta_x = 0.005,\eta_\lambda=10^{-3}$ & $\eta_x = 0.005, \eta_y = 10^{-3}$ & $\eta_x = 0.005,\eta_\lambda=10^{-3}$ & $\eta_x = 0.01, \eta_y = 10^{-3}$\\
\multicolumn{1}{c|}{gisette} & $\eta_x = 10^{-3},\eta_\lambda = 10^{-3}$ & $\eta_x = 0.01, \eta_y = 10^{-3}$ & $\eta_x = 10^{-3},\eta_\lambda=10^{-3}$ & $\eta_x = 0.01, \eta_y = 0.005$ & $\eta_x = 0.005,\eta_\lambda=10^{-3}$ & $\eta_x = 0.01, \eta_y = 0.5$\\
\multicolumn{1}{c|}{rcv1} & $\eta_x = 0.005,\eta_\lambda = 10^{-3}$ & $\eta_x = 0.1, \eta_y = 0.5$ & $\eta_x = 0.01,\eta_\lambda=10^{-3}$ & $\eta_x = 0.05, \eta_y = 0.5$ & $\eta_x = 0.01,\eta_\lambda=10^{-3}$ & $\eta_x = 0.1, \eta_y = 0.005$\\\hline
\end{tabular}
}
\end{center}
\caption{Primal gradient norm and consensus error at each iterate by our method (called D-GDMax) and DREAM in \cite{chen2022simple} for solving a decentralized formulation of the DRLR problem in \eqref{eq:drlr} on a9a, gisette, and rcv1 (from top to bottom) with $\beta_y$ varying from $\{0.1,1,10\}$ and $\alpha=10, \beta_x=10^{-3}$ fixed. The best stepsizes are listed in the table. A random graph with 20 nodes is used. Consensus error for DREAM stands at the order of $10^{-4}$ by 10 rounds of communication for each iteration.}\label{fig:dec-alg}
\end{figure}

\section{Conclusions}\label{sec:conclusion}
We have presented a decentralized algorithm for solving nonconvex strongly-concave (NCSC) composite minimax problems, by building a reformulation that decouples the nonsmoothness and consensus on the dual variable. To the best of our knowledge, this is the first attempt to handle decentralized nonsmooth NCSC problems. Global convergence (to stationarity) and iteration complexity results are both shown. In addition, on solving a distributionally robust logistic regression problem, the proposed algorithm gives superior numerical performance over a state-of-the-art decentralized method.

\appendix
\section{Proofs of a few lemmas} In this section, we give the proofs of a few lemmas that are used in our analysis.
\subsection{Proof of Lemma~\ref{lem:stepsize-choice}}
With the choice of $\eta_x$, it is easy to have
\begin{equation}\label{eq:stepsize-ineq1}
\textstyle \frac{2L^2(1+6\kappa^2)\eta_x^2}{(1-\rho)^3} \le \frac{2(1-\rho)}{25},\quad \frac{6mL^2\kappa^2\eta_x^2}{(1-\rho)^3} \le \frac{m(1-\rho)}{25},\quad 1-\rho - \frac{8\eta_x^2 L^2(1+ 6\kappa^2)}{(1-\rho)^3} \ge \frac{2(1-\rho)}{3}.
\end{equation}
By the definition of $\alpha_x$ and recalling $L_P = L\sqrt{1+4\kappa^2}$, it holds
$$\textstyle \alpha_x \ge \frac{1}{2m}\left(\frac{5L\sqrt{1+6\kappa^2}}{(1-\rho)^2} - 2L\sqrt{1+4\kappa^2} - L(\kappa+1)\right) \ge \frac{L(4\kappa+1)}{2m(1-\rho)^2}, $$
where the second inequality follows from $(1-\rho)^2 \le 1$ and simple calculations. In addition, we have from the formula of $\tilde\alpha_x$ that
\begin{align*}
\tilde\alpha_x \ge & \, 
\textstyle \frac{1}{2m\eta_x} - \frac{L^2(1+ 6\kappa^2)}{m L_P(1-\rho)^2} - \frac{2c(1-\rho)}{25} \nonumber 
\ge  \textstyle  \frac{1}{2m\eta_x} - \frac{L^2(1+ 6\kappa^2)}{m L_P(1-\rho)^2} - \frac{3}{25}\left(\frac{\rho^2}{2m\eta_x}+ \frac{L(4\kappa+1)}{2m} +\frac{4L^2(1+ 6\kappa^2)}{m L_P(1-\rho)^2} \right) \nonumber \\
\ge & \, \textstyle  \frac{11}{25m\eta_x} - \frac{37 L^2(1+ 6\kappa^2)}{25m L_P(1-\rho)^2} - \frac{3L(4\kappa+1)}{50m} \nonumber 
\ge  \textstyle  \frac{11 L\sqrt{1+6\kappa^2}}{5m (1-\rho)^2} - \frac{37 L(1+ 6\kappa^2)}{25m \sqrt{1+ 4\kappa^2}(1-\rho)^2} - \frac{3L(4\kappa+1)}{50m} \nonumber \\
\ge & \, \textstyle  \frac{L}{m (1-\rho)^2} \left(\left({\textstyle \frac{11}{5} - \frac{37\sqrt{3}}{25\sqrt{2}} }\right)\sqrt{1+6\kappa^2} - \frac{3(1+4\kappa)}{50}\right)\nonumber 
\ge  \textstyle  \frac{L(1+4\kappa)}{m (1-\rho)^2} \left(\frac{9}{50}-\frac{3}{50}\right) = \frac{3L(1+4\kappa)}{25m (1-\rho)^2} 
\end{align*}
where the first inequality uses the first result in \eqref{eq:stepsize-ineq1}; the second one follows from the definition of $c$ and the third result in \eqref{eq:stepsize-ineq1}; the third one holds by $\rho^2 < 1$; the fourth one uses the choice of $\eta_x$ and $L_P= L\sqrt{1+4\kappa^2}$; the fifth one is by the fact $\frac{\sqrt{1+6\kappa^2}}{\sqrt{1+4\kappa^2}} \le \frac{\sqrt{3}}{\sqrt{2}}$ and $(1-\rho)^2 \le 1$; the sixth one follows from  $\big(\frac{11}{5} - \frac{37\sqrt{3}}{25\sqrt{2}} \big)\sqrt{1+6\kappa^2} \ge \frac{9(1+4\kappa)}{50}$. 

For $\alpha_\lambda$, we have
\begin{align*}
\alpha_\lambda \ge &\, \textstyle  \frac{1}{\eta_\lambda}- \frac{L\kappa}{2} - \frac{L_P}{2} - \frac{3L^2\kappa^2}{L_P(1-\rho)^2} - \frac{cm(1-\rho)}{25}\\
\ge & \, \textstyle  \frac{1}{\eta_\lambda}- \frac{L\kappa}{2} - \frac{L_P}{2} - \frac{3L^2\kappa^2}{L_P(1-\rho)^2} - \frac{3m}{50}\left(\frac{\rho^2}{2m\eta_x}+ \frac{L(4\kappa+1)}{2m} +\frac{4L^2(1+ 6\kappa^2)}{m L_P(1-\rho)^2} \right)\\
= & \, \textstyle \frac{1}{\eta_\lambda} - \frac{3 L\rho^2 \sqrt{1+6\kappa^2}}{20(1-\rho)^2}- \frac{L}{2}\left(\frac{62\kappa}{50}+\frac{3}{50}\right) - \frac{L \sqrt{1+4\kappa^2}}{2} - \frac{3L}{\sqrt{1+4\kappa^2}(1-\rho)^2}\left(\frac{37\kappa^2}{25}+\frac{2}{25}\right) \\
\ge & \, \textstyle \frac{1}{\eta_\lambda} - \frac{L}{(1-\rho)^2}\left(\frac{3(\sqrt{6}\kappa + 1)}{20} + \frac{62\kappa}{100} + \frac{3}{100} + \kappa + \frac{1}{2}+ 3\kappa\right)
\ge  \textstyle \frac{1}{\eta_\lambda} - \frac{L(5\kappa+1)}{(1-\rho)^2} = \frac{L(4\kappa+1)}{(1-\rho)^2},
\end{align*}
where the first inequality uses the second result in \eqref{eq:stepsize-ineq1}; the second inequality follows from the definition of $c$ and the third result in \eqref{eq:stepsize-ineq1}; the third inequality is by $\sqrt{1+6\kappa^2} \le \sqrt{6}\kappa+1$, $\sqrt{1+4\kappa^2} \le 2\kappa+1$, $\frac{37\kappa^2+2}{25\sqrt{1+4\kappa^2}} \le \kappa$, and $(1-\rho)^2 \le 1$. 
Moreover, we have
\begin{align*}
\textstyle \frac{3\kappa^2}{2L_P(1-\rho)^2} + \frac{3c m\eta_x^2\kappa^2}{(1-\rho)^3} \le &\, \textstyle \frac{3\kappa^2}{2L_P(1-\rho)^2} + \frac{9 m\eta_x^2\kappa^2}{2(1-\rho)^4}\left(\frac{\rho^2}{2m\eta_x}+\frac{L(4\kappa+1)}{2m}+\frac{4L^2(1+ 6\kappa^2)}{m L_P(1-\rho)^2} \right)\\
= &\, \textstyle \frac{\kappa^2}{L\sqrt{1+4\kappa^2}(1-\rho)^2}\big(\frac{3}{2} + \frac{18}{25}\big) + \frac{9 \kappa^2 \rho^2}{20 L \sqrt{1+6\kappa^2}(1-\rho)^2} + \frac{9\kappa^2(1+4\kappa)}{100 L (1+6\kappa^2)}\\
\le &\, \textstyle \frac{\kappa^2}{L(1-\rho)^2}\left( \big(\frac{3}{2} + \frac{18}{25}\big)\frac{1}{2\kappa} + \frac{9}{20\sqrt{6}\kappa} + \frac{9}{140\kappa}\right) \le \frac{3\kappa}{2L(1-\rho)^2},
\end{align*}
where the first inequality follows from the formula of $c$ in \eqref{eq:def-little-c} and the third result in \eqref{eq:stepsize-ineq1}; the second inequality uses $\sqrt{1+4\kappa^2} \ge 2\kappa$, $\sqrt{1+6\kappa^2} \ge \sqrt{6}\kappa$, $\frac{1+4\kappa}{1+6\kappa^2} \le \frac{5}{7\kappa}$, and $1-\rho \le 1$.

Finally, it holds that
\begin{align*}
\textstyle \frac{1}{2m L_P(1-\rho)}   +\frac{c\eta_x^2}{(1-\rho)^2} \le &\, \textstyle \frac{1}{2m L_P(1-\rho)} + \frac{3\eta_x^2}{2(1-\rho)^3} \left(\frac{\rho^2}{2m\eta_x}+\frac{L(4\kappa+1)}{2m}+\frac{4L^2(1+ 6\kappa^2)}{m L_P(1-\rho)^2} \right)\\
= &\, \textstyle \frac{1}{2m L_P(1-\rho)} + \frac{3(1-\rho)}{50L^2(1+6\kappa^2)} \left(\frac{5 L \rho^2\sqrt{1+6\kappa^2}}{2m(1-\rho)^2}+\frac{L(4\kappa+1)}{2m}+\frac{4L^2(1+ 6\kappa^2)}{m L_P(1-\rho)^2} \right)\\
= &\, \textstyle \frac{37}{50m L_P(1-\rho)} + \frac{3\rho^2}{20mL\sqrt{1+6\kappa^2}(1-\rho)} + \frac{3(1-\rho)(1+4\kappa)}{100 m L (1+6\kappa^2)}\\
= & \, \textstyle \frac{37}{50m L\sqrt{1+4\kappa^2}(1-\rho)} + \frac{3\rho^2}{20mL\sqrt{1+6\kappa^2}(1-\rho)} + \frac{3(1-\rho)(1+4\kappa)}{100 m L (1+6\kappa^2)}\\
\le &\, \textstyle \frac{1}{m L (1-\rho)} \left(\frac{37}{100\kappa}+\frac{3}{20\sqrt{6}\kappa} + \frac{3}{140\kappa}\right) \le \frac{1}{2m L \kappa (1-\rho)},
\end{align*}
where the first inequality uses the definition of $c$ and the third result in \eqref{eq:stepsize-ineq1}; the second inequality follows from $\sqrt{1+4\kappa^2} \ge 2\kappa$, $\sqrt{1+6\kappa^2} \ge \sqrt{6}\kappa$, $\frac{1+4\kappa}{1+6\kappa^2} \le \frac{5}{7\kappa}$, and $1-\rho \le 1$. This completes the proof.

\subsection{Proof of Lemma~\ref{lem:bd-dit}}
By the concavity of $d_i^t$, it holds that $d_i^t(\vy_i^{t\star}) - d_i^t(\vy_i^{t-1}) \le \langle \vzeta_i^{t-1}, \vy_i^{t\star} - \vy_i^{t-1} \rangle $ for any $\vzeta_i^{t-1} \in \partial d_i^t(\vy_i^{t-1})$, namely,
\begin{equation}\label{eq:diff-d-i-t}
d_i^t(\vy_i^{t\star}) - d_i^t(\vy_i^{t-1}) \le \textstyle \left \langle \nabla_\vy f_i(\vx_i^t, \vy_i^{t-1}) - \vxi_i^{t-1} - \frac{L\sqrt{m}}{2}\tilde\vlam_i^t, \vy_i^{t\star} - \vy_i^{t-1} \right \rangle, \forall\, \vxi_i^{t-1} \in \partial h(\vy_i^{t-1}).
\end{equation}
From the condition $\dist\big(\vzero, \partial d_i^{t-1}(\vy_i^{t-1})\big) \le \delta_{t-1}$, it follows that there is a subgradient $\vxi_i^{t-1}\in \partial h(\vy_i^{t-1})$ such that $\|\nabla_\vy f_i(\vx_i^{t-1}, \vy_i^{t-1}) - \vxi_i^{t-1} - \frac{L\sqrt{m}}{2}\tilde\vlam_i^{t-1}\| \le \delta_{t-1}$, which together with \eqref{eq:diff-d-i-t} implies
\begin{equation*}
\begin{aligned}
&\,d_i^t(\vy_i^{t\star}) - d_i^t(\vy_i^{t-1}) \\
\le &\, \textstyle \left \langle \nabla_\vy f_i(\vx_i^t, \vy_i^{t-1}) - \nabla_\vy f_i(\vx_i^{t-1}, \vy_i^{t-1}) - \frac{L\sqrt{m}}{2}(\tilde\vlam_i^t - \tilde\vlam_i^{t-1}), \vy_i^{t\star} - \vy_i^{t-1} \right \rangle + \delta_{t-1}\|\vy_i^{t\star} - \vy_i^{t-1}\|.
\end{aligned}
\end{equation*}
Now using the $L$-smoothness of $f_i$ and the Young's inequality, we have from the inequality above that
\begin{equation}\label{eq:diff-d-i-t-2}
d_i^t(\vy_i^{t\star}) - d_i^t(\vy_i^{t-1}) \le \textstyle \frac{L}{2}\|\vx_i^t - \vx_i^{t-1}\|^2 + \frac{L\sqrt{m}}{4}\|\tilde\vlam_i^t - \tilde\vlam_i^{t-1}\|^2 + \left(\frac{L}{2} + \frac{L\sqrt{m}}{4} + \frac{1}{2}\right)\|\vy_i^{t\star} - \vy_i^{t-1}\|^2 + \frac{\delta_{t-1}^2}{2}. 
\end{equation}
Moreover, $\|\vy_i^{t\star} - \vy_i^{t-1}\|^2 \le 2\|\vy_i^{t\star} - \vy_i^{t-1\star}\|^2 + 2\|\vy_i^{t-1\star} - \vy_i^{t-1}\|^2 \le 2\|\vy_i^{t\star} - \vy_i^{t-1\star}\|^2 + \frac{2\delta_{t-1}^2}{\mu^2}$, where the second inequality follows from the proof of Prop.~\ref{prop:error-bd-vYt}. Hence, \eqref{eq:diff-d-i-t-2} indicates
\begin{align}\label{eq:diff-d-i-t-3}
&\,d_i^t(\vy_i^{t\star}) - d_i^t(\vy_i^{t-1}) \cr
\le &\, \textstyle \frac{L}{2}\|\vx_i^t - \vx_i^{t-1}\|^2 + \frac{L\sqrt{m}}{4}\|\tilde\vlam_i^t - \tilde\vlam_i^{t-1}\|^2 + \left(L + \frac{L\sqrt{m}}{2} + 1\right)\big(\|\vy_i^{t\star} - \vy_i^{t-1*}\|^2 + \frac{\delta_{t-1}^2}{\mu^2}\big) + \frac{\delta_{t-1}^2}{2} \cr
\le &\, \textstyle \frac{L}{2}\|\vX^t -\vX^{t-1}\|_F^2 + \frac{L\sqrt{m}}{4}\|\tilde\vLam^t - \tilde\vLam^{t-1}\|_F^2 + \left(L + \frac{L\sqrt{m}}{2} + 1\right)\big(\|\vY^{t\star} - \vY^{t-1*}\|_F^2 + \frac{\delta_{t-1}^2}{\mu^2}\big) + \frac{\delta_{t-1}^2}{2} . 
\end{align}
Furthermore, by \eqref{eq:lip-S2}, it follows that $\|\vY^{t\star} - \vY^{t-1*}\|_F^2 \le 2\kappa^2\|\vX^t - \vX^{t-1}\|_F^2 + 2m\kappa^2\|\vLam^t - \vLam^{t-1}\|_F^2$. Thus from \eqref{eq:diff-d-i-t-3} and $\|\tilde\vLam^t - \tilde\vLam^{t-1}\|_F^2 \le 4\|\vLam^t - \vLam^{t-1}\|_F^2$, we obtain
\begin{equation}\label{eq:diff-d-i-t-4}
\begin{aligned}
d_i^t(\vy_i^{t\star}) - d_i^t(\vy_i^{t-1}) \le &\,  \textstyle \left(\frac{L}{2} + \kappa^2 (2L + L\sqrt{m} + 2)\right)\|\vX^t -\vX^{t-1}\|_F^2 \\
& \, \hspace{-2cm}\textstyle + \big(L\sqrt{m} + m\kappa^2(2L + L\sqrt{m} + 2)\big)  \|\vLam^t - \vLam^{t-1}\|_F^2 +  \left(L + \frac{L\sqrt{m}}{2} + 1\right) \frac{\delta_{t-1}^2}{\mu^2} + \frac{\delta_{t-1}^2}{2}. 
\end{aligned}
\end{equation}
Finally, we use the second inequality in \eqref{eq:x-perp} and \eqref{eq:bd-sum-x-lam-2}--\eqref{eq:x-perp2-1} in the inequality above to obtain 
\begin{equation*}
\begin{aligned}
&\, d_i^t(\vy_i^{t\star}) - d_i^t(\vy_i^{t-1}) \\
\le &\,  \textstyle \left(\frac{L}{2} + \kappa^2 (2L + L\sqrt{m} + 2)\right)\left(\frac{50 mC_{0,t-1} (1-\rho)^2}{3L(1+4\kappa)} + \frac{12(1-\rho)}{25 L^2(1+6\kappa^2)} \big\|\vV_\perp^{0}\big\|_F^2 + \frac{212m C_{0,t-1} (1-\rho)^2}{25L(4\kappa+1)} + \frac{12m}{25 L^2} \Delta_{t-1}\right) \\
& \, \textstyle + \big(L\sqrt{m} + m\kappa^2(2L + L\sqrt{m} + 2)\big) \frac{C_{0,t-1}(1-\rho)^2}{L(4\kappa+1)} +  \left(L + \frac{L\sqrt{m}}{2} + 1\right) \frac{\delta_{t-1}^2}{\mu^2} + \frac{\delta_{t-1}^2}{2}.
\end{aligned}
\end{equation*}
Combining like terms in the inequality above and noting $\frac{50}{3} + \frac{212}{25}\le 21, \frac{12}{25}\le \frac{1}{2}$ gives the desired result.

\section{Failure of Minty's VI condition for NCSC problems} We give a simple example to show that the Minty's VI condition  assumed in \cite{liu2020decentralized} to handle nonconvex nonconcave problems can fail for NCSC minimax problems. Let $f(x,y) = -x^2y + \frac{1}{2}y^2$ for any $x\in \RR, y \in \RR$. Then $\min_{-1\le x \le 1} \max_y f(x,y)$ is a well-defined NCSC minimax problem and satisfies Assumption~\ref{assump-func} that we make. To see this, given any $x$, we have $y^* :=\argmax_y f(x, y) = x^2$. Hence, the primal problem is $\min_{-1\le x \le 1} p(x):= -\frac{1}{2}x^4$, which has two finite minimizers $x^*=\pm 1$. However, the Minty's condition, i.e., there is $(\bar x, \bar y)$ with $-1\le \bar x \le 1$ such that $\big\langle (f'_x(x,y), -f'_y(x,y)), (x,y) - (\bar x, \bar y) \big\rangle = -2xy(x-\bar x) + (x^2-y)(y-y^*) \ge 0, \forall \, -1\le x \le 1, \forall\, y\in \RR$, cannot hold, because for any $(\bar x, \bar y)$ and a fixed $x\in [-1,1]$, it is easy to have $-2xy(x-\bar x) + (x^2-y)(y-y^*) \to -\infty$ as $|y|\to \infty$. 

\bibliographystyle{abbrv}

\end{document}